\documentclass[a4paper,12pt]{article}
\usepackage{amsfonts,amsmath,amssymb,amsthm}
\usepackage{txfonts}
\usepackage[utf8]{inputenc}

\usepackage{amsmath}
\usepackage{amssymb}
\usepackage{mathtools}
\usepackage{bbm}

\RequirePackage[colorlinks,citecolor=blue,urlcolor=red, linkcolor=blue]{hyperref}
\usepackage[margin=0.8in]{geometry}

\usepackage[dvips]{graphics}
\usepackage{epsfig,rotating}
\usepackage{tabularx}
\usepackage{bm}
\usepackage{bbm}
\usepackage{array}
\newcolumntype{P}[1]{>{\raggedright\arraybackslash}p{#1}}

\usepackage{mdframed}
\usepackage{lipsum}
\usepackage{enumitem}

\usepackage[authoryear,square]{natbib}

\usepackage[nottoc,numbib]{tocbibind} 
\usepackage[toc,page]{appendix}

\newtheorem{thm}{Theorem}
\newtheorem{lemma}{Lemma}
\newtheorem{remark}{Remark}
\newtheorem{prop}{Proposition}

\DeclareMathOperator{\Var}{Var}

\newcommand{\beaa}{\begin{eqnarray*}}
\newcommand{\eeaa}{\end{eqnarray*}}
\newcommand{\bea}{\begin{eqnarray}}
\newcommand{\eea}{\end{eqnarray}}

\newcommand{\la}{\left\{}
\newcommand{\ra}{\right\}}
\newcommand{\lb}{\left(}
\newcommand{\rb}{\right)}
\newcommand{\mb}{\mathbb}
\newcommand{\ve}{\varepsilon}
\newcommand{\mc}{\mathcal}
\newcommand{\wT}{\widetilde{T}}
\newcommand{\wf}{\widetilde{f}}
\newcommand{\wPi}{\widetilde{\Pi}}

\begin{document}
\title{Exact detection threshold of the packing test}

\author{ Tuan Pham \\ \small{Department of Statistics and Data Science, University of Texas, Austin} \\  \small{\href{mailto:tuan.pham@utexas.edu}{tuan.pham@utexas.edu} }}
\date{}
\maketitle

\begin{abstract}

Using Poisson approximation techniques, we derive the detection threshold
of the packing test in \cite{Jiang13} when testing spherical uniformity under high-dimensional
Fisher--von Mises--Langevin (FvML) and Watson alternatives.
Our result rigorously confirms the empirical observation that the packing test is strictly suboptimal for testing uniformity in these two popular
models.  In the high-dimensional FvML model, its detection threshold is precisely
\(\kappa=\Theta\lb p^{3/4}/(\log n)^{1/4}\rb\).  In the high-dimensional Watson model,  its detection threshold is
\(p-2\kappa=\Theta(\sqrt{p\log n})\), or equivalently
\(\kappa=p/2-\Theta(\sqrt{p\log n})\). We show that the limiting scalings of the largest squared inner product undergo a discontinuous phase transition in the Watson model, whereas no analogous phenomenon occurs in the FvML model.

\end{abstract}




\tableofcontents

 \section{Introduction}

Assessing uniformity is a classical goodness-of-fit problem in directional statistics. Arguably one of the most fundamental questions in the field, uniformity testing has attracted considerable attention over the past few decades. Let us briefly describe the setting. Let $\bm{X}_1,\bm{X}_2,\ldots,\bm{X}_n$ be observations with $\bm{X}_i \in \mathbb{S}^{p-1}$ for all $i=1,2,\ldots,n$. We assume that the variables $\bm{X}_i$ are drawn independently from an unknown distribution $\mu$ supported on the hypersphere $\mathbb{S}^{p-1}$. Denote the uniform distribution on $\mathbb{S}^{p-1}$ by $\mathrm{Unif}(\mathbb{S}^{p-1})$. The uniformity testing problem can then be formulated as
\begin{align} \label{uniform-test}
    H_0:\ \mu=\mathrm{Unif}(\mathbb{S}^{p-1})
    \qquad \text{versus} \qquad
    H_1:\ \mu \neq \mathrm{Unif}(\mathbb{S}^{p-1}).
\end{align}
In fixed dimensions, particularly for $p=2$ and $p=3$, uniformity testing plays an important role in applications such as geology, paleomagnetism, and cosmology. For detailed background and examples, we refer to the monographs \cite{Fisher,Ley-Verdebout,M-Jupp}. Recent tests that apply in arbitrary but fixed dimensions include \cite{garcia2023projection,fernandez2023new,garcia2021cramer} (see also the references therein). We also refer to the survey papers \cite{survey-uni,pewsey2021recent} for an overview of recent progress and additional testing procedures.

More recently, the testing problem \eqref{uniform-test} has found new applications in high-dimensional statistics and machine learning. The situation in high dimensions is more interesting than its fixed-dimensional counterpart because the sphere suffers from the curse of dimensionality in high dimensions, and a test can be suboptimal in terms of sample size, dimension, or both. Let us briefly highlight some applications below.

\textit{\underline{Testing spherical symmetry}.} It is well known that a random vector $\bm{X} \in \mathbb{R}^p$ is spherically symmetric if and only if $\|\bm{X}\|$ is independent of $\bm{X}/\|\bm{X}\|$, and $\bm{X}/\|\bm{X}\|$ is uniformly distributed on $\mathbb{S}^{p-1}$. Therefore, spherical symmetry can be assessed by combining an independence test with a uniformity test. In practice, departures from spherical symmetry often manifest as violations of uniformity for the projected directions (see \cite{Cutting-P-V} for details).

\textit{\underline{Contrastive learning}.} Contrastive learning is a form of representation learning in which one seeks to learn a map $\phi$, given data $\{\bm{X}_1,\ldots,\bm{X}_n\}$, such that $\phi$ extracts useful features for a downstream task, which may range from estimation and classification to clustering. 

More precisely, for each $i \in [1,n]$, one generates two augmented views $\lb \bm X^{(1)}_i, \bm X^{(2)}_i \rb$ of $\bm X_i$ using random transformations (one may think of $\bm X_i$ as an image and of the random transformations as random cropping, blurring, or color distortion). An ideal $\phi$ should align the paired views and separate the unpaired views. One way to separate the unpaired views is to require that the normalized transformed data
\[
\bm z^{(j)}_i := \frac{ \phi \lb \bm x^{(j)}_i \rb }{\left\| \phi \lb \bm x^{(j)}_i \rb \, \right\|}; \qquad j=1,2
\]
satisfy 
\[
\bm z^{(1)}_1,\dots,\bm z_n^{(1)} \stackrel{i.i.d.}{\sim} \mbox{Unif} \lb \mb S^{p-1} \rb \qquad \text{and} \qquad \bm z^{(2)}_1,\dots,\bm z_n^{(2)} \stackrel{i.i.d.}{\sim} \mbox{Unif} \lb \mb S^{p-1} \rb.
\]
A high-dimensional version of Ajne's test \cite{Ajne} for uniformity was used in \cite{furst2022cloob} to assess the quality of the learned map $\phi$; see also \cite{wang2020understanding,jing2021understanding} and the references therein.

\textit{\underline{Regularization of neural networks}.} 
In overparameterized neural networks, regularization is crucial for preventing overfitting and improving generalization. 
Given $n$ neuron weight vectors $\bm{w}_1,\bm{w}_2,\ldots,\bm{w}_n \in \mathbb{R}^p$, a regularization term typically takes the form
\begin{align*} 
    \underbrace{\lambda_{1} \sum_{i=1}^{n} h(\bm{w}_i)}_{\text{magnitude regularization}}
    + \underbrace{\lambda_{2}\, g(\bm{w}_1,\bm{w}_2,\ldots,\bm{w}_n)}_{\text{structure regularization}} .
\end{align*}
The penalty function $h$ typically regularizes the magnitudes of the neurons, for example through $\ell_1$ or $\ell_2$ norms, analogous to the Lasso or ridge penalties in high-dimensional statistics. 
In contrast, the term $g$ aims to promote specific structures among the neurons. 
A line of empirical work \cite{lin2020regularizing,liu2018learning,xie2017diverse} has demonstrated that when $g$ is designed to encourage \emph{hyperspherical uniformity} among the neurons, it can lead to improved generalization. 
Uniformly distributed neurons have also been shown to help avoid spurious local minima during training \cite{xie2017diverse}. 
Connections between this perspective and testing hyperspherical uniformity have been discussed in \cite{liu2021learning}.

Despite the abundance of existing tests in fixed- and low-dimensional settings, uniformity testing in high dimensions remains comparatively underexplored. 
To the best of the author's knowledge, only a few high-dimensional tests have been investigated in the literature. 
We provide a brief overview of these tests below.

\begin{enumerate}
    \item \textit{Rayleigh test} \cite{Cutting-P-V,Ley-P}. 
    This test can be formulated as a U-statistic of the data points with the inner-product kernel:
    \begin{align}
        R_n 
        &:= \frac{\sqrt{2p}}{n} 
        \sum_{1 \leq i<j \leq n} \bm{X}^{\top}_i \bm{X}_j. 
        \label{Rayleigh}
    \end{align}

    \item \textit{Bingham test} \cite{Cutting-P-V2,Zou14,Ley-P}. 
    This test is also based on a U-statistic of the data points, but with a quadratic inner-product kernel:
    \begin{align}
        B_n 
        &:= \frac{p}{n} 
        \sum_{1 \leq i<j \leq n} 
        \left[ \bigl( \bm{X}^{\top}_i \bm{X}_j \bigr)^2 - \frac{1}{p} \right].
        \label{Bingham}
    \end{align}

    \item \textit{Packing test} \cite{Jiang13}. 
    This test is based on the smallest pairwise angle, or equivalently on the largest absolute inner product:
    \begin{align}
        P_n 
        &:= p \cdot \max_{1 \leq i<j \leq n} 
        \bigl( \bm{X}^{\top}_i \bm{X}_j \bigr)^2
        - 4 \log n + \log \log n .
        \label{packing}
    \end{align}


    \item \textit{High-dimensional Sobolev tests} \cite{ebner2025high}. 
    This is a family of Sobolev-based tests that generalizes the classical tests proposed in \cite{Gine}. 
\end{enumerate}

Among the tests above, the packing test \eqref{packing} remains comparatively less understood in terms of sharp detection rates and consistency against standard local parametric alternatives. 
Empirically, it does not appear to be optimal against common classes of alternatives, although it is provably consistent against certain types of non-local alternatives \cite{jiang2025asymptotic}. 
For the Rayleigh and Bingham tests, several detection-rate and optimality results have already been established in the literature. The Rayleigh test \eqref{Rayleigh} is known to be optimal against the FvML distributions (see \eqref{class fvml} below for a precise definition) and blind to the Watson distributions (see \eqref{class watson} below for a precise definition) \cite{Cutting-P-V}. 
The Bingham test exhibits the opposite behavior: it is optimal against the Watson distributions but suboptimal against the FvML distributions \cite{Cutting-P-V2}.

The goal of the present article is to rigorously derive and justify the detection rates of the packing test under the classes of FvML and Watson distributions using Poisson approximation techniques; see, for example, \cite{arratia1989two} and the references therein. Our result shows that the detection threshold for the packing test is $\Theta \lb p^{3/4} / \lb \log n \rb^{1/4} \rb$ under the FvML model and $p/2 - \Theta \lb \sqrt{p \log n} \rb$ under the Watson model. 
The rest of the paper is organized as follows. The models and known results are presented in Section \ref{sec:model}. The main results can be found in Sections \ref{sec:FVML} and \ref{sec:Watson}, respectively. Section \ref{sec:conclusion} contains some remarks and discussion. The proofs are given in Sections \ref{sec:proof-fvml}, \ref{sec:proof-watson-1}, \ref{sec:proof-watson-2}, and \ref{sec:proof-watson-3}. Technical results are provided in the appendix. 

\section{Model and known results} \label{sec:model}

Recall that the data are $\bm{X}_1,\dots,\bm{X}_n$, which are i.i.d. points on the hypersphere $\mb{S}^{p-1}$ with distribution $\mu$. 
Under the null, that is, $\mu \equiv \mbox{Uni} \lb \mb{S}^{p-1} \rb$, the asymptotic distribution of $P_n$ is well known to be the Gumbel distribution \cite{Jiang13}. Under nonuniform alternatives, however, the limiting distribution is unknown. Since the packing test is of extreme-value type, its analysis requires techniques different from those used for U-statistic-type tests. 

In what follows, we describe two common parametric models in directional statistics: the Fisher--von Mises--Langevin (FvML) distributions and the Watson distributions. These models consist of densities that have the form
\begin{align} \label{class fvml}
x \;\longmapsto\; c^{\rm FvML}_{p,\kappa} \, \exp\!\left(\kappa\, \bm{x}^\top \bm{\mu}\right),
\qquad x \in \mathbb{S}^{p-1},
\end{align}
and
\begin{align} \label{class watson}
x \;\longmapsto\; c^{\rm Wat}_{p,\kappa} \, \exp\!\left(\kappa\, (\bm{x}^\top \bm{\mu})^{2}\right),
\qquad x \in \mathbb{S}^{p-1},
\end{align}
respectively. Here $\kappa \in \mb R^+$ is the concentration parameter, $\bm \mu$ is the unknown location (nuisance) parameter, and $c^{\rm FvML}_{p,\kappa}, c^{\rm Wat}_{p,\kappa} $ are normalizing constants. 

The optimal detection rates in these two models have been studied in \cite{Cutting-P-V, Cutting-P-V2, jiang2025detecting}. It was shown in \cite{Cutting-P-V} that the optimal detection rate for the FvML model is $\kappa \sim p^{3/4}/\sqrt{n}$ and is achieved by the Rayleigh test \eqref{Rayleigh}. Moreover, the Rayleigh test is asymptotically optimal among all rotationally invariant tests. The optimal detection rate for the Watson model \eqref{class watson} is more interesting and complicated: it exhibits a phase transition between the low- and high-dimensional regimes. It was shown in \cite{Cutting-P-V2} that when $p \ll n$, the optimal detection rate is $\kappa \sim p^{3/2}/\sqrt{n}$, whereas it was shown in \cite{jiang2025detecting} that when $p \gg n$, the optimal detection rate is $p/2 -\kappa \sim \sqrt{pn}$.

Using Le Cam's third lemma, the authors of \cite{jiang2025asymptotic} show that the packing test \eqref{packing} is asymptotically blind at the optimal detection threshold in the FvML model \eqref{class fvml}: if one sets $\kappa = \tau p^{3/4}/\sqrt{n}$, then regardless of how large $\tau$ is, the packing test has asymptotic power equal to its size. Extending this result to the Watson model \eqref{class watson} is highly nontrivial because a high-dimensional LAN result is unavailable in this model. 

Using Le Cam's third lemma, the authors of \cite{ebner2025high} further establish a somewhat surprising result: within a general class of Sobolev-based tests that includes the Rayleigh and Bingham tests as special cases, the Rayleigh test is the only one that achieves the optimal detection rate under the FvML model. 
In particular, any linear combination of $R_n$ and $B_n$ in \eqref{Rayleigh} and \eqref{Bingham}, respectively, that is not proportional to $R_n$ fails to achieve the optimal detection rate.

It is also clear that the packing test \eqref{packing} should not be blind in these two models: when the concentration parameter $\kappa$ is large enough, the maximum of squared inner products is approximately one with overwhelming probability, and thus $P_n \to \infty$ in probability. This motivates the problem of precisely identifying the detection thresholds of $P_n$ in these two models. We do so by deriving the non-null asymptotic distribution of $M_n$ in \eqref{S- M} below under the two models \eqref{class fvml} and \eqref{class watson}.

\textbf{Notation}. Throughout the paper, we define 
\begin{align}
 S_{ij}:&=\bm X_i^\top \bm X_j,\qquad
 M_n:=\max_{1\le i<j\le n}|S_{ij}| ; \label{S- M} \\
  u=u_n(x):&=\sqrt{4\log n-\log\log n+x},\qquad 
 t=t_n(x):=\frac{u}{\sqrt p} ; \label{u-t} \\
  I_{ij}:&=\mathbf{1}_{\la |S_{ij}|>t \ra},\qquad
 W:=\sum_{1\le i<j\le n}I_{ij}. \label{I-W}
\end{align}
Let \(Z_d\) denote the inner product
of two independent uniform points on \(\mathbb S^{d-1}\).  Thus
\begin{equation}\label{eq:spherical-density}
 f_d(z)=c_d(1-z^2)^{(d-3)/2},\qquad -1<z<1,
 \qquad
 c_d=\frac{\Gamma(d/2)}{\sqrt\pi\Gamma((d-1)/2)}.
\end{equation}
Define
\begin{align} \label{Psi}
 \Psi_d(y):=\mb P \left(Z_d>\frac{y}{\sqrt d}\right)
\end{align}
with the natural convention \(\Psi_d(y)=1\) if \(y/\sqrt d\le -1\) and
\(\Psi_d(y)=0\) if \(y/\sqrt d\ge1\).

We will use $\mb P_{\kappa, \bm \mu}$ to denote either the FvML distribution or the Watson distribution with parameters $\lb \kappa, \bm \mu  \rb$. The intended model will be clear from context and will be specified whenever ambiguity may arise. Throughout the paper, we use $L=L_n$ to denote either $\log n$ or a generic positive sequence diverging to infinity, depending on the context. The notation $A_n \lesssim B_n$ will be used to indicate that $A_n \leq C B_n$ for some universal constant $C>0$.  We will also write $A_n \lesssim_{a,b} B_n$ to indicate that the constant $C$ may depend on $a$ and $b$,  but remains independent of $n$. 
Similarly, the notation $a_n=O_{\rho,\ve} \lb b_n \rb$ means $a_n \leq C b_n$ for some constant $C$ independent of $n$, but may depend on $\rho$ and $\ve$.



\section{Detection rate and limiting distribution under the FvML model} \label{sec:FVML}



Recall the FvML model in \eqref{class fvml}. Our first main result gives a non-null asymptotic distribution of $P_n$ under this model.

\begin{thm} \label{thm: FvML}
    Recall $P_n$ in \eqref{packing}. Suppose $p/ (\log n)^2 \to \infty$ and 
    $$\kappa_n =  \tau \cdot \frac{p^{3/4}}{ \lb \log n  \rb^{1/4}}$$ 
    for some $\tau \geq 0$. Then, under the FvML model with concentration parameter $\kappa_n$,
    \[
    P_n - 2\log \cosh \lb 2\tau^2 \rb \stackrel{d}{\to} G
    \]
    where $G$ is the Gumbel law with distribution function 
    \[
     \mathbb P(G\le y)=\exp\{-(8\pi)^{-1/2}e^{-y/2}\}.
    \]
\end{thm}

Theorem~\ref{thm: FvML} states that the asymptotic distribution of $P_n$ under the class of FvML distributions is a shifted Gumbel distribution, with a deterministic shift. Note that we recover the null asymptotic distribution by setting $\tau \equiv 0$.
As a direct consequence of Theorem~\ref{thm: FvML}, the packing test $P_n$ has detection rate of order 
\[
\frac{p^{3/4}}{(\log n)^{1/4}}.
\]
Comparing this rate with the known minimax rate $p^{3/4}/\sqrt{n}$ from \cite{Cutting-P-V}, we observe an interesting phenomenon: the packing test achieves the optimal scaling in terms of the dimension dependence, namely $p^{3/4}$, but is suboptimal in its sample-size dependence, with $(\log n)^{-1/4}$ replacing the optimal scale $n^{-1/2}$. This helps explain its low empirical power at the minimax scale. 

As a consequence of Theorem~\ref{thm: FvML}, we can calculate the power function of $P_n$ explicitly. 
Let $q_\alpha$ denote the asymptotic level-$\alpha$ critical value under the null, defined by
\[
q_\alpha
= -2 \log\!\left[ \sqrt{8\pi}\, \lb -\log(1-\alpha) \rb \right].
\]
Suppose that the FvML concentration parameter satisfies
\[
\frac{\kappa_n(\log n)^{1/4}}{p^{3/4}} \to \tau \in [0,\infty).
\]
Theorem~\ref{thm: FvML} then gives the asymptotic power of the level-$\alpha$ packing test to be
\begin{align*}
\beta_{\rm FvML}(\kappa_n, \bm \mu_n)
&:= \lim_{n\to\infty} \mathbb{P}_{\kappa_n, \bm \mu_n}\!\left(P_n > q_\alpha\right) \\
&= 1 -
\exp\!\left\{
-\frac{1}{\sqrt{8\pi}}
\exp\!\left(-\frac{q_\alpha - 2\log\cosh(2\tau^2)}{2}\right)
\right\} = 1 - (1-\alpha)^{\cosh(2\tau^2)}.
\end{align*}
The proof of Theorem \ref{thm: FvML} is based on Poisson approximation techniques \cite{arratia1989two} (see also \cite{feng2026principal,jammalamadaka2015asymptotic} and the references therein). 
The deviation from the Gumbel law in Theorem \ref{thm: FvML} comes from the mean shift of the inner product, which makes the FvML model easier to analyze than the Watson model, which is symmetric and deviates from the null through its variance.

\section{Detection rate and limiting distribution under the Watson model} \label{sec:Watson}


Recall the Watson model in \eqref{class watson}. We will assume that $\kappa_n < p/2$ and
\begin{align*}
    \rho_n:=\frac{\Delta_n}{\sqrt{p \log n}} \to \rho \in (0,\infty), \qquad \Delta_n:= \frac{p-2\kappa_n}{2}.
\end{align*}
Recall $M_n$ from \eqref{S- M}. We will investigate the asymptotic distribution of $M_n$ in three regimes: $\rho \in (1,\infty)$, $\rho \in (0,1)$, and $\rho=1$. The regime $\rho \in (1,\infty)$ will be called {\it local} because the packing test  has nontrivial asymptotic power (strictly less than one) in this regime. The regime $\rho \in (0,1)$ will be called {\it non-local} since the packing test always has asymptotic power one in this regime. The regime $\rho =1$ will be called {\it critical} since its second-order approximation differs from those in the other regimes; see Section \ref{sec:watson-critical} below for more details.  

It turns out that although $M_n$ is asymptotically Gumbel after appropriate normalization in all three regimes, the corresponding asymptotic scalings undergo a subtle and interesting phase transition in the second-order approximation of the asymptotic mean, with a surprisingly discontinuous transition at $\rho=1$.

\subsection{The local regime $\rho \in (1,\infty)$}
Our result is as follows.
\begin{thm} \label{thm: Watson}
  Suppose $p/ \lb \log n \rb^2 \to \infty$ and $\rho_n \to \rho >1$.  For every fixed $x \in \mb{R}$, we have  
  \begin{align*}
      \mb P_{\kappa_n, \bm \mu_n} \lb P_n \leq x \rb {\to}            \exp\left\{-(8\pi)^{-1/2} \cdot \frac{1}{\sqrt{1-\rho^{-2}}} \cdot e^{-x/2}\right\}, \qquad &\text{if $\rho \in (1,\infty)$}.
  \end{align*}
  Here $P_n$ is defined as in \eqref{packing}.
\end{thm}

Theorem \ref{thm: Watson} states that when $\rho> 1$, $P_n$ converges to a shifted Gumbel law
\[
 P_n+\log(1-\rho^{-2}) \stackrel{d}{\to}  G,
 \qquad
 \mathbb P(G\le y)=\exp\la -(8\pi)^{-1/2}e^{-y/2} \ra
\]
under the Watson alternatives, with deterministic shift $\log \lb 1 - \rho^{-2} \rb$. Thus the detection rate achieved by the packing test in this case is of order 
\[
 \kappa_n
 =\frac p2-\rho\sqrt{p\log n}\cdot \lb 1+o(1)\rb.
\]
In other words, the detection rate of the packing test in the Watson model is 
\[
\kappa_{\rm packing} = p/2 - \Theta \lb \sqrt{p \log n} \rb.
\]
Let us compare this detection rate to the minimax rate in the Watson model \cite{jiang2025detecting,Cutting-P-V2}
\[
\kappa_{\rm minimax} = \begin{cases}
    \Theta \lb \frac{p^{3/2}}{\sqrt{n}} \rb, \qquad &\text{if $p \ll n$} \\
    \frac{p}{2} - \Theta \lb \sqrt{pn} \rb,   &\text{if $p \gg n$}.
\end{cases}
\]
Comparing $\kappa_{\rm packing}$ and $\kappa_{\rm minimax}$, we see that the level-$\alpha$ packing test is suboptimal in both dimension and sample size in the low-dimensional regime $p=o(n)$. Interestingly, in the high-dimensional regime, it correctly captures the dimensional dependence but is suboptimal in terms of sample-size scaling.

As a direct consequence of Theorem \ref{thm: Watson}, we can calculate the asymptotic power function of the packing test against the Watson alternatives as follows. Let \(q_\alpha\) denote the asymptotic level-\(\alpha\) critical value under
the null, defined by
\[
 q_\alpha
 =
 -2\log\left[
 -\sqrt{8\pi}\log(1-\alpha)
 \right].
\]
Suppose that the Watson concentration parameter satisfies
\[
 \frac{\Delta_n}{\sqrt{p\log n}}\to \rho\in[0,\infty) .
\]
Thus the asymptotic power function is
\begin{align} \label{power near 1}
 \beta_{\rm Wat}(\kappa_n,\bm\mu_n)
 =
 \begin{cases}
 1-(1-\alpha)^{(1-\rho^{-2})^{-1/2}},
 & \rho>1,\\[4pt]
 1,
 & 0\le \rho\le1.
 \end{cases}
\end{align}
The assertion of asymptotic power one in the regime $0 \leq \rho \leq 1$ is a consequence of Theorems \ref{thm:watson-rho<1} and \ref{thm:Watson-critical} below. Note that although the case $\rho=1$ does not exactly follow from Theorem \ref{thm:Watson-critical}, its proof requires only minor modifications (and is even simpler because we only need to prove consistency). The needed modifications can be found in Appendix \ref{appendix: power near 1}.

From the asymptotic power formula above, as \(\rho\downarrow1\), the exponent
\((1-\rho^{-2})^{-1/2}\) diverges and the power tends to one. On the other
hand, as \(\rho\to\infty\), the exponent tends to one and the limiting power
approaches the size \(\alpha\). This is consistent with the interpretation
that the packing test has nontrivial asymptotic power only when
\[
 \Delta_n=\frac p2-\kappa_n
 =
 \Theta\left(\sqrt{p\log n}\right),
\]
or equivalently
\[
 \kappa_n
 =
 \frac p2- \Theta\left(\sqrt{p\log n}\right).
\]

\subsection{The non-local regime $\rho \in (0,1)$} \label{sec:Watson-non-local}

In the non-local regime $\rho \in (0,1)$, the scaling limit of $M_n$ in \eqref{S- M} becomes $\rho$-dependent. This dependence arises because the distribution of $M_n$ is also affected by the pairwise products of the cosine components (see the representations \eqref{representation} and \eqref{representation2}), rather than solely by the inner products of the uniform components. Before stating our result, define
\begin{align}
    & A(\rho) := \sqrt{2(1+\rho^2)}; \qquad a(\rho):= 2 \left[ \sqrt{2(1+\rho^2)} - \rho \right]; \label{A-a}\\
    & K(\rho):= \frac{\rho}{\sqrt{2\pi} \cdot A(\rho) \cdot \sqrt{\lb A(\rho) - \rho \rb\cdot \lb A(\rho) -2\rho \rb}}. \label{K rho}
\end{align}
Our main result in the non-local regime $\rho<1$ is as follows.
\begin{thm} \label{thm:watson-rho<1}
    Assume that
    \[
    \frac{p}{\lb \log n \rb^3} \to \infty; \qquad \rho_n=\frac{\Delta_n}{\sqrt{p \log n}} \to \rho \in (0,1).
    \]
    With $M_n$ defined as in \eqref{S- M} and $x  \in \mb R$, we have 
    \[
    p \cdot M_n^2 -  a\lb \rho_n \rb^2 \cdot \log n + \frac{a\lb \rho_n \rb}{A\lb \rho_n \rb}\cdot \log \log n \stackrel{d}{\to} G_\rho
    \]
    where $G_\rho$ is a Gumbel distribution with CDF
    \[
    F_\rho (x):= \exp\left[ -K(\rho)\cdot \exp \lb - \frac{A(\rho)}{2 a(\rho)} \cdot x \rb \right].
    \]
\end{thm}

Let us compare the results of Theorem \ref{thm:watson-rho<1} and Theorem \ref{thm: Watson}. Although both limiting distributions are of Gumbel type, their scalings are markedly different: Theorem \ref{thm: Watson} implies that
\[
M_n^2 = \frac{4 \log n}{p} - \frac{\log \log n}{p} + o_{\mb P} \lb \frac{\log \log n}{p} \rb
\] 
In particular, the first- and second-order approximations of $M_n^2$ are independent of the value of $\rho$ as long as $\rho >1$. This scaling is also the same as under the null (uniformity). In contrast, Theorem \ref{thm:watson-rho<1} gives
\begin{align*}
M_n^2  &= \la 4\left[ \sqrt{2(1+\rho^2)} - \rho \right]^2 + o(1) \ra \cdot \frac{\log n}{p} - \left[ \frac{2\sqrt{1+\rho^2} - \sqrt{2}\rho}{\sqrt{1+\rho^2}} + o\lb 1 \rb \right] \cdot \frac{\log \log n}{p}  \\
&+  o_{\mb P} \lb \frac{\log \log n}{p} \rb
\end{align*}
for which the first- and second-order approximations depend on the value of $\rho$ when $\rho < 1$. Here the $o(1)$ term is due to $\rho_n \to \rho$. Interestingly, the transition from the local regime to the non-local regime is continuous but not smooth. 

We note that in this non-local regime, we require a slightly stronger condition on the relative growth of the dimension and sample size: we need $p/(\log n)^3 \to \infty$ instead of $p/(\log n)^2 \to \infty$. This is a technical condition that arises in the large-deviation analysis. It might be possible to get a sharper condition by a more careful asymptotic analysis. However, since our goal is to highlight the phenomenon rather than to investigate the sharpest possible conditions, we will not pursue this technical extension in the current paper.

\subsection{The critical regime $\rho=1$} \label{sec:watson-critical}

We now investigate the critical regime $\rho=1$. Interestingly, the scaling in this regime is discontinuous at the second-order approximation. 

\begin{thm} \label{thm:Watson-critical}
      Assume that
    \[
    \frac{p}{\lb \log n \rb^3} \to \infty; \qquad  \sqrt{\log n} \lb \rho_n -1 \rb \to \Delta \in \mb{R}. 
    \]
    Define
\[
 K_{\rm cr} \lb \Delta \rb
 :=
 \frac{1}{4\pi}
 \int_{-\infty}^{\infty}
 \exp\left\{-\Delta y^2-\frac{y^4}{4}\right\}\,dy .
\]
    With $M_n$ defined as in \eqref{S- M} and $x  \in \mb R$, we have 
    \[
 \mb P_{\kappa_n,\bm\mu_n}
 \left(
 pM_n^2\le 4 \log n-\frac12\log \log n +x
 \right)
 {\to}
 \exp\lb -K_{\rm cr} \lb \Delta \rb \cdot  e^{-x/2} \rb.
\]
\end{thm}

In addition to the assumption that $\rho_n \to 1$, we also require that it fluctuates around $1$ at the scale $1/\sqrt{\log n}$. This is a technical assumption that simplifies the analysis. We believe it is possible to prove the result in a slightly more general setting, but we do not pursue such an extension in the present article. 

Now Theorem \ref{thm:Watson-critical} implies that
\[
M_n^2 = \frac{4 \log n}{p} -  \frac{1}{2} \cdot \frac{\log \log n}{p} + o_{\mb P} \lb \frac{\log \log n}{p} \rb.
\]
Comparing the above with Theorems \ref{thm: Watson} and \ref{thm:watson-rho<1}, we see that the first-order approximation term undergoes a continuously differentiable transition (as a function of $\rho$) across the three regimes (ignoring the $o(1)$ terms), as shown in Table \ref{tab:watson-1st-order} below.  

\begin{table}[ht]
\centering
\begingroup
\renewcommand{\arraystretch}{1.8}
\begin{tabular}
{
|>{\centering\arraybackslash}m{0.22\textwidth}
|>{\centering\arraybackslash}m{0.58\textwidth}|
}
\hline
Regime & First-order approximation \\ \hline

\(\rho>1\)
&
\(\displaystyle \frac{4 \log n}{p}\)
\\ \hline

\(\rho=1\)
&
\(\displaystyle \frac{4 \log n}{p}\)
\\ \hline

\(\rho\in(0,1)\)
&
\(\displaystyle
4 \left[ \sqrt{2(1+\rho^2)}-\rho \right]^2
\cdot \frac{\log n}{p}
\)
\\ \hline

\end{tabular}
\endgroup
\caption{First-order approximation.}
\label{tab:watson-1st-order}
\end{table}

From Table \ref{tab:watson-1st-order}, it is easy to check that the leading constant in the first-order approximation, which is of order $\log n/p$, is a differentiable function of $\rho$. However, the second-order approximation undergoes a discontinuous transition at $\rho=1$, as demonstrated in Table \ref{tab:watson-2nd-order} below.
\begin{table}[ht]
\centering
\begingroup
\renewcommand{\arraystretch}{1.8}
\begin{tabular}{
|>{\centering\arraybackslash}m{0.22\textwidth}
|>{\centering\arraybackslash}m{0.58\textwidth}|
}
\hline
Regime & Second-order approximation \\ \hline

\(\rho>1\)
&
\(\displaystyle  - \frac{\log \log n}{p} \)
\\ \hline

\(\rho=1\)
&
\(\displaystyle  - \frac{1}{2}  \cdot\frac{\log \log n}{p}  \)
\\ \hline

\(\rho\in(0,1)\)
&
\(\displaystyle
-\frac{2\sqrt{1+\rho^2} - \sqrt{2}\rho}{\sqrt{1+\rho^2}} \cdot \frac{\log \log n}{p}
\)
\\ \hline

\end{tabular}
\endgroup
\caption{Second-order approximation.}
\label{tab:watson-2nd-order}
\end{table}

As $\rho \to 1$, one can check that the coefficients of the term $\log \log n/p$ on both sides agree and are equal to $-1$, whereas the coefficient is $-1/2$ at $\rho=1$. We believe this is related to the scale at which $\rho_n$ converges to $1$, namely, that convergence at different rates might yield different second-order approximations. There are technical difficulties in extending the results to these settings since the analysis of the correlation term in the Chen--Stein framework (namely the term $b_2$ in Lemma \ref{lem:AGG} below) becomes much more complicated.

\section{Conclusion and discussion} \label{sec:conclusion}
In this paper, we derive the asymptotic distribution of the packing test for uniformity under the high-dimensional FvML and Watson models. 
As a consequence, we precisely characterize the detection thresholds of the packing test under these two models. 
In the high-dimensional FvML model, we find that the detection threshold of the packing test is
\[
\Theta\!\left( \frac{p^{3/4}}{(\log n)^{1/4}} \right),
\]
which has the optimal dependence on the dimension but is suboptimal in terms of the sample size. 
In contrast, under the Watson model, the detection threshold is of the form
\[
\frac{p}{2} - \Theta\!\left(\sqrt{p \log n}\right),
\]
which differs from both the FvML threshold and the known minimax rates.

Our method is based on Poisson approximation and truncation techniques to prove asymptotic independence. We precisely identify the local regime under the Watson model in which the packing test has strictly nontrivial asymptotic power (strictly between the size $\alpha$ and one). 
Outside the local regime, we identify a non-local regime in which the scaling of the largest squared inner product undergoes a discontinuous phase transition, and we show that the packing test is asymptotically powerful throughout this regime.

We believe the framework developed in the current paper can be useful for other models as well. A natural extension is to consider the class of monotone, rotationally invariant distributions in \cite{Cutting-P-V}, which generalizes the two models considered in the present paper. However, extending the framework to those settings requires a new analysis to address the loss of strong convexity, for which Laplace's method might no longer be effective.

\section{Proof idea}

We give a brief heuristic derivation of the three regimes in the Watson model. 
The same type of derivation can also be adapted to the FvML model, where it is simpler because the deviation from uniformity comes from the mean, whereas in the Watson model it comes from the variance. 
By the tangent--normal decomposition \eqref{representation2}, we have the heuristic representation
\[
\la \bm X_i^\top \bm X_j \ra_{1\leq i<j \leq n}
\stackrel{d}{\approx}
\la
T_iT_j
+
\sqrt{1-T_i^2}\sqrt{1-T_j^2}\, W_{ij}
\ra_{1\leq i<j \leq n},
\]
where the variables \(W_{ij}\) are approximately independent normal random variables with variance \(1/p\), and the \(T_i\) are i.i.d. with density \eqref{wT}.

Let us first consider the case \(\rho_n\to\rho>1\). In this regime, the limiting distribution can be guessed from a point-process heuristic. Since the \(T_i\) are small, we make the approximation
\[
\la \sqrt{p}\,\bm X_i^\top \bm X_j \ra_{1\leq i<j \leq n}
\stackrel{d}{\approx}
\la
\sqrt{p}\,T_iT_j+\sqrt{p}\,W_{ij}
\ra_{1\leq i<j \leq n}.
\]
This approximation should be understood in the point-process sense. It can be justified rigorously, although the proof below does not rely on such a point-process result. Since
\[
\sqrt{p-2\kappa_n}\,T_i \approx Z_i,
\qquad Z_i\sim N(0,1),
\]
we can rewrite the preceding approximation as
\[
\la \sqrt{p}\,\bm X_i^\top \bm X_j \ra_{1\leq i<j \leq n}
\stackrel{d}{\approx}
\la
\frac{\sqrt p}{p-2\kappa_n}\,Z_iZ_j+\sqrt p\,W_{ij}
\ra_{1\leq i<j \leq n},
\]
where the \(Z_i\) are i.i.d. standard normal random variables.

It is well known that, for \(\widetilde a_n\sim 2\sqrt{\log n}\),
\[
\la
\widetilde a_n\left(\sqrt p\,W_{ij}-\widetilde a_n\right)
\ra_{1\leq i<j\leq n}
\]
converges in distribution to a Poisson point process on \(\mathbb R\) with intensity measure \(e^{-x}\,dx\). Therefore, heuristically,
\[
\la
\widetilde a_n
\left(
\sqrt p\,\bm X_i^\top \bm X_j-\widetilde a_n
\right)
\ra_{1\leq i<j\leq n}
\stackrel{d}{\approx}
\la
\frac{\sqrt p\cdot\widetilde a_n}{p-2\kappa_n}\,\widetilde Z_i+\Xi_i
\ra_{i\geq 1},
\]
where \(\{\Xi_i:i\geq1\}\) are the points of a Poisson point process with intensity \(e^{-x}\,dx\), and the \(\widetilde Z_i\) are i.i.d. copies of the product of two independent standard normal random variables, independent of the \(\Xi_i\).

The last approximation is again in the point-process sense. In this heuristic derivation, we have ignored the weak dependence among the products \(Z_iZ_j\), which does not affect the first-order extreme-value behavior. Since
\[
\frac{\sqrt p\cdot\widetilde a_n}{p-2\kappa_n}\to \rho^{-1},
\]
the limiting point process is the image of the marked point process
\[
\la \lb \widetilde Z_i,\Xi_i \rb: i\geq1\ra
\]
under the map
$
(z,\xi)\longmapsto \rho^{-1}z+\xi.
$
Thus we expect the limiting distribution of the maximum to satisfy
\begin{align*}
\mb P\left(
\widetilde a_n
\left(
\sqrt p \cdot \max_{i<j}\bm X_i^\top\bm X_j-\widetilde a_n
\right)
\leq x
\right)  
& \approx
\mb P\left(
\text{no pair }(\widetilde Z_i,\Xi_i)
\text{ satisfies }
\rho^{-1}\widetilde Z_i+\Xi_i>x
\right) \\
&=
\exp\left\{
-
\iint_{\{(\xi,z):\,\rho^{-1}z+\xi>x\}}
e^{-\xi}\,d\xi\,dF_{\widetilde Z}(z)
\right\} \\
&=
\exp\left\{
-
e^{-x}
\int_{\mathbb R}
\exp\left(\frac{z}{\rho}\right)
\,dF_{\widetilde Z}(z)
\right\} \\
&=
\exp\left\{
-
e^{-x}\,
\mb E\exp\left(\rho^{-1}\widetilde Z_1\right)
\right\}
=
\exp\left\{
-\frac{e^{-x}}{\sqrt{1-\rho^{-2}}}
\right\}.
\end{align*}
Here \(F_{\widetilde Z}\) denotes the law of \(\widetilde Z_1\), where \(\widetilde Z_1\) is the product of two independent standard normal random variables. The last equality follows from the identity
\[
\mb E\exp\left(\lambda Z_1Z_2\right)
=
\frac{1}{\sqrt{1-\lambda^2}},
\qquad |\lambda|<1,
\]
where \(Z_1,Z_2\) are independent standard normal random variables, applied with \(\lambda=\rho^{-1}\).

The preceding limiting distribution makes sense only when \(\rho>1\). This is precisely the local regime of Theorem \ref{thm: Watson}, and a different approach is needed for the case $\rho \leq 1$. The same calculation also gives the scaling for the largest squared inner product: squaring doubles the edge fluctuation, so in the last display we replace \(x\) by \(x/2\) to obtain the limiting distribution of the largest squared inner product. We emphasize that this is only a heuristic point-process derivation and the rigorous proof below uses Poisson approximation rather than point-process convergence.

The point-process heuristic also suggests why the dependence structure must be handled carefully. In the Watson model, the perturbation involves products \(Z_iZ_j\). Pairs such as \(Z_iZ_j\) and \(Z_iZ_k\) share the common factor \(Z_i\), and this shared factor can increase the probability of simultaneous exceedances. This is why the correlation term \(b_2\) in the Chen--Stein bound requires a more delicate argument. In the rigorous proof, we handle this issue by removing a collection of rare events \(\mc A_{ij}\); see Section \ref{sec:Poisson-watson-local} for details. The analogous derivation for the FvML model is almost identical, except that the product \(Z_1Z_2\) is replaced by \(Z_1+Z_2\). Since \(Z_1+Z_2\) is exactly Gaussian, the FvML case is easier than the Watson case.

We next explain the scaling in the remaining regimes. The same heuristic can be used to identify the correct first-order scale for all values of \(\rho\). We look for a constant \(c\) such that
\[
n^2\,\mb P\left(\sqrt p\,S_{12}\geq c\sqrt{\log n}\right)\asymp 1,
\]
where \(S_{12}\) is defined in \eqref{S- M}. Since
\[
\sqrt p\max_{i<j}W_{ij}
\]
is of order \(\sqrt{\log n}\), the product term \(\sqrt p\,T_1T_2\) must be put on the same scale. To do this, we make the change of variables
\[
T_i=\alpha_i\left(\frac{\log n}{p}\right)^{1/4},
\qquad i=1,2.
\]
The purpose of this change of variables is that \(\sqrt p\,T_1T_2\) is then on the scale \(\sqrt{\log n}\). A direct computation gives the following approximation for the density of \(\alpha_i\):
\[
g_n(\alpha_i)
\approx
\sqrt{\frac{\rho_n\log n}{\pi}}\,
\exp\left\{
-\log n
\left(
\rho_n\alpha_i^2+\frac{\alpha_i^4}{4}
\right)
\right\},
\]
where \(\rho_n\) is defined in Theorem \ref{thm:watson-rho<1}.

Thus the exceedance probability
\[
\mb P\left(\sqrt p\,S_{12}\geq c\sqrt{\log n}\right)
\]
is approximately
\begin{align*}
&\iint
\exp\left\{
-\log n
\left[
\rho_n(\alpha_1^2+\alpha_2^2)
+
\frac{\alpha_1^4+\alpha_2^4}{4}
+
\frac{(c-\alpha_1\alpha_2)_+^2}{2}
\right]
\right\}
\,d\alpha_1\,d\alpha_2 \\
&\qquad \approx
\exp\left\{
-\log n
\cdot
\inf_{\alpha_1,\alpha_2}
\left[
\rho_n(\alpha_1^2+\alpha_2^2)
+
\frac{\alpha_1^4+\alpha_2^4}{4}
+
\frac{(c-\alpha_1\alpha_2)_+^2}{2}
\right]
\right\},
\end{align*}
where the last approximation follows from Laplace's method, and \(x_+=\max\{x,0\}\). A direct calculation yields
\[
\inf_{\alpha_1,\alpha_2}
\left[
\rho_n(\alpha_1^2+\alpha_2^2)
+
\frac{\alpha_1^4+\alpha_2^4}{4}
+
\frac{(c-\alpha_1\alpha_2)_+^2}{2}
\right]
=
\begin{cases}
\frac{c^2}{2}, & \text{if } c\leq 2\rho_n,\\[0.4em]
\frac{c^2}{4}+\rho_n c-\rho_n^2, & \text{if } c>2\rho_n.
\end{cases}
\]
Suppose \(\rho_n\to\rho\). To make
\[
\mb P\left(\sqrt p\,S_{12}\geq c\sqrt{\log n}\right)
\]
of order \(\Theta(n^{-2})\), we need
\[
c=
\begin{cases}
2, & \text{if } \rho\geq1,\\[0.4em]
2\left[\sqrt{2(1+\rho^2)}-\rho\right], & \text{if } \rho<1.
\end{cases}
\]
Thus the tail probability changes behavior depending on whether \(\rho_n\to\rho<1\) or \(\rho_n\to\rho\geq1\). This explains why the scaling changes between the local and non-local regimes in the Watson model.

The proof of asymptotic independence in the Poisson approximation, specifically the treatment of the term \(b_2\) in Lemma \ref{lem:AGG}, proceeds by truncating to small neighborhoods of the minimizers of the variational problem above. These neighborhoods contain the main contribution to the relevant joint exceedance probabilities. 
The minimizers are different in the two regimes. When \(\rho\geq1\), there is a unique minimizer near \((0,0)\). When \(\rho<1\), for the right-tail problem there are two minimizers near
\[
\left(\sqrt{A(\rho)-2\rho},\sqrt{A(\rho)-2\rho}\right)
\qquad\text{and}\qquad
\left(-\sqrt{A(\rho)-2\rho},-\sqrt{A(\rho)-2\rho}\right),
\]
where \(A(\rho)\) is defined in \eqref{A-a}. The change in the saddle-point structure, together with the loss of uniform strong convexity near the critical regime, makes the proof in the non-local regime more involved than the proof in the local regime.

\section{Proof of Theorem \ref{thm: FvML}} \label{sec:proof-fvml}
 
\subsection{Preliminaries}

We first collect some useful asymptotic results. Suppose $\bm X$ has the FvML distribution with concentration parameter $\kappa$ and location $\bm \mu \equiv \bm e_1$ as in \eqref{class fvml}. Then we have the tangent--normal decomposition
 \begin{align} \label{representation}
 \bm X  \stackrel{d}{=} \lb T, \sqrt{1-T^2} \cdot \bm V \rb.
 \end{align}
Here \(T=\bm X^\top\bm\mu\) is the cosine component, equivalently,
the coordinate along $\bm \mu$, \(\bm V\sim
\mbox{Unif}(\mb S^{p-2})\), and \(T\) is independent of \(\bm V\) with density
\begin{align} \label{T}
f_T(t) \propto e^{\kappa t} \cdot \lb 1-t^2 \rb^{(p-3)/2}.
\end{align}

\begin{prop} \label{Herbst}
Let $T$ be a random variable with the density in \eqref{T}, where $\kappa =o(p)$, and put $m= \mb E T$. Then
\begin{equation}\label{eq:fvml-mean-var}
 m=\frac{\kappa}{p}(1+o(1)),
 \qquad
 \Var(T)\le \frac{1}{p-3}.
\end{equation}
Moreover, for all $\lambda \in \mb R$,
\begin{equation}\label{eq:fvml-subgaussian}
 \mb E\exp\{\lambda(T-m)\}
 \le \exp\left\{\frac{\lambda^2}{2(p-3)}\right\}.
\end{equation}
\end{prop}

\noindent \textbf{Proof of Proposition \ref{Herbst}.}
Let \(\nu\) denote the law of \(T\).  By \eqref{T}, \(T\) has density
with respect to Lebesgue measure on \((-1,1)\) of the form
\[
 f_T(z)
 \propto 
 \exp\{\ell(z)\},
 \qquad -1<z<1,
\]
where
\[
 \ell(z):=\kappa z+\frac{p-3}{2}\log(1-z^2).
\]
Note that
\[
 \ell''(z)
 =
 -(p-3)\frac{1+z^2}{(1-z^2)^2} \leq -(p-3),
\]
so the density of $T$ can be written in the form $\exp \lb -V(z) \rb$ with
\[
V''(z) \geq p-3.
\]
By the Brascamp--Lieb inequality [see equation (1.9) in \cite{nguyen2014} and the references therein], we obtain $\Var (T) \leq 1/(p-3)$.
This proves the second claim in \eqref{eq:fvml-mean-var}. The first claim
in \eqref{eq:fvml-mean-var} follows, under the standing assumption
\(\kappa=o(p)\), from the proof of Lemma 4 in \cite{Ley-P-2}.

We now prove the sub-Gaussian bound \eqref{eq:fvml-subgaussian} using the Herbst argument. For
\(\lambda\in\mb R\), define the tilted law \(\nu_\lambda\) by
\[
 d\nu_\lambda(z)
 :=
 \frac{e^{\lambda z}}{\mb E e^{\lambda T}}\,d\nu(z).
\]
Hence the tilted law has density proportional to $\exp \lb -V_\lambda(z) \rb$, where 
\[
 V_\lambda(z):=V(z)-\lambda z.
\]
Observe that $V_\lambda''(z)=V''(z)\ge p-3$.
Therefore, another application of the Brascamp--Lieb inequality under the tilted law \(\nu_\lambda\) gives
\[
 \Var_{\nu_\lambda}(T)
 \le
 \frac{1}{p-3}.
\]
Now define
\[
 \psi(\lambda)
 :=
 \log \mb E\exp\{\lambda(T-m)\},
 \qquad
 m=\mb E T.
\]
A direct calculation yields
\[
 \psi'(\lambda)
 =
 \mb E_{\nu_\lambda}T-m,
 \qquad 
 \psi''(\lambda)
 =
 \Var_{\nu_\lambda}(T) \leq \frac{1}{p-3}.
\]
Moreover,
$
 \psi(0)=0
$
and
$
 \psi'(0)=\mb ET-m=0.
$
Integrating the second-derivative bound twice gives, for all \(\lambda\ge0\),
\[
 \psi(\lambda)
 =
 \int_0^\lambda(\lambda-s)\psi''(s)\,ds
 \le
 \frac{1}{p-3}\int_0^\lambda(\lambda-s)\,ds
 =
 \frac{\lambda^2}{2(p-3)}.
\]
For \(\lambda<0\), the same argument gives
\[
 \psi(\lambda)
 =
 \int_\lambda^0(s-\lambda)\psi''(s)\,ds
 \le
 \frac{1}{p-3}\int_\lambda^0(s-\lambda)\,ds
 =
 \frac{\lambda^2}{2(p-3)}.
\]
Therefore, for every \(\lambda\in\mb R\),
\[
 \log \mb E\exp \Big[ \lambda(T-m)\Big]
 \le
 \frac{\lambda^2}{2(p-3)},
\]
which is \eqref{eq:fvml-subgaussian}. This completes the proof. $\hfill$ $\square$

We next collect some useful facts concerning the moment generating function of the product of two independent copies of $T$.

\begin{prop} \label{prop:exponential moment}
 Suppose $p/\lb  \log n\rb^2 \to \infty$. Recall $T$ from Proposition \ref{Herbst}.  If \(T_1,T_2\) are independent copies of \(T\), then 
\begin{equation}\label{eq:fvml-product-mgf-strong}
 \mb E \exp \lb \pm u_n \sqrt p\,T_1T_2 \rb \to e^{\pm2\tau^2}
\end{equation}
where $u_n$ is defined in \eqref{u-t} and $\kappa=\tau\cdot p^{3/4}/ \lb \log n \rb^{1/4}$. Moreover, for a fixed $q>1$,
\begin{align} \label{eq:sup-moment-product}
     \sup_{n \geq 1}  \Big\{ \mb E \exp \lb q u_n \sqrt p\,|T_1T_2|\rb \Big\} <\infty.
\end{align}
\end{prop}

\noindent \textbf{Proof of Proposition \ref{prop:exponential moment}.} 
Write
\[
 T_i=m+\xi_i,\qquad i=1,2,
\]
and set
\[
 \beta_n:=u_n\sqrt p.
\]
By \eqref{eq:fvml-subgaussian}, for every \(s\ge0\),
\[
 \mb P(|\xi_i|>s)
 \le
 2\exp\left\{-\frac{(p-3)s^2}{2}\right\}.
\]
Consequently, there is a universal constant \(C\) such that, for all
\(0\le a\le (p-3)/4\),
\begin{equation}\label{eq:fvml-square-exp}
 \mb E e^{a\xi_i^2}\le C.
\end{equation}

{\it \underline{Proof of \eqref{eq:fvml-product-mgf-strong}}.} Write
\[
 \beta_nT_1T_2
 =
 \beta_nm^2+\beta_nm(\xi_1+\xi_2)+\beta_n\xi_1\xi_2.
\]
The deterministic term satisfies
\begin{equation}\label{eq:fvml-deterministic-shift}
 \beta_nm^2
 =
 u_n\sqrt p\,\frac{\kappa^2}{p^2}(1+o(1))
 =
 u_n \frac{\kappa^2}{p^{3/2}}(1+o(1))
 \to
 2\tau^2.
\end{equation}
Moreover, using \(\Var(\xi_i)=\Var(T_i)\le (p-3)^{-1}\), we get
\begin{align*}
 \mb E\left[ \beta_nm(\xi_1+\xi_2)\right]^2
 & \lesssim
 \frac{\beta_n^2m^2}{p}
 =
 u_n^2m^2
 =
 o(1), \\
  \mb E(\beta_n\xi_1\xi_2)^2
 & \lesssim
 \frac{\beta_n^2}{p^2}
 =
 \frac{u_n^2}{p}
 =
 o(1).
\end{align*}
Thus
\begin{equation}\label{eq:fvml-random-remainder-L2}
 \mc Q_n:=
 \beta_nm(\xi_1+\xi_2)+\beta_n\xi_1\xi_2
 \to0
 \qquad\text{in }L^2.
\end{equation}
To deduce \eqref{eq:fvml-product-mgf-strong}, it suffices to verify exponential uniform integrability. By
\eqref{eq:fvml-subgaussian}, for every \(a\ge0\),
\[
 \mb E e^{a|\xi_i|}
 \le
 \mb E e^{a\xi_i}+\mb E e^{-a\xi_i}
 \le
 2\exp\left\{\frac{a^2}{2(p-3)}\right\}.
\]
Fix \(q>1\). Taking \(a=3q\beta_nm\) and using
\[
 \frac{\beta_n^2m^2}{p}
 =
 u_n^2m^2
 =
 O\left( \frac{\log n}{\sqrt{p}}\right)
 =
 o(1),
\]
we obtain
\begin{equation}\label{eq:fvml-linear-exp-bound}
 \sup_n \mb E e^{3q\beta_nm|\xi_i|}<\infty,
 \qquad i=1,2.
\end{equation}
We next treat the product term in \eqref{eq:fvml-random-remainder-L2}. Choose a sufficiently small fixed
\(\eta_0>0\). By the inequality \(ab\le \eta_0a^2+b^2/(4\eta_0)\),
\[
 3q\beta_n \cdot |\xi_1\xi_2|
 \le
 \eta_0p\cdot \xi_1^2
 +
 \frac{9q^2\beta_n^2}{4\eta_0p}\cdot \xi_2^2.
\]
The first coefficient is at most \((p-3)/4\) for all large \(n\), provided
\(\eta_0\) is chosen sufficiently small. The second coefficient equals
\[
 \frac{9q^2\beta_n^2}{4\eta_0p}
 =
 \frac{9q^2u_n^2}{4\eta_0}
 =
 O(\log n)
 =
 o(p),
\]
and hence is also at most \((p-3)/4\) for all sufficiently large \(n\). Therefore, by
\eqref{eq:fvml-square-exp} and independence,
\begin{equation}\label{eq:fvml-product-exp-bound}
 \sup_n \mb E e^{3q\beta_n|\xi_1\xi_2|}<\infty.
\end{equation}
By H\"older's inequality, \eqref{eq:fvml-linear-exp-bound}, and
\eqref{eq:fvml-product-exp-bound},
\[
\begin{aligned}
 \mb E e^{q|\mc Q_n|}
 &\le
 \mb E\exp\Big\{
 q\beta_nm|\xi_1|
 +q\beta_nm|\xi_2|
 +q\beta_n|\xi_1\xi_2|
 \Big\}  \\
 &\le
 \left(\mb E e^{3q\beta_nm|\xi_1|}\right)^{1/3}
 \left(\mb E e^{3q\beta_nm|\xi_2|}\right)^{1/3}
 \left(\mb E e^{3q\beta_n|\xi_1\xi_2|}\right)^{1/3}
 \le C.
\end{aligned}
\]
Hence \(\{e^{\pm \mc Q_n}\}_n\) is uniformly integrable. Since
\(\mc Q_n\to0\) in \(L^2\), and hence in probability, Vitali's theorem gives
\[
 \mb E e^{\pm \mc Q_n}\to1.
\]
Combining this with \eqref{eq:fvml-deterministic-shift}, we obtain
\[
 \mb E e^{\pm\beta_nT_1T_2}
 =
 e^{\pm\beta_nm^2}\mb E e^{\pm \mc Q_n}
 \to
 e^{\pm2\tau^2}.
\]

{\it \underline{Proof of \eqref{eq:sup-moment-product}}.} Using
$
 |T_1T_2|
 \le
 m^2+m(|\xi_1|+|\xi_2|)+|\xi_1\xi_2|
$,
we have
\[
 e^{q\beta_n|T_1T_2|}
 \le
 e^{q\beta_nm^2}
 e^{q\beta_nm|\xi_1|}
 e^{q\beta_nm|\xi_2|}
 e^{q\beta_n|\xi_1\xi_2|}.
\]
The deterministic factor \(e^{q\beta_nm^2}\) is bounded by
\eqref{eq:fvml-deterministic-shift}. Applying H\"older's inequality once
more, together with \eqref{eq:fvml-linear-exp-bound} and
\eqref{eq:fvml-product-exp-bound}, gives
\[
 \sup_n \mb E e^{q\beta_n|T_1T_2|}<\infty.
\]
This completes the proof.
\(\hfill\square\)

The following tail asymptotic result will be used in the proof of Theorem \ref{thm: FvML}. 

\begin{prop}\label{lem:fvml-onepair}
Suppose \(\bm X_1, \bm X_2\) are i.i.d. with the FvML distribution as in
\eqref{representation}, and
\[
 \frac{p}{(\log n)^2}\to\infty,
 \qquad
 \kappa=\kappa_n=\tau\frac{p^{3/4}}{(\log n)^{1/4}}.
\]
Then, with \(t_n=t_n(x)\) as in \eqref{u-t},
\begin{align}
 \mathbb P_{\kappa_n} \lb \bm X_1^\top \bm X_2>t_n \rb
 &=\frac{1+o(1)}{2\sqrt{2\pi}}\cdot n^{-2} \cdot e^{-x/2+2\tau^2},
 \label{eq:fvml-right}\\
 \mathbb P_{\kappa_n} \lb  \bm X_1^\top \bm X_2<-t_n \rb
 &=\frac{1+o(1)}{2\sqrt{2\pi}} \cdot n^{-2} \cdot e^{-x/2-2\tau^2}.
 \label{eq:fvml-left}
\end{align}
Consequently,
\begin{equation}\label{eq:fvml-two-sided}
 p_1:=\mathbb P_{\kappa_n}(|\bm X_1^\top \bm X_2|>t_n)
 =\frac{1+o(1)}{\sqrt{2\pi}} \cdot \frac{e^{-x/2}}{n^2}
 \cdot \cosh(2\tau^2).
\end{equation}
\end{prop}

\noindent \textbf{Proof of Proposition \ref{lem:fvml-onepair}.}
By rotation invariance, we may assume \(\bm\mu=\bm e_1\). Using
\eqref{representation}, write
\[
 \bm X_i=\left(T_i,\sqrt{1-T_i^2}\,\bm V_i\right),
 \qquad i=1,2,
\]
where \(\bm V_1,\bm V_2\) are independent and uniform on
\(\mb S^{p-2}\), independent of \(T_1,T_2\). Put
\[
 d:=p-1,
 \qquad
 Z:=\bm V_1^\top \bm V_2.
\]
Then \(Z\) has the same distribution as \(Z_d\) in \eqref{Psi}, and
\[
 \bm X_1^\top \bm X_2
 =
 T_1T_2+
 \sqrt{(1-T_1^2)(1-T_2^2)}\,Z.
\]
For readability, write
\[
 A_n:=\sqrt{(1-T_1^2)(1-T_2^2)},
 \qquad
 B_n:=T_1T_2.
\]
Thus $\bm X_1^\top \bm X_2=A_n Z+B_n$. 

{\it \underline{Step 1: Linearization.}} We will show that
\begin{equation}\label{eq:Y-plus-expansion}
 \sqrt d\cdot \frac{t_n-B_n}{A_n}
 =
 u_n-\sqrt p\,B_n + r_{+,n},
 \qquad
 u_n|r_{+,n}|\to0
 \quad\text{in probability}.
\end{equation}
We will also show that
\begin{equation}\label{eq:Y-minus-expansion}
 \sqrt d\cdot \frac{t_n+B_n}{A_n}
 =
 u_n+\sqrt p\,B_n +r_{-,n},
 \qquad
 u_n|r_{-,n}|\to0
 \quad\text{in probability}.
\end{equation}

By Proposition \ref{Herbst},
\[
 m=\frac{\kappa}{p}(1+o(1)),
 \qquad
 T_i=m+O_{\mb P}(p^{-1/2}).
\]
Therefore
\[
 T_i^2
 =
 O_{\mb P}\left(m^2+\frac1p\right)
 =
 O_{\mb P}\left(\frac{1}{\sqrt{p\log n}}+\frac1p\right).
\]
In particular \(T_i^2=o_{\mb P}(1)\), and hence we can write
\[
 A_n^{-1}
 =
 1+O_{\mb P}(T_1^2+T_2^2).
\]
Also, by Proposition \ref{prop:exponential moment}, $u_n\sqrt p\,B_n=u_n\sqrt p\,T_1T_2$
is tight. Since \(d=p-1\) and \(t_n=u_n/\sqrt p\), we have
\[
 \sqrt d\,t_n
 =
 u_n\sqrt{\frac{p-1}{p}}
 =
 u_n+O\left(\frac{u_n}{p}\right).
\]
Consequently,
\[
\begin{aligned}
u_n \left[  \sqrt d\cdot \frac{t_n-B_n}{A_n}-(u_n-\sqrt p\,B_n) \right]
&=
u_n \left[ (\sqrt d\,t_n-u_n)-(\sqrt d-\sqrt p)B_n
  +\sqrt d\,(t_n-B_n)(A_n^{-1}-1) \right].
\end{aligned}
\]
Multiplying by \(u_n\), the first term is \(O(u_n^2/p)=o(1)\), the
second is \(o_{\mb P}(1)\) by tightness of \(u_n\sqrt p\,B_n\), and the
last term is bounded by
\[
 O_{\mb P}\lb u_n^2(T_1^2+T_2^2)\rb
 =
 O_{\mb P}\left(\sqrt{\frac{\log n}{p}}+\frac{\log n}{p}\right)
 =
 o_{\mb P}(1).
\]
Thus \eqref{eq:Y-plus-expansion} holds. The proof of
\eqref{eq:Y-minus-expansion} is identical, with \(B_n\) replaced by
\(-B_n\).

{\it \underline{Step 2: Right-tail asymptotic}.} Conditioning on \(T_1,T_2\) and using \eqref{eq:Y-plus-expansion}, we obtain
\begin{align*}
 \mb P_{\kappa_n}\left(\bm X_1^\top\bm X_2>t_n\,\middle|\,T_1,T_2\right)
 &=
 \mb P(A_nZ+B_n>t_n\,|\,T_1,T_2) \\
 &= \Psi_d\left(\sqrt d\cdot \frac{t_n-B_n}{A_n}\right) \\
 &=  \Psi_d\left(u_n-\sqrt p\,B_n +r_{+,n}\right).
\end{align*}
We claim that
\begin{equation}\label{eq:right-tail-L1}
 \mb E
 \left|
 \frac{
 \Psi_d\left(u_n-\sqrt p\,B_n+r_{+,n}\right)
 }{\Psi_d(u_n)}
 -
 e^{u_n\sqrt p\,B_n}
 \right|
 \to 0.
\end{equation}
To see this, fix \(0 < M<\infty\) and \(\varepsilon>0\), and consider the event
\[
\mc A_n:=  \left\{ |u_n\sqrt p\,B_n|\le M,\ u_n|r_{+,n}|\le\varepsilon \right\}.
\]
On $\mc A_n$, by applying \eqref{eq:Psi-local-ratio} with $a=\sqrt p\,B_n-r_{+,n}$, we obtain
\[
 \frac{
 \Psi_d(u_n-\sqrt p\,B_n+r_{+,n})
 }{\Psi_d(u_n)} \cdot \mathbf{1}_{\mc A_n}
 =
 \exp\left\{u_n(\sqrt p\,B_n-r_{+,n})
 -\frac{(\sqrt p\,B_n-r_{+,n})^2}{2}\right\}(1+o(1)) \cdot \mathbf{1}_{\mc A_n}.
\]
Since \(|\sqrt p\,B_n-r_{+,n}| \leq (M+\ve)/u_n\) on $\mc A_n$, we may write
\[
\frac{
 \Psi_d(u_n-\sqrt p\,B_n+r_{+,n})
 }{\Psi_d(u_n)} \cdot \mathbf{1}_{\mc A_n}
 =  e^{u_n\sqrt p\,B_n} \cdot \left[ 1+ \delta_n(M,\ve) + \widetilde{\delta}_n(\ve) \right] \cdot \mathbf{1}_{\mc A_n}
\]
where $\la \delta_n(M,\ve); n \geq 1 \ra$ and $\la \  \widetilde{\delta}_n(\ve); n \geq 1 \ra$ are possibly random sequences satisfying 
\begin{align*}
|\delta_n(M,\ve)| &\leq \gamma_n(M,\ve) \qquad \text{for every fixed $M,\ve$}, \\
 |\widetilde{\delta}_n(\ve)| &\leq C \ve \qquad  \text{for a universal constant C}.
\end{align*}
Here $\gamma_n(M,\ve)$ is a positive, deterministic sequence such that 
\[
\gamma_n(M,\ve) \to 0 \qquad  \text{for every fixed $M,\ve$}.
\]
Now write
\begin{align*}
    \limsup_{n \to \infty} \Big\{ \text{LHS of \eqref{eq:right-tail-L1}} \Big\}
 &\leq \limsup_{n \to \infty} \la  \mb{E} \left[ e^{u_n \sqrt{p} B_n} \cdot \mathbf{1}_{\mc A^c_n} \right] + \gamma_n(M,\ve) \cdot \mb{E} \left[ e^{u_n \sqrt{p} B_n} \cdot \mathbf{1}_{\mc A_n} \right] 
 \ra \\
 &+  \limsup_{n \to \infty} \la \,  \widetilde{\delta}_n(\ve) \cdot \mb{E} \left[ e^{u_n \sqrt{p} B_n} \cdot \mathbf{1}_{\mc A_n} \right]  \ra  \\
 & + \limsup_{n \to \infty} \la  \mb E \left[ \frac{
 \Psi_d(u_n-\sqrt p\,B_n+r_{+,n})
 }{\Psi_d(u_n)} \cdot \mathbf{1}_{\mc A^c_n}  \right] \ra.
\end{align*}
Since $\gamma_n(M,\ve) \to 0$ and $ \widetilde{\delta}_n(\ve) \leq C \ve$, we deduce that
\begin{align*}
        \limsup_{n \to \infty} \Big\{ \text{LHS of \eqref{eq:right-tail-L1}} \Big\} &\leq   \limsup_{n \to \infty} \la  \mb{E} \left[ e^{u_n \sqrt{p} B_n} \cdot \mathbf{1}_{\mc A^c_n} \right] \ra 
        +  \limsup_{n \to \infty} \la  \mb E \left[ \frac{
 \Psi_d(u_n-\sqrt p\,B_n+r_{+,n})
 }{\Psi_d(u_n)} \cdot \mathbf{1}_{\mc A^c_n}  \right] \ra \\ 
 &+ C\ve \cdot \limsup_{n \to \infty} \mb{E} \left[ e^{u_n \sqrt{p} B_n}  \right] .
\end{align*}
Furthermore, \eqref{eq:right-domination} in Lemma \ref{lem:spherical-tail} yields
\[
\frac{
 \Psi_d(u_n-\sqrt p\,B_n+r_{+,n})
 }{\Psi_d(u_n)} 
 = \frac{
 \Psi_d\left(\sqrt d\cdot \frac{t_n-B_n}{A_n}\right)
 }{\Psi_d(u_n)}
 =
 \frac{\mb P(A_nZ+B_n>t_n\,|\,T_1,T_2)}{\Psi_d(u_n)}
 \lesssim
  e^{u_n\sqrt p\,B_n}.
\]
Thus, by \eqref{eq:fvml-product-mgf-strong},
\begin{align} \label{exponential uniform integrable}
            \limsup_{n \to \infty} \Big\{ \text{LHS of \eqref{eq:right-tail-L1}} \Big\} &\lesssim   \limsup_{n \to \infty} \la  \mb{E} \left[ e^{u_n \sqrt{p} B_n} \cdot \mathbf{1}_{\mc A^c_n} \right] \ra + O \lb e^{\tau^2} \ve \rb.
\end{align}
To bound the last term, note that, for every fixed $\ve>0$,
\begin{align} \label{A_n^c negligible}
\lim_{M \to \infty} \la \limsup_{n \to \infty} \mb{P} \lb \mc A^c_n \rb  \ra = 0  
\end{align}
due to \eqref{eq:Y-plus-expansion} and 
\[
\sup_{n \geq 1} \mb{E} \lb u_n^2 \cdot p \cdot B_n^2 \rb < \infty.
\]
Furthermore, by \eqref{eq:sup-moment-product} from Proposition \ref{prop:exponential moment}, the family \( \la e^{u_n\sqrt p\,B_n}; n \geq 1 \ra\) is uniformly integrable. This fact, \eqref{exponential uniform integrable}, and \eqref{A_n^c negligible} yield \eqref{eq:right-tail-L1} after first sending $M \to \infty$ and then $\ve \to 0$.

Next, observe that
\[
\begin{aligned}
 \mb P_{\kappa_n}\left(\bm X_1^\top\bm X_2>t_n\right)
 &=
 \mb E
 \left[ \Psi_d\left(\sqrt d\cdot \frac{t_n-B_n}{A_n}\right) \right] \\
 &=
 \Psi_d(u_n)\cdot
 \mb E e^{u_n\sqrt p\,B_n}
 +o\lb \Psi_d(u_n)\rb.
\end{aligned}
\]
Since \(B=T_1T_2\), Proposition \ref{prop:exponential moment} gives $\mb E e^{u_n\sqrt p\,B_n} \to e^{2\tau^2}$.
Also, by \eqref{eq:Psi-main},
\[
 \Psi_d(u_n)
 =
 \frac{1+o(1)}{2\sqrt{2\pi}}n^{-2}e^{-x/2}.
\]
Hence, we have \eqref{eq:fvml-right}.

{\it  \underline{Step 3: Left-tail asymptotic.}} 
The left-tail asymptotic can be proved by the same argument.  Indeed, by
symmetry of \(Z\), conditionally on \(T_1,T_2\),
\[
\begin{aligned}
 \mb P_{\kappa_n}\left(\bm X_1^\top\bm X_2<-t_n\,\middle|\,T_1,T_2\right)
 &=
 \mb P(A_nZ+B_n<-t_n\,|\,T_1,T_2) =
 \mb P(A_nZ-B_n>t_n\,|\,T_1,T_2).
\end{aligned}
\]
Thus the preceding right-tail argument applies with \(B_n\) replaced by
\(-B_n\). One can employ \eqref{eq:Y-minus-expansion} and repeat Step 2 above.
\(\hfill\square\)

\subsection{Poisson approximation} \label{sec:Poisson-approx-fvml}

We are now ready to prove Theorem \ref{thm: FvML}. Without loss of generality, we may assume $\bm \mu = \bm e_1$ in what follows. Recall $I_{ij}$ from \eqref{I-W}. Let
\[
\mathcal I_n:=\la (i,j):1\le i<j\le n\ra.
\]
For $\alpha=(i,j)\in\mathcal I_n$, let
\[
N_\alpha:=\la (k,l)\in\mathcal I_n:\la i,j\ra\cap\la k,l\ra\ne\emptyset\ra.
\]
Then $|N_\alpha|\le 2n$. If $\beta\notin N_\alpha$, then the events $I_\alpha$ and $I_\beta$  are independent.  Recall $M_n$ in \eqref{S- M}.
Lemma \ref{lem:AGG} then gives
\begin{align*}
    \left| \mb{P} \lb M_n \leq t_n  \rb - \exp \lb -\frac{n(n-1)}{2} \cdot  \mb{P} \lb I_{12} \rb \rb \, \right| &\leq b_{1n} + b_{2n}
\end{align*}
where $t_n$ is as in \eqref{u-t} and 
\begin{align*}
    b_{1n} &\lesssim n^3 \cdot  \mb{P} \lb I_{12} \rb^2, \qquad b_{2n} \lesssim n^3 \cdot \mb{P} \lb I_{12} \cap I_{13} \rb.
\end{align*}
By Proposition \ref{lem:fvml-onepair}, we have 
\[
\frac{n(n-1)}{2} \cdot  \mb{P} \lb I_{12} \rb \to \frac{e^{-x/2}}{4\sqrt{2 \pi }} \cdot \lb e^{2\tau^2} + e^{-2\tau^2} \rb = \frac{e^{-x/2}}{2\sqrt{2 \pi}} \cdot \cosh \lb 2\tau^2 \rb.
\]
Thus to finish the proof, we only need to check that $b_{1n} \to 0$ and $b_{2n} \to 0$. The first statement is obvious since $\mb{P}(I_{12})=O(n^{-2})$. 

Let us show that $b_{2n} \to 0$. Define
\[
\eta_n(\bm x):= \mb P_{\kappa_n, \bm e_1} \lb |\bm x^\top \bm X| \geq t_n \rb
\]
where $\bm X$ has the FvML distribution \eqref{class fvml} with concentration $\kappa_n$ and location $\bm e_1$.

Conditioning on $\bm X_1$, we obtain
\[
n^3 \cdot \mb{P} \lb I_{12} \cap I_{13} \rb = n^3 \cdot \mb E \lb \eta_n \lb \bm{X}_1 \rb^2 \rb
\]
To finish the proof, it suffices to show that
\begin{align} \label{Poisson approx fvml}
\mb E \lb \eta_n \lb \bm{X}_1 \rb^2 \rb = O(n^{-4}).
\end{align}

{\it \underline{Step 1: Explicit representation of $\eta_n(\bm x)$}.}
We will show that
\begin{align} \label{eta-integral}
\eta_n \lb \bm{x} \rb = \intop_{|r|>t_n} \exp \lb \kappa_n x_1 r \rb  \cdot \frac{M_{p-1} \lb \kappa_n \cdot \sqrt{1- x_1^2}\cdot \sqrt{1-r^2} \rb}{M_p(\kappa_n)} \cdot f_p(r) \, dr
\end{align}
where $x_1$ is the first coordinate of $\bm x$, $M_p$ is the mgf in Lemma \ref{lem:spherical-mgf}, and $f_p$ is the density in \eqref{eq:spherical-density}. 

To show the above, observe that the FvML density can be written as
\[
 d\mb P_{\kappa_n,\bm e_1}(\bm u)
 =
 \frac{e^{\kappa_n u_1}}{M_p(\kappa_n)}\,d\sigma_{p-1}(u)
\]
where $\sigma_{p-1}:=\mbox{Unif} \lb \mb S^{p-1} \rb$.

Thus, for fixed \(\bm x\in\mb S^{p-1}\),
\begin{align} \label{eta-1}
 \eta_n(\bm x)
 =
 \frac{
 \mb E_0\left[
 \exp(\kappa_n u_1)
 \cdot \mathbf{1}_{\la |\bm x^\top \bm U|>t_n \ra }
 \right]
 }{
 M_p(\kappa_n)
 }
\end{align}
where \(\mb E_0\) denotes expectation under $\bm U=(u_1,\dots,u_p) \sim \mbox{Unif} \lb \mb S^{p-1} \rb$.

Set $\mc R:= \bm x^\top \bm U$. It is easy to check that $\mc R$ has density $f_p$.  Conditioning on \(R=r\), we may write
\[
 \bm U=r\bm x+\sqrt{1-r^2}\,\bm W,
\]
where \(\bm W\) is uniform on \(\mb S^{p-2}\subset \bm x^\perp\). Decompose
\[
 \bm e_1=x_1\bm x+\sqrt{1-x_1^2}\,\bm \xi,
\]
where \(\bm \xi\in \bm x^\perp\) is a unit vector. Then
\[
 u_1
 =
 \bm e_1^\top \bm U
 =
 x_1r+\sqrt{1-x_1^2}\sqrt{1-r^2}\cdot \bm \xi^\top \bm W.
\]
 Hence, we obtain from \eqref{eta-1} that
\begin{align*}
\eta_n(\bm x)
 &= \frac{1}{M_p(\kappa_n)} \cdot \intop_{|r|>t_n} \mb{E}_{\bm W} \left[ \exp \la \kappa_n x_1 R + \kappa_n \cdot \sqrt{1-x_1^2} \cdot \sqrt{1-R^2} \cdot \bm \xi^\top \bm W \ra \Big| \mc R=r \right] \cdot f_p(r) \, dr  \\
 &= \frac{1}{M_p(\kappa_n)} \cdot \intop_{|r|>t_n} e^{\kappa_n x_1 r} \cdot \mb{E}_{\bm W}  \exp \left[  \kappa_n \cdot \sqrt{1-x_1^2} \cdot \sqrt{1-R^2} \cdot \bm \xi^\top \bm W \Big| \mc R=r \right] \cdot f_p(r) \, dr  \\
 &=\int_{|r|>t_n}
 \exp(\kappa_n x_1r) \cdot
 \frac{
 M_{p-1}\!\left(
 \kappa_n\sqrt{1-x_1^2}\sqrt{1-r^2}
 \right)
 }{
 M_p(\kappa_n)
 }
 \cdot f_p(r)\,dr 
\end{align*}
where the last line follows from the fact that the random variable
\(\bm \xi^\top \bm W\) has the first-coordinate distribution of a uniform point on
\(\mb S^{p-2}\) (by rotational invariance).

{\it \underline{Step 2: Bounding the integral \eqref{eta-integral}.}} 
By Lemma \ref{lem:spherical-mgf},
\[
 \log M_p(\kappa_n)
 =
 \frac{\kappa_n^2}{2p}
 +
 O\left(\frac{\kappa_n^4}{p^3}\right),
\]
and
\[
 \log M_{p-1}(\kappa_n)
 =
 \frac{\kappa_n^2}{2(p-1)}
 +
 O\left(\frac{\kappa_n^4}{p^3}\right).
\]
In the present regime,
\[
 \frac{\kappa_n^4}{p^3}
 =
 O\left(\frac1{\log n}\right)
 =
 o(1),
\]
and
\[
 \kappa_n^2\left(\frac1{p-1}-\frac1p\right)
 =
 O\left(\frac{\kappa_n^2}{p^2}\right)
 =
 o(1).
\]
Therefore, for some constant $C$ depending only on $\tau$,
\begin{align}\label{eq:Mp-ratio-bound}
 \sup_{x_1,r\in[-1,1]}
 \frac{
 M_{p-1}\!\left(
 \kappa_n\sqrt{1-x_1^2}\sqrt{1-r^2}
 \right)
 }{
 M_p(\kappa_n)
 } 
 &\leq 
 \sup_{x_1,r\in[-1,1]}
 \frac{
 M_{p-1}\!\left(
 \kappa_n 
 \right)
 }{
 M_p(\kappa_n)
 }  \leq C
\end{align}
where the first inequality follows from the fact that \(M_{p-1}(a)\) is increasing in \(|a|\).

Combining \eqref{eta-integral} and \eqref{eq:Mp-ratio-bound}, we obtain
\begin{equation}\label{eq:eta-basic-bound}
 \eta_n(\bm x)
 \lesssim
 \int_{|r|>t_n} e^{\kappa_n x_1r} \cdot f_p(r)\,dr
 \le
 2\int_{t_n}^1 e^{\kappa_n|x_1|r} \cdot f_p(r)\,dr.
\end{equation}
Define
\[
 x_\star
 :=
 \min\left\{1,\frac{pt_n}{4\kappa_n}\right\}
 =
 \min\left\{1,\frac{u_n\sqrt p}{4\kappa_n}\right\}.
\]
Observe that for $r \geq t_n$ and $|x_1| \leq x_\star$,
\begin{align*}
\frac{d}{dr} \left[ \log \lb e^{\kappa_n |x_1| r} \cdot f_p(r) \rb \right] &= \frac{d}{dr} \lb \kappa_n |x_1|r + \frac{p-3}{2} \cdot \log(1-r^2)  \rb \\
&= \kappa_n |x_1| -\frac{(p-3)r}{1-r^2}\le-\frac{pt_n}{4}
\end{align*}
for all sufficiently large $n$.

Integrating the above yields
\[
\log \lb e^{\kappa_n |x_1| r} \cdot f_p(r) \rb \leq \log \lb e^{\kappa_n |x_1| t_n} \cdot f_p(t_n) \rb
- \frac{pt_n}{4}(r-t_n).
\]

Thus
\begin{align*}
\eta_n(\bm x) \cdot \mathbf{1}_{\la |x_1| \leq x_\star \ra} &\lesssim  \mathbf{1}_{\la |x_1| \leq x_\star \ra} \cdot \int_{t_n}^1 e^{\kappa_n|x_1|r} \cdot f_p(r)\,dr \\
&\leq  e^{\kappa_n |x_1| t_n} \cdot f_p(t_n) \cdot  \mathbf{1}_{\la |x_1| \leq x_\star \ra} \cdot \int_{t_n}^1 \exp \left[ - \frac{p t_n}{4} \cdot (r-t_n) \right] \, dr \\
&\lesssim e^{\kappa_n|x_1|t_n} \cdot \sqrt{p} \cdot (1-t_n^2)^{(p-3)/2} \cdot \frac{1}{pt_n} \\
&= \frac{ e^{\kappa_n|x_1|t_n}}{u_n} \cdot (1-t_n^2)^{(p-3)/2} \lesssim \frac{\exp \lb \kappa_n |x_1| t_n - \frac{u_n^2}{2} \rb}{u_n}.
\end{align*}

Consequently, with $\bm X_1=(x_1,\dots,x_p)$, we have
\begin{align*}
    \mb E \lb \eta_n \lb \bm{X}_1 \rb^2 \rb &= \mb E \lb \eta_n \lb \bm{X}_1 \rb^2 \cdot \mathbf{1}_{\la  |x_1| \leq x_\star \ra} \rb + \mb E \lb \eta_n \lb \bm{X}_1 \rb^2 \cdot \mathbf{1}_{\la  |x_1| > x_\star \ra} \rb \cdot  \\
    &\lesssim \frac{e^{-u_n^2}}{u_n^2} \cdot \mb{E}_{\kappa_n, \bm e_1} \lb e^{2\kappa_n |x_1| t_n} \rb + \mb{P}_{\kappa_n,\bm e_1} \lb |x_1|> \frac{pt_n}{4 \kappa_n}  \rb \\
    &\leq \frac{1}{n^4} \cdot \left[ \mb{E}_{\kappa_n, \bm e_1} \lb e^{2\kappa_n x_1 t_n} \rb + \mb{E}_{\kappa_n, \bm e_1} \lb e^{-2\kappa_n x_1 t_n} \rb\right]+ \mb{P}_{\kappa_n,\bm e_1} \lb |x_1|> \frac{pt_n}{4 \kappa_n}  \rb.
\end{align*}
To bound the last display, note that Proposition \ref{Herbst} yields
\begin{align*}
     \mb{E}_{\kappa_n, \bm e_1} \lb e^{2\kappa_n x_1 t_n} \rb + \mb{E}_{\kappa_n, \bm e_1} \lb e^{-2\kappa_n x_1 t_n} \rb
     &= e^{ 2\kappa_n t_n \mb{E} x_1} \cdot      \mb{E}_{\kappa_n, \bm e_1} \left[ e^{2\kappa_n \lb x_1 -\mb Ex_1 \rb t_n} \right]
     +e^{-2\kappa_n t_n \mb{E} x_1} \cdot      \mb{E}_{\kappa_n, \bm e_1} \left[ e^{-2\kappa_n \lb x_1 -\mb Ex_1 \rb t_n} \right]\\
     &= e^{2\kappa_n t_n \mb{E} x_1} \cdot \exp \left[ O \lb \frac{\kappa_n^2 t_n^2}{p} \rb \right] +  e^{-2\kappa_n t_n \mb{E} x_1} \cdot \exp \left[ O \lb \frac{\kappa_n^2 t_n^2}{p} \rb \right]
\end{align*}
Again, by Proposition \ref{Herbst}, we have $\mb E x_1 = \mb E_{\kappa_n, \bm e_1} x_1 = (\kappa_n/p)(1+o(1))$, hence
\begin{align*}
2\kappa_n t_n \mb{E} x_1 &= \frac{2 \kappa_n^2 t_n}{p} \cdot (1+o(1)) = \frac{4 \kappa_n^2 \sqrt{\log n}}{p^{3/2}} \cdot (1+o(1))
=4 \tau^2(1+o(1)) \\
\frac{\kappa_n^2 t_n^2}{p} &= O \lb \frac{\sqrt{\log n}}{\sqrt{p}} \rb =o(1).
\end{align*}
Thus
\[
\sup_{n \geq 1} \la      \mb{E}_{\kappa_n, \bm e_1} \lb e^{2\kappa_n x_1 t_n} \rb + \mb{E}_{\kappa_n, \bm e_1} \lb e^{-2\kappa_n x_1 t_n} \rb \ra < \infty.
\]
Regarding the second term, we use the sub-Gaussian tail bound (which follows from Proposition \ref{Herbst})
\begin{align*}
    \mb{P}_{\kappa_n,\bm e_1} \lb |x_1|> \frac{pt_n}{4 \kappa_n}  \rb &\leq  \mb{P}_{\kappa_n,\bm e_1} \lb \Big|x_1 - \mb{E}_{\kappa_n, \bm e_1} x_1 \Big|> \frac{pt_n}{4 \kappa_n} - \mb{E}_{\kappa_n, \bm e_1} x_1 \rb  \\
    &\leq 2\exp \left[ -\Omega(p) \cdot \lb \frac{pt_n}{4 \kappa_n} - \mb{E}_{\kappa_n, \bm e_1} x_1 \rb^2 \right].
\end{align*}
To finish the proof, note that
\begin{align*}
\frac{pt_n}{4 \kappa_n} - \mb{E}_{\kappa_n, \bm e_1} x_1 = \frac{pt_n}{4 \kappa_n} - \frac{\kappa}{p}(1+o(1)) 
&= \Omega \lb \frac{\lb \log n \rb^{3/4}}{p^{1/4}} \rb,
\end{align*}
which yields
\[
 \mb{P}_{\kappa_n,\bm e_1} \lb |x_1|> \frac{pt_n}{4 \kappa_n}  \rb \leq \exp -\lb \Omega \lb \sqrt{p} \cdot \log n \rb \rb = O(n^{-5}).
\]
Consequently, we obtain \eqref{Poisson approx fvml}. This completes the proof. $\hfill$ $\square$

\section{Proof of Theorem \ref{thm: Watson}} \label{sec:proof-watson-1}

\subsection{Preliminaries}

We start with a fact regarding the tangent--normal decomposition in the Watson model. As in \eqref{representation}, if $\bm X$ has the Watson distribution as in \eqref{class watson}, we have the representation 
 \begin{align} \label{representation2}
 \bm X  \stackrel{d}{=} \lb \widetilde{T}, \sqrt{1-\widetilde{T}^2} \cdot \bm V \rb 
 \end{align}
where $\bm V \sim \mbox{Unif} \lb \mb S^{p-2} \rb$, and $\wT $ is the cosine component supported on $[-1,1]$, independent of $\bm V$, with density
\begin{align} \label{wT}
\widetilde{f}_T(t) \propto e^{\kappa t^2} \cdot \lb 1-t^2 \rb^{(p-3)/2}.
\end{align}

Let us first collect some facts about the cosine component $\wT$.   

\begin{prop} \label{prop:local-watson-mgf}
    Suppose $\la \bm X_n; n \geq 1 \ra$ has the Watson distribution on $\mb S^{p-1}$ as in \eqref{class watson} with concentration parameter $\kappa_n$. The following statements hold in the local regime of Theorem \ref{thm: Watson}.
\begin{enumerate}[label=(\roman*)]
    \item  $\sqrt{p-2\kappa_n} \cdot \wT_n \stackrel{d}{\to} N(0,1)$, where $\wT_n$ is the cosine component of $\bm X_n$ as in \eqref{representation2}.
    \item If $\wT_1$ and $\wT_2$ are two independent copies of $\wT_n$ (the subscript $n$ is dropped for the sake of presentation), then 
  \begin{equation}\label{eq:watson-product-mgf}
 \mb E\exp\lb \pm u_n \sqrt{p} \cdot \wT_1 \wT_2 \rb \to(1-\rho^{-2})^{-1/2}.
\end{equation}
Moreover, for any fixed $\delta \in (0,\rho-1)$,
\begin{equation}\label{eq:watson-product-ui}
 \sup_n\mb E\exp\left[ (1+\delta) \cdot u_n \sqrt{p} \cdot |\wT_1 \wT_2| \right] <\infty.
\end{equation}

\end{enumerate}

\end{prop}

\noindent \textbf{Proof of Proposition \ref{prop:local-watson-mgf}.}

{\it \underline{Proof of (i)}.} To prove (i), note that the unnormalized density of $\sqrt{p-2\kappa_n} \cdot \wT_n$ has the form
\begin{align*}
& \exp \left[ \frac{\kappa_n}{2 \Delta_n} \cdot t^2 + \frac{p-3}{2}\cdot \log \lb 1 - \frac{t^2}{2 \Delta_n} \rb \right] \cdot \mathbf{1}_{\la |t| \leq \sqrt{2 \Delta_n} \ra} \\
= & \exp \left[ \frac{\kappa_n}{2\Delta_n} \cdot t^2 - \frac{p-3}{2} \cdot \frac{t^2}{2 \Delta_n}   - \frac{p-3}{16\Delta_n^2} \cdot t^4 + O \lb \frac{p}{\Delta_n^3} \rb \right] \cdot \mathbf{1}_{\la |t| \leq \sqrt{2\Delta_n} \ra} \\
= & \exp \left[ -\frac{t^2}{2} + o(1)   \right] \cdot \mathbf{1}_{\la |t| \leq \sqrt{2\Delta_n} \ra}
\end{align*}
where the $o(1)$ term is a sequence of functions that converges to zero uniformly on compact intervals.

Moreover, we have the upper bound 
\[
\exp \left[ \frac{\kappa_n}{p-2\kappa_n} \cdot t^2 + \frac{p-3}{2}\cdot \log \lb 1 - \frac{t^2}{p-2\kappa_n} \rb \right]
 \le \exp \lb -\frac{p-3-2\kappa_n}{2(p-2\kappa_n)} \cdot t^2-\frac{p-3}{4\lb p -2\kappa_n \rb^2} \cdot t^4 \rb.
\]
Therefore, the dominated convergence theorem gives (i). 

{\it \underline{Proof of (ii)}.} Observe that $\lb u_n \sqrt{p}/(p-2\kappa_n)  \rb \to \rho^{-1}$. Moreover, by part (i),
\[
\left(
\sqrt{p-2\kappa_n} \cdot \wT_1,
\sqrt{p-2\kappa_n} \cdot \wT_2
\right)
\stackrel{d}{\to} (Z_1,Z_2),
\]
where \(Z_1,Z_2\) are independent standard normal random variables. Hence
\[
u_n\sqrt p\,\wT_1\wT_2
\stackrel{d}{\to}
\rho^{-1} \cdot Z_1Z_2.
\]
Suppose \eqref{eq:watson-product-ui} has been proved. Then the preceding display, \eqref{eq:watson-product-ui}, and Vitali's theorem give
\[
\mb E\exp\left\{
\pm u_n\sqrt p\,\wT_1\wT_2
\right\}
\to
\mb E\exp\left\{
\pm \rho^{-1}Z_1Z_2
\right\} = \lb 1 - \rho^{-2} \rb^{-1/2}
\]
which is \eqref{eq:watson-product-mgf}. 

Therefore, it suffices to prove $\eqref{eq:watson-product-ui}$. First note that, for every
fixed \(a<1/2\),
\[
\sup_n
\mb E\exp\left\{
a(p-2\kappa_n)\cdot \wT_n^2
\right\}<\infty .
\]
Indeed, this follows from the density bound in the proof of (i), since the
normalizing constants of the densities of
\(\sqrt{p-2\kappa_n} \cdot \wT_n\) are bounded away from zero.

Now fix \(\delta\in(0,\rho-1)\).  Then
\[
\frac{(1+\delta)u_n\sqrt p}{p-2\kappa_n}
\to
\frac{1+\delta}{\rho}<1.
\]
Thus, for all sufficiently large \(n\), we may choose a fixed \(a<1/2\) such
that
\[
\frac{(1+\delta)u_n\sqrt p}{4\Delta_n}
\le a.
\]
Using \(|xy|\le (x^2+y^2)/2\), we get
\begin{align*}
\sup_{n \geq 1} \la \mb E\exp\left[
(1+\delta)u_n\sqrt p\cdot |\wT_1\wT_2|
\right] \ra
& \leq
\sup_{n \geq 1} \la \mb E\exp\left[
\frac{(1+\delta)u_n\sqrt p}{2}
\cdot \left(\wT_1^2+\wT_2^2\right)
\right] \ra \\
&= \sup_{n \geq 1}
\left[
\mb E\exp\left\{
\frac{(1+\delta)u_n\sqrt p}{4\Delta_n} \cdot
(p-2\kappa_n)\cdot \wT_n^2
\right\}
\right]^2
< \infty.
\end{align*}
This proves \eqref{eq:watson-product-ui}.  \(\hfill\square\)

The next result gives an exact asymptotic formula for the tail probability of the inner product under the Watson model. 

\begin{prop}\label{lem:watson-onepair}
    Suppose $\bm X_1, \bm X_2$ are i.i.d. with the Watson distribution as in \eqref{class watson}, and assume that
    \[
    \frac{p}{(\log n)^2} \to \infty; \qquad \frac{\Delta_n}{\sqrt{p \log n}} = \frac{p-2\kappa_n}{2\sqrt{p \log n}} \to \rho \in (1,\infty).
    \]
    Then, for fixed $x \in \mb R$ and $t_n=t_n(x)$ as in \eqref{u-t}, we have 
    \begin{align}
 \mb{P}_{\kappa_n,\bm \mu} \lb \bm X_1^\top \bm X_2 > t_n \rb
 &=\frac{1+o(1)}{2\sqrt{2\pi}}\cdot n^{-2}e^{-x/2} \cdot (1-\rho^{-2})^{-1/2},
 \label{eq:watson-right}\\
 \mb P_{\kappa_n,\bm \mu}\lb \bm X_1^\top \bm X_2<-t_n \rb
 &=\frac{1+o(1)}{2\sqrt{2\pi}}\cdot n^{-2}e^{-x/2} \cdot (1-\rho^{-2})^{-1/2}.
 \label{eq:watson-left}
\end{align}
\end{prop}

\noindent \textbf{Proof of Proposition \ref{lem:watson-onepair}.}
By rotation invariance, we may assume \(\bm\mu=\bm e_1\). Using
\eqref{representation2}, write
\[
 \bm X_i=\left(\wT_i,\sqrt{1-\wT_i^2}\,\bm V_i\right),
 \qquad i=1,2,
\]
where \(\bm V_1,\bm V_2\) are independent and uniform on
\(\mb S^{p-2}\), independent of \(\wT_1,\wT_2\). Again, put
\[
 d:=p-1,
 \qquad
 Z:=\bm V_1^\top \bm V_2 .
\]
Then \(Z\) has the same distribution as \(Z_d\) in \eqref{Psi}, and
\[
 \bm X_1^\top \bm X_2
 =
 \wT_1\wT_2+
 \sqrt{(1-\wT_1^2)(1-\wT_2^2)}\,Z.
\]
With a slight abuse of notation, as in the proof of Proposition \ref{lem:fvml-onepair}, put
\[
 A_n:=\sqrt{(1-\wT_1^2)(1-\wT_2^2)},
 \qquad
 B_n:=\wT_1\wT_2.
\]
Thus
\[
 \bm X_1^\top \bm X_2=A_nZ+B_n.
\]

Arguing as in the proof of Proposition \ref{lem:fvml-onepair} and using Proposition \ref{prop:local-watson-mgf}(i) and (ii), we obtain
\begin{equation}\label{eq:watson-Y-plus-expansion}
 \sqrt d\cdot \frac{t_n-B_n}{A_n}
 =
 u_n-\sqrt p\,B_n+r_{+,n},
 \qquad
 u_n|r_{+,n}|\to0
 \quad\text{in probability},
\end{equation}
and
\begin{equation}\label{eq:watson-Y-minus-expansion}
 \sqrt d\cdot \frac{t_n+B_n}{A_n}
 =
 u_n+\sqrt p\,B_n+r_{-,n},
 \qquad
 u_n|r_{-,n}|\to0
 \quad\text{in probability}.
\end{equation}
Conditioning on \(\wT_1,\wT_2\), it follows from \eqref{eq:watson-Y-plus-expansion} that
\[
\begin{aligned}
 \mb P_{\kappa_n,\bm e_1}
 \left(\bm X_1^\top\bm X_2>t_n\,\middle|\,\wT_1,\wT_2\right)
 &=
 \mb P(A_nZ+B_n>t_n\,|\,\wT_1,\wT_2)\\
 &=
 \Psi_d\left(
 \sqrt d\cdot\frac{t_n-B_n}{A_n}
 \right)\\
 &=
 \Psi_d\left(u_n-\sqrt p\,B_n+r_{+,n}\right),
\end{aligned}
\]
Arguing as in the proof of Proposition \ref{lem:fvml-onepair}, we obtain
\begin{equation*}
 \mb E
 \left|
 \frac{
 \Psi_d\left(u_n-\sqrt p\,B_n+r_{+,n}\right)
 }{\Psi_d(u_n)}
 -
 e^{u_n\sqrt p\,B_n}
 \right|
 \to0.
\end{equation*}
By Proposition \ref{prop:local-watson-mgf}(ii),
$\left\{ e^{u_n\sqrt p\,B_n}:n\ge1 \right\}$
is uniformly integrable, and
\[
 \mb E e^{u_n\sqrt p\,B_n}
 \to
 (1-\rho^{-2})^{-1/2}.
\]
It follows that
\[
\begin{aligned}
 \mb P_{\kappa_n,\bm e_1}
 \left(\bm X_1^\top\bm X_2>t_n\right)
 &=
 \mb E
 \Psi_d\left(
 \sqrt d\cdot\frac{t_n-B_n}{A_n}
 \right)\\
 &=
 \Psi_d(u_n)\,
 \mb E e^{u_n\sqrt p\,B_n}
 +o\{\Psi_d(u_n)\}.
\end{aligned}
\]
By Proposition \ref{prop:local-watson-mgf}(ii),
\[
 \mb E e^{u_n\sqrt p\,B_n}
 \to
 (1-\rho^{-2})^{-1/2}.
\]
Also, by \eqref{eq:Psi-main},
\[
 \Psi_d(u_n)
 =
 \frac{1+o(1)}{2\sqrt{2\pi}}\cdot
 n^{-2}e^{-x/2}.
\]
Therefore,
\[
 \mb P_{\kappa_n,\bm e_1}
 \left(\bm X_1^\top\bm X_2>t_n\right)
 =
 \frac{1+o(1)}{2\sqrt{2\pi}}\cdot
 n^{-2}e^{-x/2}\cdot
 (1-\rho^{-2})^{-1/2}.
\]
This proves \eqref{eq:watson-right}.
The proof of \eqref{eq:watson-left} is similar. 
\(\hfill\square\)

\subsection{Poisson approximation} \label{sec:Poisson-watson-local}

We now prove the local-regime statement of Theorem \ref{thm: Watson}.
Without loss of generality, assume \(\bm\mu=\bm e_1\).  Recall
\(\mc I_n\), \(N_\alpha\), \(I_{ij}\), and \(W\) from
\eqref{I-W} and Section \ref{sec:Poisson-approx-fvml}.  Thus
\[
 \mc I_n=\{(i,j):1\le i<j\le n\},
\]
and, for \(\alpha=(i,j)\in\mc I_n\),
\[
 N_\alpha
 =
 \{(k,l)\in\mc I_n:\{i,j\}\cap\{k,l\}\ne\emptyset\}.
\]
As before, \(|N_\alpha|\le 2n\), \(I_\alpha\) is independent of
\(\{I_\beta:\beta\notin N_\alpha\}\) and
\[
\mb{P}_{\kappa_n, \bm e_1} \lb P_n \leq x \rb = \mb{P}_{\kappa_n, \bm e_1} \lb M_n \leq t_n \rb = \mb{P}_{\kappa_n\, \bm e_1} \lb \sum_{\alpha \in \mc I_n} I_{\alpha} = 0 \rb
\] 
where the quantities $M_n, t_n$ are defined in \eqref{S- M} and \eqref{u-t}, respectively. 

The proof reduces to establishing asymptotic independence among the events $I_\alpha$, as in Section \ref{sec:Poisson-approx-fvml}. To show asymptotic independence, let us first introduce a truncation.  Put
\[
 R_n:=(\log n)^{1/2}.
\]
For \(1\le i<j\le n\), define
\[
 \mc A_{ij}
 :=
 \left\{
 (p-2\kappa_n)\cdot (\wT_i^2+\wT_j^2)\le R_n
 \right\},
\]
and decompose
\[
 I_{ij}=I_{ij}^{\circ}+I_{ij}^{\star},
 \qquad
 I_{ij}^{\circ}:=I_{ij} \cdot \mathbf 1_{\mc A_{ij}},
 \qquad
 I_{ij}^{\star}:=I_{ij} \cdot \mathbf 1_{\mc A_{ij}^{c}}.
\]
Intuitively, the main contribution comes from $I_{ij}^\circ$ and the truncated events $I_{ij}^{\star}$ are negligible (which we will verify below).

Write
\[
 W^{\circ}:=\sum_{1\le i<j\le n} I_{ij}^{\circ},
 \qquad
 W^{\star}:=\sum_{1\le i<j\le n} I_{ij}^{\star}.
\]
Then
\[
 W=W^{\circ}+W^{\star}.
\]

{\it \underline{Step 1: The truncated events $I_{ij}^\star$ are negligible}.} Conditioning on
\(\wT_1,\wT_2\), write
\[
 \bm X_1^\top\bm X_2
 =
 \wT_1\wT_2+
 \sqrt{(1-\wT_1^2)(1-\wT_2^2)}Z,
\]
where \(Z\) has the same distribution as \(Z_d\) in \eqref{Psi}, with
\(d=p-1\).  By \eqref{eq:right-domination} and
\eqref{eq:left-domination},
\[
\begin{aligned}
\mb P
 \left(I_{12}=1\,\middle|\,\wT_1,\wT_2\right) 
&\lesssim
 \Psi_d(u_n)
 \left[
 \exp\{u_n\sqrt p\,\wT_1\wT_2\}
 +
 \exp\{-u_n\sqrt p\,\wT_1\wT_2\}
 \right]  \\
&\lesssim
 \Psi_d(u_n) \cdot 
 \exp\la u_n\sqrt p\cdot |\wT_1\wT_2|\ra.
\end{aligned}
\]
Therefore
\[
\begin{aligned}
 \mb P \lb I_{12}^\star = 1 \rb
 &\lesssim
 \Psi_d(u_n) \cdot 
 \mb E\left[
 \exp\la u_n\sqrt p\cdot|\wT_1\wT_2|\ra
\cdot  \mathbf 1_{\mc A_{12}^{c}}
 \right].
\end{aligned}
\]
By Proposition \ref{prop:local-watson-mgf}(i),
\[
 (p-2\kappa_n)\wT_i^2=O_{\mb P}(1),
 \qquad i=1,2.
\]
Since \(R_n\to\infty\), we have
$\mb P(\mc A_{12}^{c})\to 0$.
Moreover, by Proposition \ref{prop:local-watson-mgf}(ii), the family
\[
 \left\{
 \exp\lb u_n\sqrt p\,|\wT_1\wT_2|\rb :n\ge1
 \right\}
\]
is uniformly integrable. Hence
\[
 \mb E\left[
 \exp\la u_n\sqrt p\cdot |\wT_1\wT_2|\ra
\cdot  \mathbf 1_{\mc A_{12}^{c}}
 \right]
=o(1).
\]
Using \eqref{eq:Psi-main}, we obtain
\[
 \mb P(I_{12}^{\star}=1)
 =
 o(n^{-2}).
\]
Thus
\[
 \mb E W^{\star}
 =
 \binom n2 \cdot \mb P(I_{12}^{\star}=1)
 =
 o(1).
\]
In particular,
\begin{equation*}
 \mb P(W^{\star} \geq 1)
 \le
 \mb E W^{\star}
 =
 o(1).
\end{equation*}
Thus
\begin{align}\label{eq:watson-Wstar-small}
\mb P \lb P_n \leq x \rb = \mb P  \lb M_n \leq t_n \rb = \mb P \lb W=0 \rb =  \mb P \lb W^\circ=0  \rb + o(1).
\end{align}

{\it \underline{Step 2: Poisson approximation to \(W^{\circ}\)}.}  We first observe that the truncated
events \(I_{ij}^{\circ}\) have the same dependency graph as the original
indicators \(I_{ij}\), since \(I_{ij}^{\circ}\) depends only on
\((\bm X_i,\bm X_j)\). Now put
\[
 \lambda_n^{\circ}:=\mb E W^{\circ}
 =
 \binom n2 \cdot 
 \mb P(I_{12}^{\circ}=1).
\]
Since
\[
 0
 \le
 \mb P(I_{12}=1)
 -
 \mb P(I_{12}^{\circ}=1)
 =
 \mb P(I_{12}^{\star}=1)
 =
 o(n^{-2}),
\]
Proposition \ref{lem:watson-onepair} gives
\[
\begin{aligned}
 \lambda_n^{\circ}
 =
 \binom n2 \cdot
 \mb P(I_{12}^{\circ}=1) 
 \to
 (8\pi)^{-1/2} \cdot e^{-x/2}\cdot (1-\rho^{-2})^{-1/2}.
\end{aligned}
\]
Applying Lemma \ref{lem:AGG} to the family \(\la I_{ij}^{\circ}:(i,j)\in\mc I_n\ra \), we obtain
\[
 \left|
 \mb P(W^{\circ}=0)
 -
 \exp\lb -\lambda_n^{\circ}\rb
 \right|
 \le
 b_{1n}^{\circ}+b_{2n}^{\circ},
\]
where
\[
 b_{1n}^{\circ}
 \lesssim
 n^3 \cdot
 \mb P(I_{12}^{\circ}=1)^2,
 \qquad
 b_{2n}^{\circ}
 \lesssim
 n^3 \cdot
 \mb P
 \left(I_{12}^{\circ}=1,I_{13}^{\circ}=1\right).
\]
The first term satisfies
\[
 b_{1n}^{\circ}
 \lesssim
  n^3
 \cdot \mb P(I_{12}=1)^2
 =
 O(n^{-1})
 =
 o(1).
\]

We now prove that \(b_{2n}^{\circ}\to0\).  For fixed
\(\bm x=\lb x_1,\sqrt{1-x_1^2} \cdot \bm v \rb \in\mb S^{p-1}\), define
\[
 \eta_n^{\circ}(\bm x)
 :=
 \mb P_{\kappa_n,\bm e_1}
 \left(
 |\bm x^\top\bm X|>t_n,\,
 (p-2\kappa_n)(x_1^2+\wT^2)\le R_n
 \right),
\]
where the probability is taken over \(\bm X=\lb \wT,\sqrt{1-\wT^2}\bm V \rb \), which has the Watson distribution with concentration $\kappa_n$, location $\bm e_1$, and the corresponding tangent--normal decomposition $\lb \wT, \bm V \rb$ as in \eqref{representation2}.

Conditioning on \(\bm X_1\), we obtain
\[
 \mb P
 \left(I_{12}^{\circ}=1,I_{13}^{\circ}=1\right)
 =
 \mb E_{\kappa_n,\bm e_1}
 \left[
 \eta_n^{\circ}(\bm X_1)^2
 \right].
\]
We claim that there exists a constant $C_\rho>0$, depending only on $\rho$, such that, uniformly over \(\bm x\in\mb S^{p-1}\),
\begin{equation}\label{eq:watson-eta-truncated-bound}
 \eta_n^{\circ}(\bm x)
 \lesssim
 \Psi_d(u_n)\cdot \exp\lb C_\rho \cdot R_n\rb.
\end{equation}
Indeed, if
\[
 (p-2\kappa_n)x_1^2>R_n,
\]
then \(\eta_n^{\circ}(\bm x)=0\). Otherwise, conditioning on \(\wT\),
we may write
\[
 \bm x^\top \bm X
 =
 x_1\wT+
 \sqrt{1-x_1^2}\sqrt{1-\wT^2}\cdot Z
\]
where \(Z\) has the same distribution as \(Z_d\).  By
\eqref{eq:right-domination} and \eqref{eq:left-domination},
\[
\begin{aligned}
\eta_n^{\circ}(\bm x)
&\lesssim
\Psi_d(u_n) \cdot 
\mb E_{\wT}\left[
\left(
e^{u_n\sqrt p\,x_1\wT}
+
e^{-u_n\sqrt p\,x_1\wT}
\right) \cdot 
\mathbf 1_{\{(p-2\kappa_n)(x_1^2+\wT^2)\le R_n\}}
\right]  \\
&\lesssim
\Psi_d(u_n) \cdot 
\exp\left\{
u_n\sqrt p\cdot 
\sup_{\la (p-2\kappa_n)(x_1^2+y^2)\le R_n\ra}
 |x_1y|
\right\}.
\end{aligned}
\]
Since
$2|x_1y|\le x_1^2+y^2$,
the supremum above is bounded by
$
 \frac{R_n}{2(p-2\kappa_n)}.
$
Therefore
\[
 u_n\sqrt p\cdot 
 \sup_{\la (p-2\kappa_n)(x_1^2+y^2)\le R_n\ra}
 |x_1y|
 \le
 \frac{u_n\sqrt p}{2(p-2\kappa_n)}\cdot R_n
 \le
 C_\rho \cdot R_n,
\]
because
$(u_n\sqrt p)/(p-2\kappa_n) \to \rho^{-1}.
$
This proves \eqref{eq:watson-eta-truncated-bound}.

Consequently, by \eqref{eq:Psi-main},
\[
\begin{aligned}
\mb P \left(I_{12}^{\circ}=1,I_{13}^{\circ}=1\right)
&=
\mb E_{\kappa_n,\bm e_1}
\left[
\eta_n^{\circ}(\bm X_1)^2
\right]  \\
&\lesssim
\Psi_d(u_n)^2 \cdot \exp\lb C_\rho \cdot  R_n \rb \\
&=
n^{-4+o(1)},
\end{aligned}
\]
since \(R_n=(\log n)^{1/2}=o(\log n)\).  Hence
\[
 b_{2n}^{\circ}
 \lesssim
 n^3 \cdot 
 \mb P
 \left(I_{12}^{\circ}=1,I_{13}^{\circ}=1\right)
 =
 n^{-1+o(1)}
 =
 o(1).
\]

Putting everything together,
\begin{align*}
    & \limsup_{n \to \infty} \left|   \mb P\lb W^{\circ}=0 \rb - \exp\left\{
  -(8\pi)^{-1/2}\cdot e^{-x/2} \cdot (1-\rho^{-2})^{-1/2}
  \right\}  \right| \\
  \leq &\limsup_{n \to \infty} \left|    \mb P\lb W^{\circ}=0 \rb -  \exp\lb -\lambda_n^{\circ}\rb \right| +
  \limsup_{n \to \infty} \left|  \exp\lb -\lambda_n^{\circ}\rb - \exp\left\{
  -(8\pi)^{-1/2} \cdot e^{-x/2} \cdot (1-\rho^{-2})^{-1/2}
  \right\} \right| \\
  = & 0.
\end{align*}
Combining the preceding result with \eqref{eq:watson-Wstar-small} completes the proof of Theorem \ref{thm: Watson}.
\(\hfill\square\)

\section{Proof of Theorem \ref{thm:watson-rho<1}} \label{sec:proof-watson-2}

\subsection{Laplace method and tail probability asymptotic}

\begin{lemma}\label{lem:watson-coordinate-md}
Assume that
\[
    \rho_n:=\frac{\Delta_n}{\sqrt{p\log n}}\to \rho\in(0,\infty),
    \qquad
    \frac{p}{(\log n)^3}\to\infty .
\]
Recall the density $\widetilde{f}_T$ in \eqref{wT} and define
\[
g_n(\alpha):= \wf_T\left(
 \alpha\cdot \frac{(\log n)^{1/4}}{p^{1/4}}
 \right) \cdot 
 \frac{(\log n)^{1/4}}{p^{1/4}}.
\]
Then, uniformly for \(\alpha\) in compact subsets of \(\mb R\),
\[
 g_n(\alpha)
 =
 \frac{\sqrt{\rho_n}}{\sqrt\pi}\cdot (\log n)^{1/2} \cdot 
 \exp\left\{
 -(\log n)\left(\rho_n\alpha^2+\frac{\alpha^4}{4}\right)
 \right\}
 \lb 1+o(1)\rb.
\]
Moreover, for all $1/100>\ve>0$, there exists $C_{\ve,\rho}>0$ depending only on $\rho$ and $\ve$ such that
    \begin{equation}\label{eq:watson-md-crude-bound}
    g_n(\alpha)
    \le
    C_{\ve,\rho} \sqrt{\log n} \cdot 
    \exp\la -(1-\ve) \log n \lb \rho_n \alpha^2+ \frac{\alpha^4}{4} \rb\ra \cdot
    \mathbf 1_{\la |\alpha|<(p/\log n)^{1/4}\ra} .
    \end{equation}
\end{lemma}

\noindent \textbf{Proof of Lemma \ref{lem:watson-coordinate-md}.}
By \eqref{wT}, the density of \(\wT\) is
\[
 \wf_T(t)
 =
 \frac{1}{\mathcal Z_n} \cdot 
 \exp\lb \kappa_n t^2\rb \cdot (1-t^2)^{(p-3)/2}
 \mathbf 1_{\{|t|<1\}},
\]
where
\[
 \mathcal Z_n
 =
 \int_{-1}^{1}
 \exp\lb \kappa_n z^2\rb \cdot (1-z^2)^{(p-3)/2}\,dz .
\]
Put
\[
 t=\alpha\frac{(\log n)^{1/4}}{p^{1/4}}.
\]
In what follows, we assume $\alpha \in [-M,M]$ for some $M \in (0,\infty)$.
 
 {\it \underline{Step 1: Analyzing the unnormalized density.}} Fix \(M<\infty\). Uniformly for \(|\alpha|\le M\), we have
\[
 \log(1-t^2)
 =
 -t^2-\frac{t^4}{2}+O(t^6).
\]
Therefore
\[
\begin{aligned}
\kappa_n t^2+\frac{p-3}{2}\log(1-t^2) &=
\left(\frac p2-\Delta_n\right)t^2
+\frac{p-3}{2}
\left(
-t^2-\frac{t^4}{2}+O(t^6)
\right)\\
&=
-\Delta_n t^2
+\frac32 t^2
-\frac{p-3}{4}t^4
+O(pt^6).
\end{aligned}
\]
Now
\[
 \Delta_n t^2
 =
 \rho_n\sqrt{p\log n}\cdot
 \alpha^2\sqrt{\frac{\log n}{p}}
 =
 \rho_n\alpha^2\log n,
\]
and
\[
 \frac{p-3}{4}t^4
 =
 \frac{\alpha^4}{4}\log n+O\left(\frac{\log n}{p}\right).
\]
The remaining terms are negligible uniformly for \(|\alpha|\le M\), because
\[
 t^2=O_M\left(\sqrt{\frac{\log n}{p}}\right)=o(1),
\]
and
\[
 pt^6
 =
 O_M\left(
 \frac{(\log n)^{3/2}}{\sqrt p}
 \right)
 =
 O_M\left[
 \left(\frac{(\log n)^3}{p}\right)^{1/2}
 \right]
 =
 o(1).
\]
Thus
\begin{equation}\label{eq:watson-md-unnormalized}
 \kappa_n t^2+\frac{p-3}{2}\log(1-t^2)
 =
 -(\log n)
 \left(
 \rho_n\alpha^2+\frac{\alpha^4}{4}
 \right)
 +o(1),
\end{equation}
uniformly for \(|\alpha|\le M\).

We next derive an upper bound on the unnormalized density. Using the elementary inequality
\[
 \log(1-y)\le -y-\frac{y^2}{2},
 \qquad 0\le y<1,
\]
we obtain, for all \(|t|<1\),
\[
\begin{aligned}
\kappa_n t^2+\frac{p-3}{2}\log(1-t^2)
&\le
\left(\frac p2-\Delta_n\right)t^2
-\frac{p-3}{2}t^2
-\frac{p-3}{4}t^4  \\
&=
-\Delta_n t^2
+\frac32 t^2
-\frac{p-3}{4}t^4 .
\end{aligned}
\]
With $t= \alpha (\log n)^{1/4}/p^{1/4}$ for $|\alpha|<\lb p/\log n \rb^{1/4}$, we obtain

\begin{align}
\kappa_n t^2+\frac{p-3}{2}\log(1-t^2)
&\le
-\rho_n\alpha^2\log n
-\frac{\alpha^4}{4}\log n
+\frac32\alpha^2\sqrt{\frac{\log n}{p}}
+\frac34\alpha^4\frac{\log n}{p} \nonumber \\
&\leq -(1-\ve) \cdot \log n \cdot
\left(
\rho_n\alpha^2+\frac{\alpha^4}{4}
\right). \label{numerator}
\end{align}
for all sufficiently small $\ve>0$ and all sufficiently large $n$, since $\rho_n \to \rho \in (0,\infty)$.

{\it \underline{Step 2: Bounding the normalizing constant \(\mathcal Z_n\)}.}  Let us make the change of variables
\[
 z=\frac{h}{p^{1/4}(\log n)^{1/4}} .
\]
Then
\[
\begin{aligned}
\mathcal Z_n
&=
\frac{1}{p^{1/4}(\log n)^{1/4}}
\int_{|h|<p^{1/4}(\log n)^{1/4}}
\exp\left\{
\kappa_n \cdot \frac{h^2}{\sqrt{p\log n}}
+
\frac{p-3}{2}
\log\left(1-\frac{h^2}{\sqrt{p\log n}}\right)
\right\}\,dh .
\end{aligned}
\]
For fixed \(h\), the exponent equals
\[
\begin{aligned}
&\left(\frac p2-\Delta_n\right)
\frac{h^2}{\sqrt{p\log n}}
+
\frac{p-3}{2}
\log\left(1-\frac{h^2}{\sqrt{p\log n}}\right)\\
&\qquad =
-\Delta_n\frac{h^2}{\sqrt{p\log n}}
+\frac32\frac{h^2}{\sqrt{p\log n}}
-\frac{p-3}{4}\frac{h^4}{p\log n}
+O\left(
p\frac{|h|^6}{(p\log n)^{3/2}}
\right)\\
&\qquad =
-\rho_nh^2+o(1).
\end{aligned}
\]
Moreover, using \(\log(1-y)\le -y\) for \(0\le y<1\), we have
\begin{align*}
  \kappa_n \cdot \frac{h^2}{\sqrt{p\log n}}
+
\frac{p-3}{2}
\log\left(1-\frac{h^2}{\sqrt{p\log n}}\right) 
\leq & \left(\frac p2-\Delta_n\right)\frac{h^2}{\sqrt{p\log n}}-\frac{p-3}{2} \frac{h^2}{\sqrt{p\log n}}
\\
= &-\left(\Delta_n-\frac32\right)\frac{h^2}{\sqrt{p\log n}} \leq -\frac{\rho}{2}h^2
\end{align*}
for all sufficiently large \(n\), uniformly in \(h\). Hence the dominated convergence theorem gives
\[
\begin{aligned}
\mathcal Z_n
&=
\frac{1+o(1)}{p^{1/4}(\log n)^{1/4}}
\int_{-\infty}^{\infty}e^{-\rho_nh^2}\,dh
=
\frac{1+o(1)}{p^{1/4}(\log n)^{1/4}}
\sqrt{\frac{\pi}{\rho_n}}.
\end{aligned}
\]
Combining this normalization with \eqref{eq:watson-md-unnormalized}, we obtain the first claim. To finish the proof of the second claim, let us derive a lower bound on $\mc Z_n$.
  Recall that
\[
 \mathcal Z_n
 =
 \int_{-1}^{1}
 \exp\lb \kappa_n z^2\rb \cdot (1-z^2)^{(p-3)/2}\,dz .
\]
Restricting the integral above to $|z|\le (p\log n)^{-1/4}$ and using the estimate  $\log(1-y)\ge -y-y^2$
for all sufficiently small \(y\ge0\), we obtain
\[
\begin{aligned}
\kappa_n z^2+\frac{p-3}{2}\log(1-z^2)
&\ge
\left(\frac p2-\Delta_n\right)z^2
-\frac{p-3}{2}(z^2+z^4)  \\
&=
-\Delta_n z^2
+\frac32 z^2
-\frac{p-3}{2}z^4 .
\end{aligned}
\]
Since \(|z|\le (p\log n)^{-1/4}\),
\[
 \Delta_n z^2
 \lesssim
 \sqrt{p\log n}\cdot (p\log n)^{-1/2}
 =
 O(1); \qquad  pz^4\lesssim \frac1{\log n}=o(1).
\]
Therefore, on the interval \(|z|\le (p\log n)^{-1/4}\), the exponent is
bounded below by a constant.  Consequently,
\begin{equation*}
 \mathcal Z_n
 \ge
 c(p\log n)^{-1/4}
\end{equation*}
for some constant \(c>0\). Together with \eqref{numerator}, the preceding display yields \eqref{eq:watson-md-crude-bound}. This completes the proof.
\(\hfill\square\)

We will also use the following uniform version of Laplace's method. We first describe the setting. We are interested in the asymptotic behavior of the sequence of integrals
\[
 \int_{B(\bm x_n,r)}
h_n(\bm x)
\exp\{-L_n \cdot H_n(\bm x)\}\,d\bm x.
\]
Here $L=L_n$ is a sequence such that $L_n \to \infty$, and $B(\bm x_n, r) \subset \mb R^2$ is a sequence of open balls centered at $\bm x_n \in \mb R^2$ with fixed radius $r>0$. We assume that \(H_n\) is three times continuously differentiable on
\(B(\bm x _n,r)\) and \(h_n\) is positive on \(B(\bm x_n,r)\). Assume further that the following regularity conditions hold:
\begin{enumerate}[label=(\roman*)]

\item $ \nabla H_n(\bm x_n)=0.$

\item The eigenvalues of $\bm  \Sigma_n:=\nabla^2H_n(\bm x_n)$ are bounded away from $0$ and uniformly bounded above, uniformly in $n$, and
\[
H_n(\bm x) \geq H_n \lb \bm x_n \rb + c\| \bm x - \bm x_n \|^2, \qquad \bm x \in B \lb \bm x_n,r \rb
\]
for some fixed constants $c,r>0$.

\item As $n \to \infty$, 
\[
\sup_{\| \bm x\|\le \log L_n}  \left|
    \frac{
    h_n(\bm x_n+  \frac{\bm x}{\sqrt{L_n}})
    }{
    h_n(\bm x_n)
    }
    -1
    \right|
    \to0.
    \]
    \item For some constant \(C>0\),
    \[
    \sup_{\bm x\in B(\bm x_n,r)}
    \frac{h_n(\bm x)}{h_n(\bm x_n)}
    \le L_n^C 
    \]
    and
    \[
    \sup_{\| \bm y \| \leq \log L_n} \left| L_n \left[ H_n \lb \bm x_n + \frac{\bm y}{\sqrt{L_n}} \rb - H_n \lb \bm x_n \rb \right]  - \frac{\bm y^\top \bm \Sigma_n \bm y}{2} \right| \to 0.
    \]
\end{enumerate}


\begin{lemma}
\label{lem:local-laplace}
If conditions $(i)-(iv)$ above hold, then
\[
\begin{aligned}
 \int_{B(\bm x_n,r)}
h_n(\bm x)
\exp\Big[ -L_n \cdot H_n(\bm x)\Big] \,d\bm x 
& =
\frac{2\pi}{L_n}
\cdot
\frac{
h_n(\bm x_n)
}{
\det(\bm \Sigma_n)^{1/2} 
}
\cdot
\exp\Big[ -L_n\cdot H_n(\bm x_n)\Big]
\cdot \lb  1+o(1)\rb
\end{aligned}
\]
as $n \to \infty$.
\end{lemma}

\noindent \textbf{Proof of Lemma \ref{lem:local-laplace}.}
Write \(m_n:=H_n(\bm x_n)\) and \(R_n:=\log L_n\). Since \(R_n/\sqrt{L_n}\to0\), the ball \(B(\bm x_n,R_n/\sqrt{L_n})\) is contained in \(B(\bm x_n,r)\) for all sufficiently large \(n\).
By conditions (iii) and (iv), uniformly for \(\|\bm y\|\le R_n\),
\[
\begin{aligned}
&
h_n\left(\bm x_n+\frac{\bm y}{\sqrt{L_n}}\right)
\exp\left\{
-L_n\left[
H_n\left(\bm x_n+\frac{\bm y}{\sqrt{L_n}}\right)
-m_n
\right]
\right\}
=
h_n(\bm x_n)
\exp\left\{
-\frac12\bm y^\top\bm\Sigma_n\bm y
\right\}
\{1+o(1)\}.
\end{aligned}
\]
Using the change of variables
\[
\bm x=\bm x_n+\frac{\bm y}{\sqrt{L_n}},
\qquad
d\bm x=L_n^{-1}\,d\bm y,
\]
we obtain
\[
\begin{aligned}
&
\int_{\|\bm x-\bm x_n\|\le R_n/\sqrt{L_n}}
h_n(\bm x) \cdot e^{-L_nH_n(\bm x)}\,d\bm x
=
\frac{h_n(\bm x_n)e^{-L_nm_n}}{L_n}
\cdot 
\int_{\|\bm y\|\le R_n}
\exp\left\{
-\frac12\bm y^\top\bm\Sigma_n\bm y
\right\}\,d\bm y
\cdot 
\{1+o(1)\}.
\end{aligned}
\]
By condition (ii),
\[
\begin{aligned}
0
&\le
\frac{2\pi}{\det(\bm\Sigma_n)^{1/2}}
-
\int_{\|\bm y\|\le R_n}
e^{-\bm y^\top\bm\Sigma_n\bm y/2}\,d\bm y
\le
\int_{\|\bm y\|>R_n}
\exp \la -\lambda_{\rm min} \lb \bm \Sigma_n \rb \cdot \frac{\|\bm y\|^2}{2} \ra \,d\bm y
=o(1).
\end{aligned}
\]
Since \(\det(\bm\Sigma_n)\) is uniformly bounded above and bounded
away from zero,
\[
\int_{\|\bm y\|\le R_n}
e^{-\bm y^\top\bm\Sigma_n\bm y/2}\,d\bm y
=
\frac{2\pi}{\det(\bm\Sigma_n)^{1/2}}
\{1+o(1)\}.
\]
It follows that
\[
\begin{aligned}
&
\int_{\|\bm x-\bm x_n\|\le R_n/\sqrt{L_n}}
h_n(\bm x)e^{-L_nH_n(\bm x)}\,d\bm x
=
\frac{2\pi}{L_n} \cdot 
\frac{h_n(\bm x_n)}
{\det(\bm\Sigma_n)^{1/2}}
e^{-L_nm_n}
\cdot 
\{1+o(1)\}.
\end{aligned}
\]
It remains to control the annular region. By conditions (ii) and (iv),
\[
\begin{aligned}
\int_{\{R_n/\sqrt{L_n}<\|\bm x-\bm x_n\|<r\}}
h_n(\bm x) \cdot e^{-L_nH_n(\bm x)}\,d\bm x
&\le
h_n(\bm x_n) \cdot L_n^C \cdot e^{-L_nm_n} \cdot 
\int_{\|\bm x-\bm x_n\|>R_n/\sqrt{L_n}}
e^{-cL_n\|\bm x-\bm x_n\|^2}\,d\bm x
\\
&\le
\frac{h_n(\bm x_n)e^{-L_nm_n}}{L_n} \cdot 
L_n^C
\cdot 
\int_{\|\bm y\|>R_n}
e^{-c\|\bm y\|^2}\,d\bm y.
\end{aligned}
\]
Since \(R_n=\log L_n\),
\[
L_n^C \cdot 
\int_{\|\bm y\|>\log L_n}
e^{-c\|\bm y\|^2}\,d\bm y
=o(1).
\]
Hence the annular contribution is
\[
o\left(
\frac{h_n(\bm x_n)e^{-L_nm_n}}{L_n}
\right),
\]
This completes the proof.
\(\hfill\square\)

\begin{prop} \label{prop:watson-ldp-rho<1}
    Assume the conditions of Theorem \ref{thm:watson-rho<1}. For fixed $x \in \mb R$, define
    \[
    s^\circ_n=s_n^\circ (x) := a \lb \rho_n \rb^2 \cdot \log n - \frac{a \lb \rho_n \rb}{A \lb \rho_n \rb} \cdot  \log \log n +x; \qquad t^\circ_n=t^\circ_n(x):= \sqrt{\frac{s^\circ_n(x)}{p}}.
    \]
    Then
    \[
    \mb{P}_{\kappa_n, \bm \mu_n} \lb \left| \bm X_1^\top \bm X_2  \right| \geq t^\circ_n \rb = \frac{2K(\rho)}{n^2} \cdot \exp \lb - \frac{A(\rho)}{2a(\rho)} \cdot x \rb \cdot (1+o(1))
    \]
    as $n \to \infty$.
\end{prop}

\noindent \textbf{Proof of Proposition \ref{prop:watson-ldp-rho<1}.}
By rotation invariance, we may assume \(\bm\mu_n=\bm e_1\).  
Using \eqref{representation2}, we may assume
\[
 \bm X_1^\top\bm X_2
 =
 \wT_1\wT_2
 +
 \sqrt{1-\wT_1^2}\sqrt{1-\wT_2^2}\cdot Z,
\]
where \(Z=\bm V_1^\top\bm V_2\), and \(\bm V_1,\bm V_2\) are independent
uniform random variables on \(\mb S^{p-2}\), independent of
\(\wT_1,\wT_2\).  

 It suffices to prove the right-tail asymptotic
\begin{equation}\label{eq:watson-rho-less-one-right-goal}
\mb P_{\kappa_n,\bm e_1}
\left(
 \bm X_1^\top\bm X_2>t_n^\circ
\right)
=
\frac{K(\rho)+o(1)}{n^2}
\exp\left\{
-\frac{A(\rho)}{2a(\rho)}x
\right\}.
\end{equation}
Put \[d:=p-1 \qquad \text{and} \qquad  L=L_n:=\log n.\] 
Also define
\[
A_n:= A \lb \rho_n \rb, \qquad a_n:=a \lb \rho_n \rb
\]
Make the change of variables
\[
 \wT_i=\alpha_i \cdot \frac{L^{1/4}}{p^{1/4}},
 \qquad i=1,2.
\]
Then
\[
\begin{aligned}
\mb P_{\kappa_n,\bm e_1}
\left(
 \bm X_1^\top\bm X_2>t_n^\circ
\right) 
&=
\iint_{\mb R^2}
 \Pi_n(\alpha_1,\alpha_2) \cdot
 g_n(\alpha_1) \cdot g_n(\alpha_2)
 \,d\alpha_1d\alpha_2,
\end{aligned}
\]
where $g_n$ is defined in Lemma \ref{lem:watson-coordinate-md}, and
\begin{align} \label{Pi}
\Pi_n(\alpha_1,\alpha_2)
:=
\mb P_{\kappa_n,\bm e_1}
\left(
 \bm X_1^\top\bm X_2>t_n^\circ
 \,\middle|\,
 \wT_1=\alpha_1 \cdot \frac{L^{1/4}}{p^{1/4}},\, \wT_2=\alpha_2 \cdot \frac{L^{1/4}}{p^{1/4}}
\right).
\end{align}

{\it \underline{Step 1: Extracting the main contribution}.} Fix a sufficiently large \(M<\infty\) and a sufficiently small \(\delta>0\). Set
\[
 \mc D_{M,\delta}^{(n)}
 :=
 \left\{
 |\alpha_1|\le M,\ |\alpha_2|\le M,\ 
 a_n-\alpha_1\alpha_2\ge \delta
 \right\}.
\]
We will show that
\begin{align} \label{main-contribution}
\iint_{\mb R^2}
 \Pi_n(\alpha_1,\alpha_2)
 g_n(\alpha_1)g_n(\alpha_2)
 \,d\alpha_1d\alpha_2 = \iint_{\mc D_{M,\delta}^{(n)}}
 \Pi_n(\alpha_1,\alpha_2)
 g_n(\alpha_1)g_n(\alpha_2)
 \,d\alpha_1d\alpha_2 +o \lb n^{-2} \rb.
\end{align}
To show this, it suffices to prove 
\begin{align} \label{tail-contribution-1}
    \iint_{\{|\alpha_1|>M\}\cup\{|\alpha_2|>M\}}
 \Pi_n(\alpha_1,\alpha_2)
 g_n(\alpha_1)g_n(\alpha_2)
 \,d\alpha_1d\alpha_2
 &=
 o(n^{-2})
\end{align}
and
\begin{align} \label{tail-contribution-2}
\iint_{\{a_n-\alpha_1\alpha_2<\delta\}}
 \Pi_n(\alpha_1,\alpha_2)
 g_n(\alpha_1)g_n(\alpha_2)
 \,d\alpha_1d\alpha_2  = o(n^{-2}).
\end{align}
Let us start with \eqref{tail-contribution-1}.  Since
\(0\le\Pi_n(\alpha_1,\alpha_2)\le1\), \eqref{eq:watson-md-crude-bound} in  Lemma
\ref{lem:watson-coordinate-md} gives
\begin{align*}
& \iint_{\{|\alpha_1|>M\}\cup\{|\alpha_2|>M\}}
 \Pi_n(\alpha_1,\alpha_2)
 g_n(\alpha_1)g_n(\alpha_2)
 \,d\alpha_1d\alpha_2 \\
 \leq & 2 \int_{|\alpha_1|>M} g_n(\alpha_1)\, d\alpha_1 \\
 \lesssim & \sqrt{L} \cdot  \int_{|\alpha_1|>M} \exp \lb -cL \lb \alpha_1^2 + \alpha_1^4  \rb \rb \, d \alpha_1 \\
 \leq & \sqrt{L} \cdot \exp \left[ \frac{-cL}{2} \cdot (M^2+M^4) \right] \cdot \intop_{\mb R} \exp \left[ \frac{-cL}{2} \cdot (\alpha_1^2+ \alpha_1^4) \right] \, d\alpha_1.
\end{align*} 
The last term is of order $o\lb n^{-2} \rb$ for all sufficiently large $M$. This proves \eqref{tail-contribution-1}.

We now prove \eqref{tail-contribution-2}. Again, the fact that \(0\le\Pi_n(\alpha_1,\alpha_2)\le1\) and Lemma \ref{lem:watson-coordinate-md} give
\begin{align*}
\iint_{\{a_n-\alpha_1\alpha_2<\delta\}}
 \Pi_n(\alpha_1,\alpha_2)
 g_n(\alpha_1)g_n(\alpha_2)
 \,d\alpha_1d\alpha_2 
 \lesssim_{\ve}
 L \cdot 
& \iint_{\{a_n-\alpha_1\alpha_2<\delta\}}
\exp\left\{
-(1-\ve)L \cdot \left[
\rho_n(\alpha_1^2+\alpha_2^2)
+\frac{\alpha_1^4+\alpha_2^4}{4}
\right]
\right\}
\,d\alpha_1d\alpha_2
\end{align*}
for all sufficiently small $\ve>0$.

Choose $\delta$ sufficiently small that $a_n>\delta$ for all sufficiently large $n$. Observe that when $a_n - \alpha_1 \alpha_2 < \delta$,
\[
\begin{aligned}
\rho_n(\alpha_1^2+\alpha_2^2)
+\frac{\alpha_1^4+\alpha_2^4}{4}
&\ge 2\rho_n \alpha_1 \alpha_2 + \frac{\alpha_1^2 \alpha_2^2}{2} \geq 2\rho_n(a_n-\delta)+\frac{(a_n-\delta)^2}{2}.
\end{aligned}
\]
Letting \(n\to\infty\) and then \(\delta\downarrow0\), the last term
converges to
\[
2\rho a(\rho)+\frac{a(\rho)^2}{2}
=
4+2\rho^2
>2.
\]
Lemma \ref{lem:strong-convex} then gives
\[
\iint_{\{a_n-\alpha_1\alpha_2<\delta\}}
 \Pi_n(\alpha_1,\alpha_2)
 g_n(\alpha_1)g_n(\alpha_2)
 \,d\alpha_1d\alpha_2 = O_{\ve,\rho} \lb L\cdot e^{-(2+\zeta)\cdot L} \rb = o \lb n^{-2} \rb
\]
for some $\zeta>0$ and all sufficiently small $\delta$.
This completes the proof of \eqref{tail-contribution-2}.

{\it \underline{Step 2: Exponential representation}.} We claim that, uniformly on $\mc D^{(n)}_{M,\delta}$,
\begin{align} \label{exponential form}
\Pi_n(\alpha_1,\alpha_2)g_n(\alpha_1)g_n(\alpha_2)
=
\mc K_n(\alpha_1,\alpha_2)\lb 1+o(1)\rb,
\end{align}
where
\begin{align*}
\mc K_n(\alpha_1,\alpha_2)
&:=
 \frac{\rho_n}{\pi} \cdot 
 \frac{L}
 {(a_n-\alpha_1\alpha_2)\sqrt{2\pi L}}  
\times
 \exp\left\{
 -L\left[
 \rho_n(\alpha_1^2+\alpha_2^2)
 +\frac{\alpha_1^4+\alpha_2^4}{4}
 +\frac12(a_n-\alpha_1\alpha_2)^2
 \right]
 \right\} \\
&\quad\times
 L^{(a_n-\alpha_1\alpha_2)/(2A_n)}
 \cdot \exp\left\{
 -\frac{a_n-\alpha_1\alpha_2}{2a_n}x
 \right\}.
\end{align*}
Indeed, conditionally on \(\wT_1,\wT_2\), the event
\( \la \bm X_1^\top\bm X_2>t_n^\circ \ra\) is equivalent to
\[
\sqrt{d} \cdot Z>
\sqrt{d} \cdot \frac{
t_n^\circ-\wT_1\wT_2
}{
\sqrt{1-\wT_1^2}\sqrt{1-\wT_2^2}
}.
\]
A direct computation gives
\[
\sqrt{s_n^\circ}
=
a_n\sqrt L
-
\frac{
\frac{a_n}{A_n}\log L-x
}{
2a_n\sqrt L
}
+o(L^{-1/2}).
\]
Moreover, uniformly for \(|\alpha_1|,|\alpha_2|\le M\),
\[
\sqrt{1-\wT_1^2}\sqrt{1-\wT_2^2}
=
1+o\lb L^{-1} \rb
\]
because \(p/L^3\to\infty\).  Therefore, uniformly on
\(\mc D_{M,\delta}^{(n)}\),
\[
\begin{aligned}
\sqrt d\cdot 
\frac{
t_n^\circ-\wT_1\wT_2
}{
\sqrt{1-\wT_1^2}\cdot \sqrt{1-\wT_2^2}
}
&=
(a_n-\alpha_1\alpha_2)\sqrt L  -
\frac{
\frac{a_n}{A_n}\log L-x
}{
2a_n\sqrt L
}
+o\lb L^{-1/2} \rb.
\end{aligned}
\]
Consequently, since $a_n-\alpha_1\alpha_2 \geq \delta$, the self-normalized moderate-deviation result (see Theorem 2.3 in \cite{jing2003self}) yields
\[
\begin{aligned}
\Pi_n(\alpha_1,\alpha_2)
&=
 \frac{1+o(1)}
 {(a_n-\alpha_1\alpha_2)\sqrt{2\pi L}} \cdot 
 \exp\left\{
 -\frac L2(a_n-\alpha_1\alpha_2)^2
 +
 \frac{a_n-\alpha_1\alpha_2}{2a_n}
 \left(
 \frac{a_n}{A_n}\log L-x
 \right)
 \right\}.
\end{aligned}
\]
uniformly on $\mc D^{(n)}_{M,\delta}$. We then obtain \eqref{exponential form} using the first claim in Lemma \ref{lem:watson-coordinate-md}.

{\it  \underline{Step 3: Laplace's method}.} Define
\begin{align*}
    H_n(\alpha_1,\alpha_2):&=  \rho_n(\alpha_1^2+\alpha_2^2)
+\frac{\alpha_1^4+\alpha_2^4}{4}
+\frac12(a_n-\alpha_1\alpha_2)^2, \\
h_n(\alpha_1,\alpha_2):&=   \frac{\rho_n}{\pi} \cdot 
 \frac{L}
 {(a_n-\alpha_1\alpha_2)\sqrt{2\pi L}} \cdot  L^{(a_n-\alpha_1\alpha_2)/(2A_n)}
 \cdot \exp\left\{
 -\frac{a_n-\alpha_1\alpha_2}{2a_n}x
 \right\}.
\end{align*}
A direct calculation shows that \(H_n\) has exactly two global minimizers,
\[
\bm \alpha^{\star}_n := \left(\sqrt{A_n-2\rho_n},\sqrt{A_n-2\rho_n}\right),
 \qquad
 \bm \alpha^{\star\star}_n:=\left(-\sqrt{A_n-2\rho_n},-\sqrt{A_n-2\rho_n}\right),
\]
and at both points,
\[
H_n=2,
\qquad
a_n-\alpha_1\alpha_2=A_n.
\]
Choosing $\delta$ sufficiently small that \(\delta<A(\rho)/2\) ensures that both minimizers lie in the cutoff
region 
$$ \la a_n-\alpha_1\alpha_2\ge\delta\ra $$
for all sufficiently large \(n\).

At either minimizer, the Hessian matrix satisfies
\[
\nabla^2H_n
=
\begin{pmatrix}
4A_n-6\rho_n & -2\rho_n\\
-2\rho_n & 4A_n-6\rho_n
\end{pmatrix},
\qquad
\det(\nabla^2H_n)
=
16(A_n-\rho_n)(A_n-2\rho_n).
\]
We apply Lemma \ref{lem:local-laplace} and show that the main contribution to the double integral involving $\mc K_n $ comes from the two balls around $\bm \alpha^{\star}_n$ and $\bm \alpha^{\star \star}_n$ (with appropriate radius). 

Let us verify that there exists \(r>0\), sufficiently small, such that the
balls
\[
    B_n^\star:=B(\bm\alpha_n^\star,r),
    \qquad
    B_n^{\star\star}:=B(\bm\alpha_n^{\star\star},r)
\]
are disjoint and contained in $\mc D^{(n)}_{M,\delta}$ for all sufficiently large $n$.

Write $\bm \alpha^{\star}_n = \lb \alpha^\star_{1n}, \alpha^\star_{2n} \rb$ and $\bm \alpha^{\star \star}_n = \lb \alpha^{\star \star}_{1n}, \alpha^{\star \star}_{2n} \rb$. Observe that 
\[
\|  \bm \alpha^{\star}_n - \bm \alpha^{\star \star}_n \| =  2\sqrt{2A_n-4\rho_n}; \qquad a_n- \alpha^\star_{1n} \alpha^\star_{2n} = a_n - \alpha^{\star \star}_{1n} \alpha^{\star \star}_{2n} = A_n.
\]
Since $A_n \to A(\rho)>0$ for $\rho \in (0,1)$ and $A_n-2\rho_n \to A(\rho)-2\rho>0$, such an $r$ exists. Now write
\begin{align*}
    \iintop_{\mc D_{M,\delta}^{(n)}} \mc K_n(\alpha_1,\alpha_2) \,d\alpha_1d\alpha_2
    &=  \iintop_{B_n^\star} \mc K_n(\alpha_1,\alpha_2) \,d\alpha_1d\alpha_2
    + \iintop_{B_n^{\star \star}} \mc K_n(\alpha_1,\alpha_2) \,d\alpha_1d\alpha_2 \\
    &+ \iint_{\mc D_{M,\delta}^{(n)}
\setminus
(B_n^\star\cup B_n^{\star\star})}
\mc K_n(\alpha_1,\alpha_2)
\,d\alpha_1d\alpha_2.
\end{align*}
It is straightforward, though tedious, to check that conditions (i)--(iv) in Lemma \ref{lem:local-laplace} apply to $H_n$ and $h_n$. A direct computation yields
\[
\begin{aligned}
h_n(\bm\alpha_n^\star)
&=
\frac{\rho_n}{\pi} \cdot 
\frac{L}
{A_n\sqrt{2\pi L}}
\cdot 
L^{A_n/(2A_n)}
\cdot 
\exp\left\{
-\frac{A_n}{2a_n}x
\right\} \\
&=
\frac{\rho_n}{\pi} \cdot
\frac{L}
{A_n\sqrt{2\pi}} \cdot
\exp\left\{
-\frac{A_n}{2a_n}x
\right\}.
\end{aligned}
\]
Therefore,
\[
\begin{aligned}
\iintop_{B_n^\star}
\mc K_n(\alpha_1,\alpha_2)\,d\alpha_1d\alpha_2
&= \iintop_{B_n^\star} h_n(\alpha_1, \alpha_2 )\cdot \exp \lb -L \cdot H_n(\alpha_1, \alpha_2) \rb \, d\alpha_1 d\alpha_2 \\
&=
\frac{2\pi}{L}
\cdot
\frac{
h_n(\bm\alpha_n^\star)
}{
\det\{\nabla^2H_n(\bm\alpha_n^\star)\}^{1/2}
}
\cdot 
e^{-2L}
\cdot \lb 1+o(1)\rb \\
&=
\frac{2\pi}{L}
\cdot
\frac{
\frac{\rho_n}{\pi} \cdot
\frac{L}{A_n\sqrt{2\pi}} \cdot
\exp\left\{-\frac{A_n}{2a_n}x\right\}
}{
4\sqrt{(A_n-\rho_n)(A_n-2\rho_n)}
}
\cdot n^{-2}
\cdot \lb 1+o(1)\rb \\
&= 
\frac{\rho_n}
{2\sqrt{2\pi}\,A_n
\sqrt{(A_n-\rho_n)(A_n-2\rho_n)}}
\cdot n^{-2}
\cdot \exp\left\{-\frac{A_n}{2a_n}x\right\}
\cdot \lb 1+o(1)\rb \\
&= 
\frac{K(\rho_n)}{2n^2} \cdot \exp\left\{-\frac{A_n}{2a_n}x\right\}
\cdot \lb 1+o(1)\rb
\end{aligned}
\]
where $K$ is defined in \eqref{K rho}. Similarly,
\[
\begin{aligned}
\iint_{B_n^{\star\star}}
\mc K_n(\alpha_1,\alpha_2)\,d\alpha_1d\alpha_2
&=
\frac{\rho_n}
{2\sqrt{2\pi}\,A_n
\sqrt{(A_n-\rho_n)(A_n-2\rho_n)}}
\cdot n^{-2} \cdot
\exp\left\{-\frac{A_n}{2a_n}x\right\}
\cdot \lb 1+o(1)\rb \\
&= \frac{K(\rho_n)}{2n^2} \cdot \exp\left\{-\frac{A_n}{2a_n}x\right\}
\cdot \lb 1+o(1)\rb.
\end{aligned}
\]
To finish the proof, it suffices to show that
\[
\begin{aligned}
& \iint_{\mc D_{M,\delta}^{(n)}
\setminus
(B_n^\star\cup B_n^{\star\star})}
\mc K_n(\alpha_1,\alpha_2)
\,d\alpha_1d\alpha_2 =
o(n^{-2}).
\end{aligned}
\]
To prove this, note that $H_n$ is a smooth function that attains its minimum value $2$ at $\bm \alpha^{\star}_n$ and $\bm \alpha^{\star \star}_n$, so away from the two balls $B_n^\star, B_n^{\star \star}$, we must have 
\[
\inf_{\mc D_{M,\delta}^{(n)}
\setminus
(B_n^\star\cup B_n^{\star\star})} H_n(\alpha_1,\alpha_2) \geq 2 +  \zeta_1
\]
for some $\zeta_1>0$ independent of $n$. 

Applying Lemma \ref{lem:strong-convex} gives
\[
\begin{aligned}
 \iintop_{\mc D_{M,\delta}^{(n)}
\setminus
(B_n^\star\cup B_n^{\star\star})}
\mc K_n(\alpha_1,\alpha_2)
\,d\alpha_1d\alpha_2 
&= (1+o(1)) \cdot \iintop_{\mc D_{M,\delta}^{(n)}
\setminus
(B_n^\star\cup B_n^{\star\star})} \Pi_n(\alpha_1,\alpha_2)g_n(\alpha_1)g_n(\alpha_2) \, d\alpha_1 d\alpha_2  \\
&\lesssim
\frac{\mbox{poly}(L)}{\exp \left[ (2+ \frac{\zeta_1}{2})L\right]}
=
o(n^{-2}).
\end{aligned}
\]
Finally, putting everything together and ignoring all the $o \lb n^{-2} \rb$ terms, we deduce the desired asymptotic result. This completes the proof. \( \hfill \square \)

\subsection{Poisson approximation} \label{sec:Poisson-approximation-rho<1}

We now prove Theorem \ref{thm:watson-rho<1}.  Without loss of generality,
assume \(\bm\mu_n=\bm e_1\).  Fix \(x\in\mb R\), and let
\(s_n^\circ=s_n^\circ(x)\) and \(t_n^\circ=t_n^\circ(x)\) be as in
Proposition \ref{prop:watson-ldp-rho<1}.  Throughout this subsection, put
\[
    L=L_n:=\log n,\qquad A_n:=A(\rho_n),\qquad a_n:=a(\rho_n),
    \qquad \zeta_n:=\sqrt{A_n-2\rho_n}.
\]
Here $A(\rho)$ and $a(\rho)$ are defined in \eqref{A-a}. By \eqref{representation2}, we can write
\[
 \bm X_i = \lb \wT_i, \sqrt{1- \wT_i^2} \cdot \bm V_i \rb
\]
where $\bm V_i \sim \mbox{Unif} \lb \mb S^{p-2} \rb$, and the $\wT_i$ are i.i.d. with the density in 
\eqref{wT} and are independent of the $\bm V_i$.

For each \(i\), write
\[
    \wT_i=\alpha_i \cdot \frac{L^{1/4}}{p^{1/4}}.
\]
Note that the $\alpha_i$ are i.i.d. with density $g_n(.)$ given in Lemma \ref{lem:watson-coordinate-md}. 

Define the non-local exceedance indicators (hence the abbreviation ``nl") by
\[
    I_{ij}^{\mathrm{nl}}
    :=
    \mathbf 1_{\la |S_{ij}|>t_n^\circ\ra},
    \qquad
    W^{\mathrm{nl}}
    :=
    \sum_{1\le i<j\le n}I_{ij}^{\mathrm{nl}}.
\]
Then
\[
\left\{
pM_n^2-a_n^2L+\frac{a_n}{A_n}\log L\le x
\right\}
=
\{M_n\le t_n^\circ\}
=
\{W^{\mathrm{nl}}=0\}.
\]
We now introduce a truncation for the Poisson approximation to create independence.  Put
\[
    I_{ij}^{+}:=\mathbf 1_{\la S_{ij}>t_n^\circ\ra},
    \qquad
    I_{ij}^{-}:=\mathbf 1_{\la S_{ij}<-t_n^\circ\ra}.
\]
Thus 
$$I_{ij}^{\mathrm{nl}}=I_{ij}^{+}+I_{ij}^{-}.$$ 
Choose \(r_0>0\) sufficiently small, to be specified below, and define 
\[
\mc B_n^{+}
:=
B\left(  \bm \alpha^\star_n ,r_0\right)
\cup
B\left(\bm \alpha^{\star \star}_n,r_0\right)
\]
where $\bm \alpha^\star_n=\lb \zeta_n, \zeta_n \rb$ and $\bm \alpha^{\star \star}_n= \lb -\zeta_n, -\zeta_n \rb$ are the saddle points defined in the proof of Proposition \ref{prop:watson-ldp-rho<1}.

We further decompose the indicators $I^+_{ij}$ and $I^-_{ij}$ as 
\begin{align*}
    I_{ij}^{+}:=  I_{ij}^{+,\mathrm{tr}} +  I_{ij}^{+,\mathrm{rem}},
    \qquad 
    I_{ij}^{-}:=  I_{ij}^{-,\mathrm{tr}} +  I_{ij}^{-,\mathrm{rem}}
\end{align*}
where
\begin{align*}
  &  I_{ij}^{+,\mathrm{tr}}
    :=
    I_{ij}^{+} \cdot \mathbf 1_{\la (\alpha_i,\alpha_j)\notin \mc B_n^{+}\ra},
    \qquad
    I_{ij}^{+,\mathrm{rem}}
    :=
    I_{ij}^{+} \cdot \mathbf 1_{\la (\alpha_i,\alpha_j) \in\mc B_n^{+}\ra}, \\
    & I_{ij}^{-,\mathrm{tr}}
    :=
    I_{ij}^{-} \cdot \mathbf 1_{\la (\alpha_i,-\alpha_j)\notin\mc B_n^{+}\ra},
    \qquad
    I_{ij}^{-,\mathrm{rem}}
    :=
    I_{ij}^{-} \cdot \mathbf 1_{\la (\alpha_i,-\alpha_j)\in\mc B_n^{+}\ra}.
\end{align*}
Here the abbreviation ``tr" stands for ``truncated," and ``rem" stands for ``remainder." Put
\[
I_{ij}^{\rm tr}:= I_{ij}^{\rm +,tr} + I_{ij}^{\rm -,tr}; \qquad I_{ij}^{\rm rem}:= I_{ij}^{\rm +,rem} + I_{ij}^{\rm -,rem}
\]
With this notation, we have 
\[
    W^{\mathrm{nl}}=W^{\mathrm{tr}}+W^{\mathrm{rem}}, \qquad I_{ij}^{\rm nl} = I_{ij}^{\rm tr} + I_{ij}^{\rm rem},
\]
where
\[
    W^{\mathrm{tr}}
    :=
    \sum_{1\le i<j\le n}I_{ij}^{\mathrm{tr}},
    \qquad
    W^{\mathrm{rem}}
    :=
    \sum_{1\le i<j\le n}I_{ij}^{\mathrm{rem}}.
\]

{\it \underline{Step 1: The truncated sum $W^{\rm tr}$ is negligible}.}
We first note that
\[
    \mb P_{\kappa_n,\bm e_1}\lb I_{12}^{+,\mathrm{tr}}=1\rb = o \lb n^{-2} \rb
\]
for every sufficiently small $r_0$.

Indeed, for a sufficiently large $M>0$, Step 1 in the proof of Proposition
\ref{prop:watson-ldp-rho<1} gives
\begin{align*}
    \mb P_{\kappa_n,\bm e_1}\lb I_{12}^{+,\mathrm{tr}}=1\rb &= \iintop_{\mc D^{(n)}_{M,\delta} \cap \la \lb \alpha_1, \alpha_2 \rb  \notin \mc B_n^{+} \ra } \Pi_n \lb \alpha_1, \alpha_2 \rb g_n(\alpha_1) g_n(\alpha_2) \, d \alpha_1 d \alpha_2 + o \lb n^{-2} \rb.
\end{align*}
The double integral above is also of order $o \lb n^{-2} \rb$, as shown in Step 3 of the proof of Proposition
\ref{prop:watson-ldp-rho<1}'s proof.

By the symmetry of the Watson distribution,
\[
    \mb P_{\kappa_n,\bm e_1} \lb I_{12}^{-,\mathrm{tr}}=1 \rb
    =
    \mb P_{\kappa_n,\bm e_1}\lb I_{12}^{+,\mathrm{tr}}=1 \rb 
    =
    o \lb n^{-2} \rb.
\]
Therefore
\begin{align} \label{truncation negligible}
    \mb P_{\kappa_n,\bm e_1}\lb I_{12}^{\mathrm{tr}}=1 \rb=o \lb n^{-2} \rb.
\end{align}
Consequently,
\[
    \mb E_{\kappa_n,\bm e_1}W^{\mathrm{tr}}
    =
    \binom n2 \cdot 
    \mb P_{\kappa_n,\bm e_1} \lb I_{12}^{\mathrm{tr}}=1 \rb
    =
    o(1),
\]
and hence
\begin{equation}\label{eq:watson-nonlocal-Wrem}
    \mb P_{\kappa_n,\bm e_1} \lb W^{\mathrm{tr}}>0 \rb
    \le
    \mb E_{\kappa_n,\bm e_1}W^{\mathrm{tr}}
    =
    o(1).
\end{equation}
It follows that
\begin{equation}\label{eq:watson-nonlocal-Wtr-Wnl}
\left|
\mb P_{\kappa_n,\bm e_1}\lb W^{\mathrm{nl}}=0 \rb
-
\mb P_{\kappa_n,\bm e_1}\lb W^{\mathrm{rem}}=0 \rb
\right|
\le
\mb P_{\kappa_n,\bm e_1}\lb W^{\mathrm{tr}}>0 \rb
=o(1).
\end{equation}

{\it \underline{Step 2: Poisson approximation}.}  We now apply Lemma \ref{lem:AGG} to the remainder events
\[
    \la I_{ij}^{\mathrm{rem}}:1\le i<j\le n \ra.
\]
As before, let $\mc I_n:=\{(i,j):1\le i<j\le n\}$.
For \(\gamma=(i,j)\in\mc I_n\), let
\[
N_\gamma
:=
\{(k,l)\in\mc I_n:\{i,j\}\cap\{k,l\}\ne\emptyset\}.
\]
Then \(|N_\gamma|\le 2n\), and \(I_\gamma^{\mathrm{rem}}\) is independent of
\(\{I_\beta^{\mathrm{rem}}:\beta\notin N_\gamma\}\).  Put
\[
    \lambda_n^{\mathrm{rem}}
    :=
    \mb E_{\kappa_n,\bm e_1}W^{\mathrm{rem}}
    =
    \binom n2 \cdot
    \mb P_{\kappa_n,\bm e_1} \lb  I_{12}^{\mathrm{rem}}=1 \rb.
\]
Lemma \ref{lem:AGG} then gives
\[
\left|
\mb P_{\kappa_n,\bm e_1} \lb W^{\mathrm{rem}}=0 \rb
-
\exp\lb -\lambda_n^{\mathrm{rem}} \rb
\right|
\le
b_{1n}^{\mathrm{rem}}+b_{2n}^{\mathrm{rem}},
\]
where
\[
    b_{1n}^{\mathrm{rem}}
    \lesssim
    n^3 \cdot \mb P_{\kappa_n,\bm e_1} \lb I_{12}^{\mathrm{rem}}=1 \rb^2, 
    \qquad  
    b_{2n}^{\mathrm{rem}}
    \lesssim
    n^3 \cdot 
    \mb P_{\kappa_n,\bm e_1}
    \left( I_{12}^{\mathrm{rem}}=1,I_{13}^{\mathrm{rem}}=1\right).
\]
Note that, by \eqref{truncation negligible},
\[
0\leq
\mb P_{\kappa_n,\bm e_1} \lb  I_{12}^{\mathrm{nl}}=1 \rb
-
\mb P_{\kappa_n,\bm e_1} \lb I_{12}^{\mathrm{rem}}=1 \rb
=
\mb P_{\kappa_n,\bm e_1} \lb I_{12}^{\mathrm{tr}}=1 \rb
=
o(n^{-2}).
\]
Thus, the preceding display, together with Proposition \ref{prop:watson-ldp-rho<1}, implies
\[
\begin{aligned}
\lambda_n^{\mathrm{rem}}
=
\binom n2 \cdot 
\mb P_{\kappa_n,\bm e_1}\lb I_{12}^{\mathrm{rem}}=1 \rb 
\to
K(\rho) \cdot
\exp\left\{
-\frac{A(\rho)}{2a(\rho)} \cdot x
\right\}.
\end{aligned}
\]
Moreover, 
\begin{align*}
b_{1n}^{\mathrm{rem}}
&\lesssim   n^3 \cdot \mb P_{\kappa_n,\bm e_1} \lb I_{12}^{\mathrm{rem}}=1 \rb^2  
\lesssim
n^3 \cdot 
\mb P_{\kappa_n,\bm e_1}(I_{12}^{\mathrm{nl}}=1)^2
=
O(n^{-1})
=
o(1).
\end{align*}
It remains to show that \(b_{2n}^{\mathrm{rem}}\to0\).  By symmetry,
\[
\begin{aligned}
\mb P_{\kappa_n,\bm e_1}
\left(I_{12}^{\mathrm{rem}}=1,I_{13}^{\mathrm{rem}}=1\right)
&=
4\cdot 
\mb P_{\kappa_n,\bm e_1}
\left(I_{12}^{+,\mathrm{rem}}=1,I_{13}^{+,\mathrm{rem}}=1\right).
\end{aligned}
\]
Thus it suffices to prove
\begin{equation}\label{eq:watson-nonlocal-right-shared}
\mb P_{\kappa_n,\bm e_1}
\left(I_{12}^{+,\mathrm{rem}}=1,I_{13}^{+,\mathrm{rem}}=1\right)
=
o(n^{-3}).
\end{equation}

{\it \underline{Step 3: Proving \eqref{eq:watson-nonlocal-right-shared}}.} For a fixed $\bm x  \in \mb S^{p-1}$, write 
\[
\bm x = \lb x_1 , \sqrt{1-x_1^2} \cdot \bm v_x \rb.
\]
As in the proof of Theorem \ref{thm: Watson}, conditioning on $\bm X_1$ gives
\[
\mb P_{\kappa_n,\bm e_1}
\left(I_{12}^{+,\mathrm{rem}}=1,I_{13}^{+,\mathrm{rem}}=1\right)
=
\mb E_{\kappa_n,\bm e_1}
\left[
\eta_n^{\mathrm{rem}}(\bm X_1)^2 \right]
\]
where
\begin{align*}
\eta_n^{\mathrm{rem}}(\bm x)
:&=
\mb P_{\kappa_n,\bm e_1}
\left(
\bm x^\top\bm X_1 >t_n^\circ,\,
\lb \frac{x_1 \cdot p^{1/4}}{L^{1/4}}, \frac{\wT_1 \cdot p^{1/4}}{L^{1/4}} \rb \in\mc B_n^{+}
\right) \\
&= \mb P_{\kappa_n,\bm e_1}
\left(
\bm x^\top\bm X_1 >t_n^\circ,\,
\lb \frac{x_1 \cdot p^{1/4}}{L^{1/4}}, \alpha_1 \rb \in\mc B_n^{+}
\right)
\end{align*}
For a fixed $\bm x \in \mb S^{p-1}$, with $\widetilde{x}_1:= \lb x_1 \cdot p^{1/4} \rb/L^{1/4}$, we can write
\begin{align*}
    \eta_n^{\mathrm{rem}}(\bm x) &= \mb P_{\lb \wT_1, \bm V_1 \rb} \lb \bm v_x^\top \bm V_1 > \frac{t_n^\circ - x_1 \wT_1}{\sqrt{1-x_1^2}\cdot \sqrt{1-\wT_1^2}}, \, \lb \widetilde{x}_1, \alpha_1 \rb \in \mc B_n^{+}  \rb  \\
    &=  \intop_{\la (\widetilde{x}_1,\alpha_1)\in\mc B_n^+\ra}
\Psi_{p-1}\left(
\sqrt{p-1}\cdot
\frac{
t_n^\circ-
\widetilde{x}_1 \alpha_1 \cdot \sqrt{L/p}
}{
\sqrt{1-\widetilde{x}_1^2\cdot \sqrt{L/p}}\cdot \sqrt{1-\alpha_1^2\sqrt{L/p}}
}
\right)
\cdot g_n(\alpha_1)\,d\alpha_1 
\end{align*}
where $g_n$ is the density in Lemma \ref{lem:watson-coordinate-md} and $\Psi_{p-1}$ is defined in \eqref{Psi} with $d=p-1$.

From the integral representation of $\eta_n^{\mathrm{rem}}(\bm x)$, we see that, on the set $\la (\widetilde{x}_1,\alpha_1)\in\mc B_n^+\ra$, it is necessary that
\[
\widetilde{x}_1 \alpha_1 = \zeta_n^2 + O(r_0).
\]
The display above and the relation $a_n- A_n =  \zeta_n^2$ yield
\[
a_n - \widetilde{x}_1 \alpha_1 \geq A_n - Cr_0
\]
for a constant $C>0$ depending only on $\rho$.

Now, observe that
\[
    \sqrt{p}\,t_n^\circ
    =
    \sqrt{s_n^\circ}
    =
    a_n\sqrt L
    +O\left(\frac{\log L}{\sqrt L}\right).
\]
Thus, uniformly on \( \lb \widetilde{x}_1,\alpha_1 \rb \in\mc B_n^+\),
\[
\sqrt{p-1}\,
\frac{
t_n^\circ-
\widetilde{x}_1\alpha_1 \cdot \sqrt{L/p}
}{
\sqrt{1-\widetilde{x}_1^2 \sqrt{L/p}} \cdot \sqrt{1-\alpha_1^2\sqrt{L/p}}
}
\ge
\lb A_n-Cr_0 \rb \sqrt L
+
O\left(\frac{\log L}{\sqrt L}\right).
\]
Choosing $r_0$ sufficiently small so that $A_n-Cr_0$ is bounded away from $0$ gives the tail bound
\[
\Psi_{p-1}\left(
\sqrt{p-1}\cdot
\frac{
t_n^\circ-
\widetilde{x}_1\alpha_1 \cdot \sqrt{L/p}
}{
\sqrt{1-\widetilde{x}_1^2 \sqrt{L/p}} \cdot \sqrt{1-\alpha_1^2\sqrt{L/p}}
}
\right)
\le
\mbox{poly}(L) \cdot 
\exp\left\{
-\frac L2 \lb A_n-Cr_0\rb^2
\right\}.
\]
Consequently,
\[
\begin{aligned}
\eta_n^{\mathrm{rem}}(\bm x)
&\le
\mbox{poly}(L) \cdot  
\exp\left\{
-\frac L2 \lb A_n-Cr_0 \rb^2
\right\}
\cdot 
\intop_{|\alpha_1-\zeta_n|\le r_0\ \text{or}\ |\alpha_1+\zeta_n|\le r_0}
g_n(\alpha_1)\,d\alpha_1.
\end{aligned}
\]
By Lemma \ref{lem:watson-coordinate-md}, for every sufficiently small fixed \(\ve>0\),
\[
g_n(\alpha_1)
\le
C_\ve L^{1/2}
\exp\left\{
-(1-\ve)L\left(\rho_n\alpha_1^2+\frac{\alpha_1^4}{4}\right)
\right\}.
\]
Moreover, on the above integration region,
\[
\rho_n\alpha_1^2+\frac{\alpha_1^4}{4}
\ge
\rho_n\zeta_n^2+\frac{\zeta_n^4}{4}
-Cr_0
\]
for some constant $C$ independent of $n$ (but possibly different from the previous one).

Consequently,
\[
\eta_n^{\mathrm{rem}}(\bm x)
\le
\mbox{poly}(L) \cdot 
\exp\left\{
-L\left[
\rho_n\zeta_n^2+\frac{\zeta_n^4}{4}
+\frac{A_n^2}{2}
-Cr_0-C\ve
\right]
\right\}
\]
uniformly in $\bm x$, whenever
$$
\widetilde{x}_1 \in
[\zeta_n-r_0,\zeta_n+r_0]\cup[-\zeta_n-r_0,-\zeta_n+r_0],
$$
and \(\eta_n^{\mathrm{rem}}(\bm x)=0\) otherwise.

Putting everything together, since
\(\eta_n^{\mathrm{rem}}(\bm x)=0\) unless
\[
\widetilde{x}_1\in
[\zeta_n-r_0,\zeta_n+r_0]
\cup
[-\zeta_n-r_0,-\zeta_n+r_0],
\]
the preceding pointwise bound yields
\begin{align*}
\mb E_{\kappa_n,\bm e_1}
\left[
\eta_n^{\mathrm{rem}}(\bm X_1)^2
\right]
&\le
\mb P_{\kappa_n,\bm e_1}
\left(
\frac{p^{1/4}X_{11}}{L^{1/4}}
\in
[\zeta_n-r_0,\zeta_n+r_0]
\cup
[-\zeta_n-r_0,-\zeta_n+r_0]
\right)\\
&\quad\times
\operatorname{poly}(L)
\exp\left\{
-L\left[
2\rho_n\zeta_n^2+\frac{\zeta_n^4}{2}
+A_n^2-Cr_0-C\ve
\right]
\right\}.
\end{align*}
Moreover, by Lemma~\ref{lem:watson-coordinate-md} and the bound
\[
\rho_n\alpha^2+\frac{\alpha^4}{4}
\ge
\rho_n\zeta_n^2+\frac{\zeta_n^4}{4}-Cr_0
\]
whenever
\(
|\alpha-\zeta_n|\le r_0
\)
or
\(
|\alpha+\zeta_n|\le r_0,
\)
we have
\[
\mb P_{\kappa_n,\bm e_1}
\left(
\frac{p^{1/4}X_{11}}{L^{1/4}}
\in
[\zeta_n-r_0,\zeta_n+r_0]
\cup
[-\zeta_n-r_0,-\zeta_n+r_0]
\right)
\le
\operatorname{poly}(L)
\exp\left\{
-L\left[
\rho_n\zeta_n^2+\frac{\zeta_n^4}{4}
-Cr_0-C\ve
\right]
\right\}.
\]
Consequently,
\[
\mb E_{\kappa_n,\bm e_1}
\left[
\eta_n^{\mathrm{rem}}(\bm X_1)^2
\right]
\le
\operatorname{poly}(L)
\exp\left\{
-L\left[
3\left(
\rho_n\zeta_n^2+\frac{\zeta_n^4}{4}
\right)
+A_n^2-Cr_0-C\ve
\right]
\right\}.
\]
Furthermore,
\begin{align} \label{truncation-advantage}
3\left(\rho_n\zeta_n^2+\frac{\zeta_n^4}{4}\right)+A_n^2
\to
3\cdot\frac{1-\rho^2}{2}
+
2(1+\rho^2)
=
\frac{7+\rho^2}{2} > 3.
\end{align}
Choosing $r_0$ and $\ve$ sufficiently small yields \eqref{eq:watson-nonlocal-right-shared}. This completes the proof. $\hfill$ $\square$

\begin{remark}
    The truncation introduced in Section \ref{sec:Poisson-approximation-rho<1} above removes the contribution outside a small neighborhood of the saddle points. The benefit of the truncation is quantified by \eqref{truncation-advantage}: the expectation decays like $\mbox{polylog}(n) \cdot n^{-3.5-\rho^2/2}$. Without the truncation, we only have the first term, which is asymptotically
    
    $$ \frac 32 \lb 1 - \rho^2 \rb< 3.$$
    This is therefore insufficient to show \eqref{eq:watson-nonlocal-right-shared}.
\end{remark}

\section{Proof of Theorem \ref{thm:Watson-critical}} \label{sec:proof-watson-3}

\subsection{Tail-probability asymptotics}

We begin with a tail-probability asymptotic in the case $\rho=1$.

\begin{prop} \label{prop:large-deviation-critical}
    Assume the conditions in Theorem \ref{thm:Watson-critical} and put
    \[
    s_n^{\circ \circ} := 4 \log n - \frac 12 \log \log n + x, \qquad t_n^{\circ \circ}:= \sqrt{\frac{s_n^{\circ \circ}}{p}}.
    \]
    Then 
    \begin{align} \label{eq:critical-right-tail-goal}
    \mb{P}_{\kappa_n, \bm \mu_n} \lb \left| \bm X_1^\top \bm X_2  \right| \geq t^{\circ \circ}_n \rb = 2K_{\rm cr} \lb \Delta \rb \cdot \frac{e^{-x/2}}{n^2} \cdot (1+o(1))
    \end{align}
    where the quantities $K_{\rm cr}, \Delta$ are defined in Theorem \ref{thm:Watson-critical}.
\end{prop}

\noindent \textbf{Proof of Proposition \ref{prop:large-deviation-critical}.}
By rotation invariance, we may assume \(\bm\mu_n=\bm e_1\). Put
\[
    L:=\log n.
\]
By \eqref{representation2}, we may write
\[
 \bm X_i
 =
 \left(\wT_i,\sqrt{1-\wT_i^2}\bm V_i\right),
 \qquad i=1,2,
\]
where \(\bm V_1,\bm V_2\) are independent and uniform on \(\mb S^{p-2}\),
independent of \(\wT_1,\wT_2\).  Hence
\[
 \bm X_1^\top\bm X_2
 =
 \wT_1\wT_2
 +
 \sqrt{1-\wT_1^2}\sqrt{1-\wT_2^2}\cdot Z_{p-1},
\]
where \(Z_{p-1}\) is defined in \eqref{Psi}.

As in the proof of Proposition \ref{prop:watson-ldp-rho<1}, define 
\[
    \wT_i=\alpha_i \cdot \frac{L^{1/4}}{p^{1/4}},
    \qquad i=1,2.
\]
Decompose
\begin{align*}
\mb{P}_{\kappa_n, \bm \mu_n} \lb  \bm X_1^\top \bm X_2 \geq t^{\circ \circ}_n \rb=& \iint_{\mb R^2}
\wPi_n(\alpha_1,\alpha_2)g_n(\alpha_1)g_n(\alpha_2)
\,d\alpha_1d\alpha_2 \\
= &
\iint_{ \mc E_{n,M}}
\wPi_n(\alpha_1,\alpha_2)g_n(\alpha_1)g_n(\alpha_2)
\,d\alpha_1d\alpha_2  
+ 
\iint_{\mc E_{n,M}^c}
\wPi_n(\alpha_1,\alpha_2)g_n(\alpha_1)g_n(\alpha_2)
\,d\alpha_1d\alpha_2 
\end{align*}
where
\begin{align} \label{tilde Pi}
\wPi_n(\alpha_1,\alpha_2)  :&= \mb P_{\kappa_n,\bm e_1}
\left(
 \bm X_1^\top\bm X_2>t_n^{\circ \circ}
 \,\middle|\,
 \wT_1=\alpha_1 \cdot \frac{L^{1/4}}{p^{1/4}},\, \wT_2=\alpha_2 \cdot \frac{L^{1/4}}{p^{1/4}} \right)  \\
\mc E_{n,M} :&= \la \lb \alpha_1, \alpha_2 \rb \in \mb{R}^2: \left| \alpha_1 + \alpha_2 \right| \leq  \frac{M}{L^{1/4}};\, \left| \alpha_1 - \alpha_2 \right| \leq \frac{M}{\sqrt{L}}  \ra \nonumber
\end{align}
for some fixed $M>0$.

To finish the proof, it suffices to prove 
\begin{align} \label{mc E contribution}
 & \iint_{ \mc E_{n,M}}
\wPi_n(\alpha_1,\alpha_2)g_n(\alpha_1)g_n(\alpha_2)
\,d\alpha_1d\alpha_2 \nonumber \\
=& \frac{e^{-x/2}}{n^2}\cdot 
\frac{1+o(1)}{2\pi\sqrt{2\pi}} \cdot
\int_{-M/\sqrt{2}}^{M/\sqrt{2}}\int_{-M/\sqrt{2}}^{M/\sqrt{2}}
\exp\left\{
-\Delta y^2-2z^2-\frac{y^4}{4}
\right\}
\,dz\,dy .
\end{align} 
and 
\begin{align} \label{mc E complement contribution}
  \lim_{M \to \infty} \limsup_{n \to \infty} \la n^2 \cdot   \iint_{\mc E_{n,M}^c}
\wPi_n(\alpha_1,\alpha_2)g_n(\alpha_1)g_n(\alpha_2)
\,d\alpha_1d\alpha_2  \ra = 0.
\end{align}

{\it \underline{Proof of \eqref{mc E contribution}}.}   
On $\mc E_{n,M}$, the expansion of $g_n$ in Lemma \ref{lem:watson-coordinate-md} applies, so 
\begin{align} \label{compact 1}
g_n(\alpha_i)
=
\frac{\sqrt{\rho_n L}}{\sqrt{\pi}} \cdot 
\exp\left\{
-L\left(\rho_n\alpha_i^2+\frac{\alpha_i^4}{4}\right)
\right\}
\cdot \lb 1+o(1)\rb,
\qquad i=1,2,
\end{align}
uniformly on $\mc E_{n,M}$.

Also, with $\Psi$ as in \eqref{Psi},
\begin{align} \label{compact 2}
    \wPi_n(\alpha_1,\alpha_2)
&=
\Psi_{p-1}\left(
\sqrt{p-1}\cdot
\frac{
t_n^{\circ\circ}-\wT_1\wT_2
}{
\sqrt{1-\wT_1^2}\sqrt{1-\wT_2^2}
}
\right) \nonumber \\
&= \frac{1+o(1)}
{(2-\alpha_1\alpha_2)\sqrt{2\pi L}}  
\cdot
\exp\left\{
-\frac L2(2-\alpha_1\alpha_2)^2
+
\frac{2-\alpha_1\alpha_2}{4}
\left(
\frac12\log L-x
\right)
\right\},
\end{align}
because
\begin{align*}
&\sqrt{p-1}\cdot
\frac{
t_n^{\circ\circ}-\wT_1\wT_2
}{
\sqrt{1-\wT_1^2}\sqrt{1-\wT_2^2}
}  =
\lb 2-\alpha_1\alpha_2 \rb \sqrt L
-
\frac{\frac12\log L-x}{4\sqrt L}
+
o\lb L^{-1/2} \rb
\end{align*}
uniformly on $\mc E_n$. Note that \eqref{compact 1} and \eqref{compact 2} hold uniformly over compact sets in $\mb R^2$ away from the curve $\la \lb \alpha_1, \alpha_2 \rb: \alpha_1 \alpha_2 =2 \ra$. 

Consequently,
\begin{align*}
   & \iint_{ \mc E_{n,M}}
\wPi_n(\alpha_1,\alpha_2)g_n(\alpha_1)g_n(\alpha_2)
\,d\alpha_1d\alpha_2 \\
= & \lb 1 + o(1) \rb \cdot \frac{\rho_n \sqrt{L}}{\pi \sqrt{2\pi }} \cdot \iint_{\mc E_{n,M}} L^{\frac{2-\alpha_1\alpha_2}{8}} \cdot \frac{\exp \la -LH_n \lb \alpha_1,\alpha_2 \rb - \frac{2-\alpha_1 \alpha_2}{4}x \ra}{2 - \alpha_1 \alpha_2 } \, d\alpha_1d\alpha_2
\end{align*}
where
\[
H_n(\alpha_1,\alpha_2)
:=
\rho_n(\alpha_1^2+\alpha_2^2)
+
\frac{\alpha_1^4+\alpha_2^4}{4}
+
\frac12(2-\alpha_1\alpha_2)^2 .
\]
Now make the change of variables
\[
y:= \frac{L^{1/4}}{\sqrt{2}} \cdot \lb \alpha_1 + \alpha_2 \rb; \qquad z:= \sqrt{\frac L2} \cdot \lb \alpha_1 - \alpha_2 \rb.
\]
The Jacobian is $dydz=L^{3/4} \, d\alpha_1 d\alpha_2$, so
\begin{equation*}
\begin{aligned}
&\iint_{\mc E_{n,M}}
\wPi_n(\alpha_1,\alpha_2)g_n(\alpha_1)g_n(\alpha_2)
\,d\alpha_1d\alpha_2 \\
= & 
\frac{e^{-x/2}}{n^2} \cdot 
\frac{1+o(1)}{2\pi\sqrt{2\pi}} \cdot 
\int_{-M/\sqrt{2}}^{M/\sqrt{2}}\int_{-M/\sqrt{2}}^{M/\sqrt{2}}
\exp\left\{
-\Delta y^2-2z^2-\frac{y^4}{4}
\right\}
\,dz\,dy,
\end{aligned}
\end{equation*}
because
\begin{align*}
LH_n(\alpha_1,\alpha_2)
&= 2L+
\sqrt L(\rho_n-1) y^2
+
(\rho_n+1)z^2
+
\frac14\left(y^2+\frac{z^2}{\sqrt L}\right)^2, \\
L^{\lb2 - \alpha_1 \alpha_2 \rb/8} &=L^{1/4} \lb 1 + o(1) \rb,\\
\frac{2-\alpha_1 \alpha_2}{4}x &= \frac{x}{2} \cdot \lb 1 + o(1) \rb, \qquad \sqrt{L} \lb \rho_n -1 \rb \to \Delta.
\end{align*}
This completes the proof of \eqref{mc E contribution}.

{\it  \underline{Proof of \eqref{mc E complement contribution}.}} Fix sufficiently small parameters $\delta_0, \ve_0$ to be specified later. We bound the integrand on $\mc E_{n,M}^c$ as
\begin{align*}
    &   \iint_{\mc E_{n,M}^c} \wPi_n(\alpha_1,\alpha_2)g_n(\alpha_1)g_n(\alpha_2) \,d\alpha_1d\alpha_2   \\
    \leq  &    \iint_{\mc E_{n,M, 1}^c} \wPi_n(\alpha_1,\alpha_2)g_n(\alpha_1)g_n(\alpha_2) \,d\alpha_1d\alpha_2 + \mbox{Int}_1 + \mbox{Int}_2 + \mbox{Int}_3 
\end{align*}
where
\begin{align*}
      \mc E^c_{n,M,1}:&= \mc E_{n,M}^c \cap \la \left|\alpha_1 \right| + |\alpha_2| \leq \ve_0 \ra \cap \la 2 - \alpha_1 \alpha_2 \geq \delta_0 \ra,   \\
     \mbox{Int}_1 &\leq \iintop_{ \mc E_{n,M}^c \cap \la 2- \alpha_1 \alpha_2 \leq \delta_0 \ra} \wPi_n(\alpha_1,\alpha_2)g_n(\alpha_1)g_n(\alpha_2) \,d\alpha_1d\alpha_2, \\
     \mbox{Int}_2 &\leq \iintop_{  \mc E_{n,M}^c \cap \la \ve_0^{-1} \geq  |\alpha_1| + |\alpha_2| \geq \ve_0 \ra \cap \la 2- \alpha_1 \alpha_2 \geq \delta_0 \ra} \wPi_n(\alpha_1,\alpha_2)g_n(\alpha_1)g_n(\alpha_2) \,d\alpha_1d\alpha_2, \\
     \mbox{Int}_3 &\leq \iintop_{ \mc E_{n,M}^c \cap \la |\alpha_1| + |\alpha_2| \geq \ve_0^{-1} \ra } \wPi_n(\alpha_1,\alpha_2)g_n(\alpha_1)g_n(\alpha_2) \,d\alpha_1d\alpha_2.
\end{align*}
We split the proof into two steps as follows.

{\it  \underline{Step 1: The contribution from $\mc E^c_{n,M} \setminus \mc E^c_{n,M,1}$ is of order $o \lb n^{-2} \rb$}.}  It suffices to show that 
\[
     \mbox{Int}_i = o \lb n^{-2} \rb, \qquad i=1,2,3.
\]
The above actually holds for all $M>0$. Consider the term $\mbox{Int}_1$. The trivial bound $\wPi_n \leq 1 $ and Lemma \ref{lem:watson-coordinate-md} yield
\begin{align*}
         \mbox{Int}_1 \lesssim \mbox{poly} \lb L \rb \cdot   \iintop_{ \la 2- \alpha_1 \alpha_2 \leq \delta_0 \ra} \exp \la - \frac{9L}{10} \left[ \rho_n \lb \alpha_1^2 + \alpha_2^2 \rb + \frac{\alpha_1^4 + \alpha_2^4}{4} \right]  \ra \, d\alpha_1 d\alpha_2.
\end{align*}
Note that, when $\delta_0$ is sufficiently small, the Cauchy--Schwarz inequality gives
\[
\inf_{\lb \alpha_1, \alpha_2 \rb: 2-\alpha_1\alpha_2 \leq \delta_0} \la \rho_n \lb \alpha_1^2 + \alpha_2^2 \rb + \frac{\alpha_1^4 + \alpha_2^4}{4}  \ra \geq \inf_{\lb \alpha_1, \alpha_2 \rb: 2-\alpha_1\alpha_2 \leq \delta_0} \la 2\rho_n \cdot \alpha_1 \alpha_2  + \frac{\alpha_1^2\alpha_2^2}{2}  \ra > 3.
\]
for all sufficiently large $n$. 
We then obtain $\mbox{Int}_1 = o \lb n^{-2} \rb$ using Lemma \ref{lem:strong-convex}.

Consider the term $\mbox{Int}_3$. Again,  the trivial bound $\wPi_n \leq 1 $ and Lemma \ref{lem:watson-coordinate-md} yield
\begin{align*}
         \mbox{Int}_3 \lesssim \mbox{poly} \lb L \rb \cdot   \iintop_{\mc E_{n,M}^c \cap \la |\alpha_1| + |\alpha_2| \geq \ve_0^{-1} \ra } \exp \la - \frac{L}{2} \left[ \rho_n \lb \alpha_1^2 + \alpha_2^2 \rb + \frac{\alpha_1^4 + \alpha_2^4}{4} \right]  \ra \, d\alpha_1 d\alpha_2.
\end{align*}
Observe that on the set $\la |\alpha_1| + |\alpha_2| \geq \ve_0^{-1} \ra$, one can make 
\[
 \rho_n \lb \alpha_1^2 + \alpha_2^2 \rb + \frac{\alpha_1^4 + \alpha_2^4}{4} >20
\]
by choosing $\ve_0$ sufficiently small. We then obtain $\mbox{Int}_3 = o \lb n^{-2} \rb$ using Lemma \ref{lem:strong-convex}.

Now consider the term $\mbox{Int}_2$. Since the domain of integration is compact in this case, it follows from \eqref{compact 1}, \eqref{compact 2}, and Lemma \ref{lem:watson-coordinate-md} that 
 \begin{align*}
     \mbox{Int}_2 &\lesssim \frac{\rho_n \sqrt{L}}{\pi \sqrt{2\pi }} \times \iintop_{\mc E_{n,M}^c \cap \la \ve_0^{-1} \geq  |\alpha_1| + |\alpha_2| \geq \ve_0 \ra \cap \la 2- \alpha_1 \alpha_2 \geq \delta_0 \ra} L^{\frac{2-\alpha_1\alpha_2}{8}} \cdot \frac{\exp \la -LH_n \lb \alpha_1,\alpha_2 \rb - \frac{2-\alpha_1 \alpha_2}{4}x \ra}{2 - \alpha_1 \alpha_2 } \, d\alpha_1d\alpha_2  \\
     &\lesssim_{x,\ve_0} \, \mbox{poly}(L) \times  \iintop_{\mc E_{n,M}^c \cap \la \ve_0^{-1} \geq  |\alpha_1| + |\alpha_2| \geq \ve_0 \ra \cap \la 2- \alpha_1 \alpha_2 \geq \delta_0 \ra}  \exp \la -L H_n \lb \alpha_1, \alpha_2 \rb \ra \, d\alpha_1d\alpha_2.
 \end{align*}
It is easy to check that, for all sufficiently large $n$,
\begin{align*}
&\inf_{ \lb \alpha_1, \alpha_2 \rb \in \la \ve_0^{-1} \geq  |\alpha_1| + |\alpha_2| \geq \ve_0 \ra \cap \la 2- \alpha_1 \alpha_2 \geq \delta_0 \ra} H_n \lb \alpha_1, \alpha_2 \rb \\
=&\inf_{ \lb \alpha_1, \alpha_2 \rb \in \la \ve_0^{-1} \geq  |\alpha_1| + |\alpha_2| \geq \ve_0 \ra \cap \la 2- \alpha_1 \alpha_2 \geq \delta_0 \ra} \left[ \rho_n(\alpha_1^2+\alpha_2^2)
+
\frac{\alpha_1^4+\alpha_2^4}{4}
+
\frac12(2-\alpha_1\alpha_2)^2  \right]  \geq 2+ \frac{\ve_0^4}{10^3}.
\end{align*}
We then get $\mbox{Int}_2 = o \lb n^{-2} \rb$ since the domain of integration is compact.

{\it  \underline{Step 2: The contribution from $ \mc E^c_{n,M,1}$ is of order $o \lb n^{-2} \rb$}.}
We now bound the contribution from \(\mc E^c_{n,M,1}\).  Recall that
\[
\mc E^c_{n,M,1}
=
\mc E_{n,M}^c
\cap
\left\{
|\alpha_1|+|\alpha_2|\le \ve_0
\right\}
\cap
\left\{
2-\alpha_1\alpha_2\ge \delta_0
\right\}.
\]
Since our domain of interest has compact closure and is away from the curve $\la \lb \alpha_1, \alpha_2 \rb: \alpha_1 \alpha_2=2 \ra$, the bound in Lemma \ref{lem:watson-coordinate-md} applies and yields
\begin{align*}
    \wPi_n(\alpha_1,\alpha_2) g_n(\alpha_1)g_n(\alpha_2) 
& \lesssim_{x,\delta_0,\ve_0} \,
L^{3/4 - \alpha_1 \alpha_2/8} \cdot \exp\{-L H_n(\alpha_1,\alpha_2)\} \\
&= L^{3/4} \cdot \exp \la - \frac{\alpha_1 \alpha_2 \cdot \log L}{8} -L H_n \lb \alpha_1, \alpha_2 \rb  \ra
\end{align*}
uniformly on $\mc E^c_{n,M,1}$.

Now perform the change of variables
\[
y:=\frac{L^{1/4}}{\sqrt2} \lb \alpha_1+\alpha_2 \rb,
\qquad
z:=\sqrt{\frac L2} \lb \alpha_1-\alpha_2\rb.
\]
With this change of variables,
\[
d\alpha_1d\alpha_2=L^{-3/4}\,dy\,dz, \qquad \mc E^c_{n,M} \subset  \la  |y|>\frac{M}{\sqrt2} \ra \cup \la |z|>\frac{M}{\sqrt2} \ra.
\]
Moreover,
$
    \alpha_1 \alpha_2 = \frac{y^2}{2\sqrt{L}} - \frac{z^2}{2L}, 
$
and
\begin{align*}
    L \cdot H_n \lb \alpha_1, \alpha_2 \rb &= 2L+ \sqrt L(\rho_n-1)y^2
+
(\rho_n+1)z^2
+
\frac14\left(y^2+\frac{z^2}{\sqrt L}\right)^2 .
\end{align*}
Consequently,
\begin{align*}
        & \iint_{\mc E_{n,M, 1}^c} \wPi_n(\alpha_1,\alpha_2)g_n(\alpha_1)g_n(\alpha_2) \,d\alpha_1d\alpha_2 \\
     \lesssim_{x,\delta_0,\ve_0}
       & \iintop_{ \la  |y|>\frac{M}{\sqrt2} \ra \cup \la |z|>\frac{M}{\sqrt2} \ra} 
       \exp \la  - \frac{\log L}{16\sqrt{L}} \cdot y^2 + \frac{\log L}{16L} \cdot z^2 -2L - \sqrt L(\rho_n-1)y^2  \ra \\
&\times \exp \la - (\rho_n+1)z^2
-
\frac14\left(y^2+\frac{z^2}{\sqrt L}\right)^2  \ra \, dy \, dz \\
\lesssim_{x,\delta_0,\ve_0} & \, n^{-2} \cdot \iintop_{ \la  |y|>\frac{M}{\sqrt2} \ra \cup \la |z|>\frac{M}{\sqrt2} \ra} 
\exp \la  Cy^2 - cz^2 - \frac{y^4}{8} \ra \, dy \, dz.
\end{align*}
for some universal constants $C,c>0$. 

The last double integral goes to $0$ as $M \to \infty$. This yields \eqref{mc E complement contribution} and completes the proof. $\hfill$ $\square$

\subsection{Poisson approximation}

We now prove Theorem \ref{thm:Watson-critical}.  Without loss of
generality, assume \(\bm\mu_n=\bm e_1\).  Fix \(x\in\mb R\), and let
\(s_n^{\circ\circ}\) and \(t_n^{\circ\circ}\) be as in Proposition
\ref{prop:large-deviation-critical}.  
Define
\[
L:= \log n, \qquad 
    I_{ij}^{\circ\circ}
    :=
    \mathbf 1_{\{|S_{ij}|>t_n^{\circ\circ}\}},
    \qquad
    W^{\circ\circ}
    :=
    \sum_{1\le i<j\le n}I_{ij}^{\circ\circ}.
\]
Then
\[
\left\{
pM_n^2\le 4\log n-\frac12\log\log n+x
\right\}
=
\{M_n\le t_n^{\circ\circ}\}
=
\{W^{\circ\circ}=0\}.
\]

Let
$
    \mc I_n:=\{(i,j):1\le i<j\le n\}.
$
By Lemma
\ref{lem:AGG},
\[
\left|
\mb P_{\kappa_n,\bm e_1}(W^{\circ\circ}=0)
-
\exp\{-\lambda_n^{\circ\circ}\}
\right|
\le
b_{1n}^{\circ\circ}+b_{2n}^{\circ\circ},
\]
where
\[
    \lambda_n^{\circ\circ}
    :=
    \mb E_{\kappa_n,\bm e_1}W^{\circ\circ}
    =
    \binom n2
    \mb P_{\kappa_n,\bm e_1}(I_{12}^{\circ\circ}=1),
\]
and
\[
    b_{1n}^{\circ\circ}
    \lesssim
    n^3\mb P_{\kappa_n,\bm e_1}(I_{12}^{\circ\circ}=1)^2,
    \qquad
    b_{2n}^{\circ\circ}
    \lesssim
    n^3
    \mb P_{\kappa_n,\bm e_1}
    \left(I_{12}^{\circ\circ}=1,I_{13}^{\circ\circ}=1\right).
\]
As in the proof of Theorem \ref{thm:watson-rho<1}, Proposition \ref{prop:large-deviation-critical} gives the mean convergence and $ b_{1n}^{\circ\circ} = O \lb n^{-1} \rb = o(1)$.

It remains to prove that \(b_{2n}^{\circ\circ}\to0\).  Put
\[
 I_{ij}^{\circ\circ}
    =
    I_{ij}^{+,\circ\circ}+I_{ij}^{-,\circ\circ}, \qquad 
    I_{ij}^{+,\circ\circ}:=\mathbf 1_{\{S_{ij}>t_n^{\circ\circ}\}},
    \qquad
    I_{ij}^{-,\circ\circ}:=\mathbf 1_{\{S_{ij}<-t_n^{\circ\circ}\}}.
\]
By the symmetry of the Watson distribution and independent sign changes of
\(\bm X_2\) and \(\bm X_3\),
\[
\begin{aligned}
&\mb P_{\kappa_n,\bm e_1}
\left(
I_{12}^{\circ\circ}=1,I_{13}^{\circ\circ}=1
\right) =
4\,
\mb P_{\kappa_n,\bm e_1}
\left(
I_{12}^{+,\circ\circ}=1,I_{13}^{+,\circ\circ}=1
\right).
\end{aligned}
\]
Thus it suffices to show that
\begin{equation}\label{eq:critical-right-shared-square}
\mb P_{\kappa_n,\bm e_1}
\left(
I_{12}^{+,\circ\circ}=1,I_{13}^{+,\circ\circ}=1
\right)
=
o(n^{-3}).
\end{equation}

Now, for fixed
$\bm x=\left(x_1,\sqrt{1-x_1^2} \cdot \bm v_x\right)\in\mb S^{p-1}$,
write
\[
    x_1= \widetilde{x}_1 \cdot \frac{L^{1/4}}{p^{1/4}}.
\]
Define
\[
\eta_n^+(\bm x)
:=
\mb P_{\kappa_n,\bm e_1}
\left(
\bm x^\top\bm X>t_n^{\circ\circ}
\right),
\qquad 
\bm X_1=\left(\wT_1,\sqrt{1-\wT^2_1} \cdot \bm V_1 \right)
\]
 as in \eqref{representation2}. Conditioning on \(\bm X_1\), we have
\[
\mb P_{\kappa_n,\bm e_1}
\left(
I_{12}^{+,\circ\circ}=1,I_{13}^{+,\circ\circ}=1
\right)
=
\mb E_{\kappa_n,\bm e_1}
\left[
\eta_n^+(\bm X_1)^2
\right].
\]
where
\[
\eta_n^+(\bm x)
=
\int_{\mb R}
\wPi_n( \, \widetilde{x}_1,\alpha)g_n(\alpha)\,d\alpha,
\]
where \(g_n\) is the density in Lemma \ref{lem:watson-coordinate-md} and $\wPi_n$ is as in \eqref{tilde Pi}.

  {\it \underline{Step 1: Pointwise bound on $\eta^+_n \lb \bm x \rb$.}} We first note that for every sufficiently small fixed
\(\ve>0\),
\begin{equation}\label{eq:critical-wPi-pointwise-bound}
\wPi_n(a,b)
\le
\operatorname{poly}(L) \cdot 
\exp\left\{
-\frac L2\left(2-|ab|-\ve\right)_+^2
\right\},
\end{equation}
uniformly in \(a,b\).  

Indeed, if \(2-|ab|-\ve\le0\), the bound is trivial.
If \(2-|ab|-\ve>0\), then a crude tail bound using the self-normalized moderate-deviation result (Theorem 2.3 in \cite{jing2003self}) applies. 
Therefore \eqref{eq:critical-wPi-pointwise-bound} holds. 
Put
\[
    \phi_n(\alpha):=\rho_n\alpha^2+\frac{\alpha^4}{4}.
\]
By Lemma \ref{lem:watson-coordinate-md}, for every sufficiently small fixed
\(\ve>0\),
\[
g_n(\alpha)
\le
\operatorname{poly}(L)
\exp\{-L(1-\ve)\phi_n(\alpha)\}.
\]
Combining this with \eqref{eq:critical-wPi-pointwise-bound}, we obtain,
uniformly in \(\bm x\),
\[
\eta_n^+(\bm x)
\le
\operatorname{poly}(L)
\int_{\mb R}
\exp\left\{
-L\left[
(1-\ve)\phi_n(\alpha)
+
\frac12 \lb 2-|\widetilde{x}_1\alpha|-\ve\rb_+^2
\right]
\right\}
\,d\alpha .
\]
We next use the fact that the term inside the exponent is strongly convex as a function of $\alpha$ to bound the integral above. To this end, for \(u\ge0\), put
\[
F_{n,u}(v):= (1-\ve)\phi_n(v)
+
\frac12(2-uv-\ve)_+^2 \, , \qquad
J_n(u)
:=
\inf_{v\ge0}
\left\{
F_{n,u}(v)
\right\}.
\]
As a function of $v$, $F_{n,u}(v)$ is $1/2$-strongly convex because it is the sum of a $1/2$-strongly convex function (for large $n$) and a convex function. 
Thus, for $ \widetilde{x}_1 \geq 0$, applying Lemma \ref{lem:strong-convex} with $\Omega=\mb R^+$, $G=F_{n,u}(v)$, and $\ve=1/L$ yields
\begin{align*}
    \eta_n^+(\bm x)  \leq \mbox{poly}(L) \cdot \int_{0}^\infty \exp \la -L\cdot F_{n, \widetilde{x}_1 } \lb \alpha \rb \ra \, d\alpha
    &\lesssim  \mbox{poly}(L) \cdot \exp \la -L \cdot J_n \lb \, \widetilde{x}_1 \rb \ra.
\end{align*}
Applying the same argument when $\widetilde{x}_1 < 0$ yields the unified bound
\begin{equation}\label{eq:critical-eta-pointwise}
\eta_n^+(\bm x)
\le
\operatorname{poly}(L)
\exp\left\{-LJ_n \lb |  \widetilde{x}_1 | \rb\right\}.
\end{equation}

  {\it \underline{Step 2: Bounding $\mb E_{\kappa_n,\bm e_1}
\left[ \eta_n^+(\bm X_1)^2 \right]$}.}
We integrate \eqref{eq:critical-eta-pointwise} to obtain
\[
\begin{aligned}
\mb E_{\kappa_n,\bm e_1}
\left[
\eta_n^+(\bm X_1)^2
\right] 
&\le  \mbox{poly}(L) \cdot \int_{\mb R}  \exp\left\{-2L \cdot J_n \lb | \alpha | \rb \right\} \cdot g_n(\alpha) \, d\alpha     \\
&\lesssim \operatorname{poly}(L) \cdot
\int_{0}^\infty
\exp\left\{
-L\left[
(1-\ve)\phi_n(\alpha)+2J_n(|\alpha|)
\right]
\right\}
\,d\alpha .
\end{aligned}
\]
where $g_n$ is the density in Lemma \ref{lem:watson-coordinate-md}, and $\ve$ is sufficiently small and will be specified later.

Observe that, for $u,v\geq0$, $\ve>0$ sufficiently small, and
$n$ sufficiently large, we have $\rho_n\geq1-\ve$, and hence
\begin{align*}
&\frac12(1-2\ve)\phi_n(u)
+
(1-\ve)\phi_n(v)
+
\frac12\left(2-uv-\ve\right)_+^2\\
&\qquad\geq
(1-3\ve)
\left[
\frac{u^2}{2}+v^2+\frac{u^4}{8}+\frac{v^4}{4}
\right]
+
\frac12\left(2-uv-\ve\right)_+^2\\
&\qquad\geq
(1-3\ve)
\left[
\sqrt{2}\,uv+\frac{u^2v^2}{2\sqrt{2}}
\right]
+
\frac12\left(2-uv-\frac{3\ve}{2}\right)_+^2\\
&\qquad\geq
\frac{3+c}{2},
\end{align*}
where the second inequality follows from the Cauchy--Schwarz inequality, and the last inequality follows from
Lemma~\ref{lem:technical-minimum}, with $3\ve/2$ in place of $\ve$.
Taking the infimum over $v\geq0$ gives
\[
\frac12(1-2\ve)\phi_n(u)+J_n(u)
\geq
\frac{3+c}{2}.
\]
Consequently,
\begin{align*}
(1-\ve)\phi_n(u)+2J_n(u)
&=
2\left[
\frac12(1-2\ve)\phi_n(u)+J_n(u)
\right]
+\ve\phi_n(u)\\
&\geq
3+c+\ve\phi_n(u).
\end{align*}
It follows that
\begin{align*}
\mb E_{\kappa_n,\bm e_1}
\left[
\eta_n^+(\bm X_1)^2
\right]
&\leq
\operatorname{poly}(L)e^{-(3+c)L}
\int_0^\infty
\exp\left\{-\ve L\phi_n(u)\right\}\,du\\
&\leq
\operatorname{poly}(L)e^{-(3+c)L}
\int_0^\infty
\exp\left\{-\frac{\ve L}{4}u^4\right\}\,du\\
&=
\operatorname{poly}(L)L^{-1/4}e^{-(3+c)L}
=
o(n^{-3}).
\end{align*}
Thus
\[
\mb P_{\kappa_n,\bm e_1}
\left(
I_{12}^{+,\circ\circ}=1,
I_{13}^{+,\circ\circ}=1
\right)
=
o(n^{-3}),
\]
which proves \eqref{eq:critical-right-shared-square}. This completes the proof. $\hfill$ $\square$

\begin{appendix}
 \section{Technical results} 

We start with a useful tail bound. 
 \begin{lemma}\label{lem:spherical-tail}
Suppose \(d\sim p\) and \(p/(\log n)^2\to\infty\). The following statements hold:

\begin{itemize}
    \item With \(u_n=u_n(x)\) as in
\eqref{u-t} (for fixed $x$),
\begin{equation}\label{eq:Psi-main}
 \Psi_d(u_n)
 =\frac{1+o(1)}{u_n \sqrt{2\pi}}e^{-u_n^2/2}
 =\frac{1+o(1)}{2\sqrt{2\pi}}n^{-2}e^{-x/2}.
\end{equation}

\item For each fixed \(M<\infty\), uniformly over \(|a|\le M/u_n\),
\begin{equation}\label{eq:Psi-local-ratio}
 \frac{\Psi_d(u_n-a)}{\Psi_d(u_n)}
 =\exp\left\{u_na-\frac{a^2}{2}\right\}(1+o(1)).
\end{equation}

\item Suppose \(d=p-1\). For all \(A\in[0,1]\) and all \(B\in[-1,1]\), we have
\begin{align}
 \frac{\mb{P}(AZ_d+B>t_n)}{\Psi_d(u_n)} &\le C\exp(u_n\sqrt p\,B), \label{eq:right-domination} \\
 \frac{\mb P(AZ_d+B<-t_n)}{\Psi_d(u_n)}&\le C\exp(-u_n\sqrt p\,B),
 \label{eq:left-domination}
\end{align}
for some constant $C=C_{x}>0$ that depends only on $x$, where $t_n$ is defined in \eqref{u-t}.
\end{itemize} 

\end{lemma}

\noindent \textbf{Proof of Lemma \ref{lem:spherical-tail}.} 
Note that \eqref{eq:Psi-main} follows directly from the standard tail asymptotics of the beta distribution. Given \eqref{eq:Psi-main}, \eqref{eq:Psi-local-ratio} follows by applying \eqref{eq:Psi-main} to $\Psi_d(u_n-a)$ and $\Psi(u_n)$. It remains to prove \eqref{eq:right-domination} and \eqref{eq:left-domination}.  

By the symmetry of \(Z_d\), it
suffices to prove \eqref{eq:right-domination} because replacing $B$ by $-B$ in \eqref{eq:right-domination} gives \eqref{eq:left-domination}. We now prove \eqref{eq:right-domination}. 
We first note that if \(t_n-B\le0\), then \eqref{eq:Psi-main} gives
\[
\frac{\mb{P}(AZ_d+B>t_n)}{\Psi_d(u_n)} \lesssim u_n \cdot e^{u_n^2/2} \leq e^{u_n^2} = \exp \lb u_n \sqrt{p} t_n \rb \leq \exp \lb u_n \sqrt{p} B \rb.
\]
Thus \eqref{eq:right-domination} holds trivially. Therefore, we only need to consider the case $t_n-B > 0$. Since \(A\le1\), we have the bound
\[
 \mb P(AZ_d+B>t_n)\le \mb P\lb Z_d>t_n-B \rb=\Psi_d(y_B),
 \qquad
 y_B:=\sqrt d\,\lb t_n-B \rb >0.
\]
We consider two cases: $y_B \in [u_n/2,\infty)$ and $y_B \in (0,u_n/2)$. 

{\it \underline{Case 1: \(y_B\ge u_n/2\)}.} We have the Mills ratio bound \(\Psi_d(y_B)\lesssim y_B^{-1}e^{-y_B^2/2}\) in this regime.  Dividing by
\eqref{eq:Psi-main} and using \(y_B\ge u_n/2\),
\[
\frac{\mb{P}(AZ_d+B>t_n)}{\Psi_d(u_n)}  \leq  \frac{\Psi_d(y_B)}{\Psi_d(u_n)}
 \lesssim \exp\left\{\frac{u_n^2-y_B^2}{2}\right\}.
\]
Since \(d=p-1\) and \(t=u_n/\sqrt p\),
\[
 \frac{u_n^2-y_B^2}{2}
 =\frac{u_n^2}{2}\left(1-\frac{p-1}{p}\right)
   +\frac{p-1}{p} \, u_n\sqrt p\,B-\frac{p-1}{2}B^2
 \le u_n\sqrt p\,B+o(1),
\]
uniformly in \(B\in[-1,1]\).  Hence \eqref{eq:right-domination} holds in this case.

{\it \underline{Case 2: \(0<y_B<u_n/2\)}}.  Note that
\(B=t_n-y_B/\sqrt d\), and the same Mills ratio bound gives
\[
 \frac{\Psi_d(y_B)}{\Psi_d(u_n)}
 \lesssim \frac{u_n}{1+y_B}
 \exp\left\{\frac{u_n^2-y_B^2}{2}\right\} \leq u_n \cdot \exp\left\{\frac{u_n^2-y_B^2}{2}\right\}.
\]
On the other hand,
\[
 u_n\sqrt p\,B
 =u_n^2-u_n y_B\sqrt{p/d}
 =u_n^2-u_n y_B+o(1).
\]
Thus, uniformly for \(0<y_B<u_n/2\),
\[
 u_n\sqrt p\,B- \frac{u_n^2-y_B^2}{2}
 =\frac{(u_n-y_B)^2}{2}+o(1)
 \ge \frac{u_n^2}{8}+o(1).
\]
Consequently,
\[
u_n \cdot \exp \la \frac{u_n^2-y_B^2}{2} \ra \leq (1+o(1))\cdot u_n \cdot \exp \la  u_n\sqrt p\,B - \frac{u_n^2}{8} \ra
\lesssim \exp \lb u_n\sqrt{p} B \rb 
\]
  Hence \eqref{eq:right-domination} also holds in this case. This completes the proof. $\hfill$ $\square$

\begin{lemma}\label{lem:spherical-mgf}
Let \(R_m\) denote the first coordinate of a uniform point on
\(\mathbb S^{m-1}\), and put \(M_m(a):=\mb E e^{aR_m}\).  Uniformly for
\(|a|=o(m)\),
\begin{equation}\label{eq:spherical-mgf-expansion}
 \log M_m(a)=\frac{a^2}{2m}
 +O\left(\frac{a^4}{m^3}\right).
\end{equation}
\end{lemma}



\noindent \textbf{Proof of Lemma \ref{lem:spherical-mgf}.}
We use the exact moments of a uniform spherical coordinate. By equation (9.6.3) in \cite{M-Jupp}, we have
\[
\mathbb E U_1^{2k}
=
\frac{(1/2)_k}{(m/2)_k}
=
\frac{(2k-1)!!}{m(m+2)\cdots(m+2k-2)}
\]
with the convention that the empty product equals \(1\). 

 The monotone convergence theorem then yields
\[
\begin{aligned}
M_m(a)
&=
\sum_{k=0}^{\infty}
\frac{a^{2k}}{(2k)!} \cdot \mathbb E U_1^{2k} =
\sum_{k=0}^{\infty}
\frac{a^{2k}}{(2k)!} \cdot
\frac{(2k-1)!!}{m(m+2)\cdots(m+2k-2)}.
\end{aligned}
\]
Since
\[
(2k-1)!!=\frac{(2k)!}{2^k k!},
\]
we get
\[
M_m(a)
=
\sum_{k=0}^{\infty}
\frac{1}{k!}
\left(\frac{a^2}{2m}\right)^k \cdot 
\prod_{\ell=0}^{k-1}\left(1+\frac{2\ell}{m}\right)^{-1}.
\]
Put
\[
\theta:=\frac{a^2}{2m},
\qquad
w_{k,m}:=
\prod_{\ell=0}^{k-1}\left(1+\frac{2\ell}{m}\right)^{-1}.
\]
Then
$
M_m(a)=\sum_{k=0}^{\infty}\frac{\theta^k}{k!} \cdot w_{k,m}
$.
Let \(N\sim\mathrm{Poisson}(\theta)\) and note that, since
$\mathbb P(N=k)=e^{-\theta} \cdot \theta^k/k!$, we may rewrite the preceding display as
\[
M_m(a)=e^\theta \cdot  \mathbb E w_{N,m}.
\]
Now define
\[
w_{k,m}=e^{-S_{k,m}},
\qquad
S_{k,m}:=
\sum_{\ell=0}^{k-1}\log\left(1+\frac{2\ell}{m}\right).
\]
Using the bounds \(\log(1+x)\ge0\) and \(\log(1+x)\le x\) for \(x\ge0\), we have 
\[
0\le S_{k,m}
\le
\sum_{\ell=0}^{k-1}\frac{2\ell}{m}
=
\frac{k(k-1)}{m}.
\]
Consequently,
\[
0\le S_{N,m}\le \frac{N(N-1)}{m}.
\]
On the other hand, Jensen's inequality gives
\[
\log \mathbb E e^{-S_{N,m}}
\ge
-\mathbb E S_{N,m}.
\]
Combining the two bounds above with \(\mathbb E[N(N-1)]=\theta^2\), we obtain
\[
\frac{-\theta^2}{m}
=
\frac{-\mb E \left[ N(N-1) \right]}{m}
\leq
-\mb E S_{N,m} 
\leq
\log \mathbb E e^{-S_{N,m}}
\le
0.
\]
Consequently,
\[
\log M_m(a)
=
\theta
+
O\left(\frac{\theta^2}{m}\right)
=
\frac{a^2}{2m}
+
O\left(\frac{a^4}{m^3}\right).
\]
This completes the proof. $\hfill$ $\square$

The following Poisson approximation result will be used throughout the paper. It gives a total variation bound between a sum of Bernoulli random variables and a Poisson distribution with the same mean. The bound is expressed in terms of the dependency graph of the Bernoulli variables. Readers are referred to the paper \cite{arratia1989two} for a proof. 
\begin{lemma}[\cite{arratia1989two}]\label{lem:AGG}
Let \(\{I_\alpha\}_{\alpha\in\mathcal I}\) be Bernoulli variables with a
dependency graph: \(I_\alpha\) is independent of
\(\{I_\beta:\beta\notin B_\alpha\}\).  Let
\(W=\sum_{\alpha\in\mathcal I}I_\alpha\) and \(\lambda=\mb E W\).  Define
\[
 b_1:=\sum_{\alpha\in\mathcal I}\sum_{\beta\in B_\alpha}
     \mb E I_\alpha\,\mb E I_\beta,
 \qquad
 b_2:=\sum_{\alpha\in\mathcal I}\sum_{\substack{\beta\in B_\alpha\\\beta\ne\alpha}}
     \mb E(I_\alpha I_\beta).
\]
Then
\[
 d_{\mathrm{TV}}\lb \mathcal L(W),\operatorname{Poisson}(\lambda)\rb
 \le \min(1,\lambda^{-1})(b_1+b_2).
\]
\end{lemma}

\begin{lemma} \label{lem:technical-minimum}
    There exists a universal constant $c>0$ such that 
    \[
    \inf_{\ve \in \left[ 0, \frac{1}{10} \right], q \geq 0} \la \lb 1 - 2\ve \rb \left[ \sqrt{2}q + \frac{q^2}{2\sqrt{2}} \right] + \frac 12  \lb 2-q-\ve \rb_+^2 \ra > \frac{3+c}{2}.
    \]
\end{lemma}

\noindent \textbf{Proof of Lemma \ref{lem:technical-minimum}.}
Fix $\ve \in [0,1/10]$. Consider the case $q \geq 2-\ve$. Then
\begin{align*}
    \inf_{q \geq 0} \la \lb 1 - 2\ve \rb \left[ \sqrt{2}q + \frac{q^2}{2\sqrt{2}} \right] + \frac 12  \lb 2-q-\ve \rb_+^2 \ra
    &= \inf_{q \geq 2-\ve} \la  \lb 1 - 2\ve \rb \left[ \sqrt{2}q + \frac{q^2}{2\sqrt{2}} \right]  \ra \\
    &\geq \lb 1 - 2\ve \rb \cdot \left[ \sqrt{2}q + \frac{q^2}{2\sqrt{2}} \right] \Bigg|_{q=2-\ve} > 3/2
\end{align*}
for all $\ve<=1/10$.

If $q \leq 2-\ve$, then
\begin{align*}
    & \inf_{q \geq 0} \la \lb 1 - 2\ve \rb \left[ \sqrt{2}q + \frac{q^2}{2\sqrt{2}} \right] + \frac 12  \lb 2-q-\ve \rb_+^2 \ra \\
    = & \inf_{0 \leq q \leq 2-\ve} \la  \lb 1 - 2\ve \rb \left[ \sqrt{2}q + \frac{q^2}{2\sqrt{2}} \right] + \frac 12  \lb 2-q-\ve \rb^2  \ra 
    > 3/2
\end{align*}
for all $\ve<=1/10$, where we omit a straightforward minimization of the quadratic function. $\hfill$ $\square
$

Recall that a function $F: \mb{R}^d \to \mb{R}$ is $\mu$-strongly convex if $F(\bm x) - \frac{\mu}{2} \cdot \| \bm x \|^2$ is a convex function. The following elementary bound is used repeatedly in the proofs. 
\begin{lemma} \label{lem:strong-convex}
    Suppose $G: \Omega \subset \mb{R}^d \to \mb{R}$ is a function such that there exists a nonnegative, $\mu$-strongly convex function $F: \mb R^d \to \mb R$ satisfying the following condition:
    \[
     G \lb \bm x \rb \geq \max \la  a, F(\bm x)  \ra, \qquad \text{for all $\bm x \in \Omega$}.
    \]
    Assume that this condition holds for some constant $a>0$. Then, for any $\lambda>0$ and $\ve \in (0,1)$,
    \[
    \int_{\Omega} \exp \left[ - \lambda \cdot G \lb \bm x \rb \right] \, d\bm x \leq \exp \left[ -\lambda(1-\ve)a \right] 
    \cdot \lb \frac{2\pi}{\lambda \ve \mu} \rb^{d/2}   .
    \]
\end{lemma}

\noindent \textbf{Proof of Lemma \ref{lem:strong-convex}.}
Let $\bm x_{\rm min}$ be the unique global minimizer of $F$. By strong convexity and the nonnegativity of $F$, for all $\bm x \in \mb R^d$, we have 
$
F \lb \bm x \rb \geq \lb \mu /2 \rb \cdot \| \bm x - \bm x_{\rm min}  \|^2.
$

Note that $G(\bm x) \geq (1-\ve)a+\ve F(\bm x)$. Therefore, 
\[
    \int_{\Omega} \exp \left[ - \lambda \cdot G \lb \bm x \rb \right] \, d\bm x \leq   e^{-\lambda(1-\ve) a} \cdot \int_{\mb R^d} \exp \la  \frac{-\lambda \ve \mu}{2} \cdot \| \bm x - \bm x_{\rm min}  \|^2 \ra \, d\bm x = \lb \frac{2\pi}{\lambda \ve \mu} \rb^{d/2}  e^{-\lambda (1-\ve) a}. \qquad 
\]
$\hfill$ $\square$

\section{Proof of (\ref{power near 1})} \label{appendix: power near 1}
It suffices to show that $\mb P \lb W_n >0 \rb \to 1$ for a fixed $x \in \mb R$, where $W_n$ is as in \eqref{I-W}. Here we assume the regime 
\[
\rho_n = \frac{p-2\kappa_n}{2\sqrt{p L}} \to 1. 
\]
By the Paley--Zygmund inequality,
\[
\mb P \lb W_n>0 \rb \geq \frac{\lb \mb E W_n \rb^2}{\mb E W_n^2}.
\]
Now, with the indicators $I_{ij}$ defined in \eqref{I-W},
\[
\mb E W_n =  \binom{n}{2} \mb P \lb I_{12}=1 \rb
\]
and 
\[
\mb E W_n^2 \lesssim  \mb E W_n \lb  \mb E W_n +1  \rb + n^3  \mb P \lb I_{12}=1 , I_{13}=1\rb.
\]
Arguing as in the proof of \eqref{eq:critical-right-tail-goal}, we obtain $\mb E W_n \to \infty$ when $\rho_n \to 1$. Moreover, Step 2 of the proof of Theorem \ref{thm:Watson-critical} still yields $n^3  \mb P \lb I_{12}=1 , I_{13}=1\rb \to 0$. This completes the proof. $\hfill$ $\square$

\end{appendix}

\bibliographystyle{plainnat}
\bibliography{paper-ref}

\end{document}